\batchmode
\documentclass[english]{article}
\usepackage{lmodern}
\usepackage[T1]{fontenc}
\usepackage[latin9]{inputenc}
\usepackage{geometry}
\usepackage{amssymb}
\usepackage{setspace}
\makeatletter
\@ifundefined{date}{}{\date{}}
\usepackage[OT1]{fontenc}
\usepackage{amsmath,amsthm,amssymb,amsfonts}
\usepackage{mathcomp,mathrsfs}

\usepackage[all]{xy}
\usepackage{tikz-cd,arraycols,youngtab}
\usepackage{hyperref}
\usepackage{doi}

\theoremstyle{definition}

\newtheorem{df}{Definition}[subsection]

\theoremstyle{plain}

\newtheorem{thm}[df]{Theorem}
\newtheorem*{thm*}{Theorem}
\newtheorem{cj}[df]{Conjecture}
\newtheorem*{cj*}{Conjecture}
\newtheorem{prop}[df]{Proposition}
\newtheorem*{prop*}{Proposition}
\newtheorem{lem}[df]{Lemma}
\newtheorem*{lem*}{Lemma}
\newtheorem{cor}[df]{Corollary}

\theoremstyle{remark}

\newtheorem{exmp}[df]{Example}
\newtheorem{rmk}[df]{Remark}

\DeclareMathOperator{\Ker}{Ker}

\DeclareMathOperator{\End}{End}

\DeclareMathOperator{\Aut}{Aut}
\DeclareMathOperator{\Id}{Id}
\DeclareMathOperator{\Hom}{Hom}
\DeclareMathOperator{\Ext}{Ext}

\DeclareMathOperator{\Tr}{Tr}

\DeclareMathOperator{\GL}{GL}

\DeclareMathOperator{\supp}{supp}
\DeclareMathOperator{\rk}{rk}
\DeclareMathOperator{\ord}{ord}

\DeclareMathOperator{\Rep}{Rep}
\DeclareMathOperator{\vol}{vol}

\DeclareMathOperator{\Exp}{Exp}

\DeclareMathOperator{\Hilb}{Hilb}

\DeclareMathOperator{\Spec}{Spec}

\newcommand\dd{\mathbf{d}}
\newcommand\ee{\mathbf{e}}
\newcommand\mm{\mathbf{m}}
\newcommand\rr{\mathbf{r}}
\newcommand\HH{\mathrm{H}}
\newcommand\NN{\mathbb{N}}
\newcommand\ZZ{\mathbb{Z}}
\newcommand\QQ{\mathbb{Q}}
\newcommand\FF{\mathbb{F}}
\newcommand\KK{\mathbb{K}}

\newcommand\CC{\mathbb{C}}
\newcommand\git{/\!\!/}
\makeatother

\usepackage{babel}
\usepackage[style=alphabetic,maxnames=20]{biblatex}
\begin{document}
\title{Positivity for toric Kac polynomials in higher depth}
\author{Tanguy Vernet\thanks{Universit\"{a}t Duisburg-Essen, \texttt{tanguy.vernet@uni-due.de}}}
\maketitle
\begin{abstract}
We prove that the polynomials counting locally free, absolutely indecomposable,
rank 1 representations of quivers over rings of truncated power series
have non-negative coefficients. This generalizes the positivity of
toric Kac polynomials to higher depth. The proof goes by inductively
contracting/deleting arrows of the quiver and is inspired by a previous
work of Abdelgadir, Mellit and Rodriguez-Villegas on toric Kac polynomials.

We also relate counts of absolutely indecomposable quiver representations
in higher depth and counts of jets over fibers of quiver moment maps.
This is expressed as a plethystic identity involving generating series
of these counts. In rank 1, we prove a cohomological upgrade of this
identity, by computing the compactly supported cohomology of jet spaces
over preprojective stacks. This is reminiscent of PBW isomorphisms
for preprojective cohomological Hall algebras. Finally, our plethystic
identity allows us to prove two conjectures by Wyss on the asymptotic
behavior of both counts, when depth goes to infinity.
\end{abstract}

\section{Introduction}

Given a quiver $Q$ and a dimension vector $\dd\in\ZZ_{\geq0}^{Q_{0}}$,
there is a polynomial $A_{Q,\dd}$ which counts absolutely indecomposable
representations of $Q$ of dimension $\dd$ over finite fields. These
polynomials were introduced by Kac in \cite{Kac83} and have motivated
considerable work in geometric representation theory since then. Of
particular interest are conjectures \cite[Conj. 1-2.]{Kac83} on the
coefficients of $A_{Q,\dd}$, later refined by Bozec and Schiffmann
\cite[Conj. 1.3.]{BS19}. These conjectures state that the coefficients
of $A_{Q,\dd}$ are non-negative and can be interpreted as (graded)
dimensions of a Borcherds algebra. Their proofs largely rely on the
study of quiver moment maps $\mu_{Q,\dd}$ and their geometry. Kac's
original conjectures were proved by studying Nakajima's quiver varieties
\cite{Nak94,Nak98} and representations of Kac-Moody algebras on their
cohomology \cite{CBVB04,Hau10,HLRV13b}. More recently, the development
of cohomological Hall algebras built from the preprojective stacks
$\left[\mu_{Q,\dd}^{-1}(0)/\GL(\dd)\right]$ \cite{SV13b,RS17,YZ18a,YZ20,SV20}
led to the discovery of a suitable Borcherds algebra $\mathfrak{g}_{Q}$
whose graded dimensions are given by Kac's polynomials - in other
words, a proof of Bozec and Schiffmann's conjecture - along with new
actions of $\mathfrak{g}_{Q}$ on the cohomology of Nakajima's quiver
varieties \cite{DM20,Dav25a,DHSM23}.

In this paper, we are interested in generalizations of Kac's polynomials
to quiver representations over truncated polynomial rings $\mathcal{O}_{\alpha}:=\FF_{q}[t]/(t^{\alpha})$,
where $\alpha\geq1$ is an integer, which we (informally) call depth.
More precisely, we study the counts $A_{Q,\rr,\alpha}$ of locally
free, absolutely indecomposable representations of $Q$ of rank $\rr$
over $\mathcal{O}_{\alpha}$ (see Section \ref{Sect/QuiverRep} for
definitions). These are still quite mysterious; for instance, showing
that $A_{Q,\rr,\alpha}$ is polynomial in $q$ is still an open problem
for $\rr>\mathbf{1}$, where $\mathbf{1}\in\ZZ_{\geq0}^{Q_{0}}$ is
the rank vector with all entries equal to $1$\footnote{We use the partial order on $\ZZ^{Q_{0}}$ given by $\rr_{1}\geq\rr_{2}$
if, and only if, $\rr_{1}-\rr_{2}\in\ZZ_{\geq0}^{Q_{0}}$.}. More is known about toric representations, i.e. representations
of rank $\rr=\mathbf{1}$. These were recently studied by Hausel,
Letellier, Rodriguez-Villegas \cite{HLRV24} and Wyss \cite{Wys17b},
who established explicit polynomial formulas for $A_{Q,\mathbf{1},\alpha}$
and conjectured that $A_{Q,\mathbf{1},\alpha}$ has non-negative coefficients
\cite[Rmk. 7.7.i.]{HLRV24}. We prove this conjecture.

We also investigate the relation of $A_{Q,\rr,\alpha}$ with quiver
moment maps. Since we are working over $\mathcal{O}_{\alpha}$, a
natural object to consider is the moment map $\mu_{Q,\rr,\alpha}$
induced by $\mu_{Q,\rr}$ on jet spaces of quiver moduli. In \cite{Wys17b},
Wyss computed the counts $\#\mu_{Q,\rr,\alpha}^{-1}(0)(\FF_{q})$
for $\rr=\mathbf{1}$, in the form of an Igusa local zeta function,
and conjectured an \textit{asymptotic} relation between $\#\mu_{Q,\rr,\alpha}^{-1}(0)(\FF_{q})$
and $A_{Q,\rr,\alpha}$, when $\alpha$ goes to infinity. Wyss also
conjectured that the limits of both counts have non-negative coefficients
(see below for more details). In this paper, we find an identity between
generating series encoding $\#\mu_{Q,\rr,\alpha}^{-1}(0)(\FF_{q})$
and $A_{Q,\rr,\alpha}$, \textit{for all} $\rr$ and a \textit{fixed}
value of $\alpha$. This identity allows us to prove both of Wyss'
conjectures.

Regarding geometric representation theory, Geiss, Leclerc and Schr\"{o}er
recently explored representations of quivers over rings of truncated
power series in a series of papers \cite{GLS17a,GLS17b,GLS16,GLS18a,GLS18b}.
In particular, they use moduli of quiver representations in higher
depth to build realizations of symmetrizable Kac-Moody algebras \cite{GLS16,GLS18a}.
Their constructions exploit nilpotent subvarieties of $\mu_{Q,\rr,\alpha}^{-1}(0)$,
whose relations to cohomological Hall algebras are well understood
when working over a field \cite{Hen24a}. However, as far as we know,
no relations with $A_{Q,\rr,\alpha}$ have been found up to now. Instead,
using a purity argument, we find that the Poincar\'{e} series of $\left[\mu_{Q,\rr,\alpha}^{-1}(0)/\GL(\rr,\mathcal{O}_{\alpha})\right]$
in compactly supported cohomology recovers $A_{Q,\rr,\alpha}$ when
$\rr\leq\mathbf{1}$ - see Theorem \ref{Thm/IntroCohIntgr} below.
This is reminiscent of the PBW isomorphism for preprojective cohomological
Hall algebras established in \cite{Dav23c} and we regard this as
evidence that the cohomology of jet spaces over preprojective stacks
may also carry representation-theoretic structure. Let us describe
these results in more details.

\paragraph*{Higher depth Kac polynomials and preprojective stacks}

Our first results generalize the relation between Kac polynomials
and counts of $\FF_{q}$-points over $\mu_{Q,\dd}^{-1}(0)$ to higher
depth. The following formula was established by Mozgovoy \cite{Moz11a},
using dimensional reduction:\[
\sum_{\dd\in\NN^{Q_0}}
\frac{\#\mu_{Q,\dd}^{-1}(0)(\FF_q)}{\#\GL(\dd,\FF_q)}
\cdot q^{\langle\dd,\dd\rangle_Q}t^{\dd}
=
\Exp_{q,t}\left(
\sum_{\dd\in\NN^{Q_0}\setminus\{0\}}
\frac{A_{Q,\dd}(q)}{1-q^{-1}}\cdot t^{\dd}
\right),
\] where $\langle\bullet,\bullet\rangle_{Q}$ is the Euler form of $Q$.
More precisely, the formula relates the stacky point count of $\left[\mu_{Q,\dd}^{-1}(0)/\GL(\dd)\right]$
- the moduli stack of objects in a category of homological dimension
2 - to the count of objects in the category of representations of
$Q$, called $\Rep(Q)$, which has homological dimension 1. Because
$\Rep(Q)$ has dimension 1, the point count of $\left[\mu_{Q,\dd}^{-1}(0)/\GL(\dd)\right]$
is directly related to the count of objects of $\Rep(Q)$ with endomorphisms.
This is in turn related to Kac polynomials, using Krull-Schmidt decomposition
and Galois descent.

As was shown by Geiss, Leclerc and Schr\"{o}er \cite[Cor. 7.2]{GLS17a}\footnote{Note that Geiss, Leclerc and Schr\"{o}er assume that $Q$ has no oriented
cycles. However, the proof of \cite[Cor. 7.2]{GLS17a} works for arbitrary
quivers.}, locally free representations of $Q$ over $\mathcal{O}_{\alpha}$
have projective dimension at most $1$ in the abelian category of
\emph{all} representations of $Q$ over $\mathcal{O}_{\alpha}$.
We show that the above argument can be generalized to that setting,
which results in:

\begin{thm} \label{Thm/IntroExpFmlKacPol}

Let $Q$ be a quiver and $\alpha\geq1$. Then: \[
\sum_{\rr\in\NN^{Q_0}}
\frac{\#\mu_{Q,\rr,\alpha}^{-1}(0)(\FF_q)}{\#\GL(\rr,\mathcal{O}_{\alpha})}
\cdot q^{\alpha\langle\rr,\rr\rangle_Q}t^{\rr}
=
\Exp_{q,t}\left(
\sum_{\rr\in\NN^{Q_0}\setminus\{0\}}
\frac{A_{Q,\rr,\alpha}(q)}{1-q^{-1}}\cdot t^{\rr}
\right).
\]

\end{thm}

Note that an analogous formula holds for locally free representations
of quivers over a split local Frobenius algebra, see Remark \ref{Rmk/plethysticFmlFrobeniusAlg}.
One advantage of the above formula is that it holds for any fixed
$\alpha\geq1$. This strengthens the asymptotic connection between
$A_{Q,\rr,\alpha}$ and $\#\mu_{Q,\rr,\alpha}^{-1}(0)(\FF_{q})$ conjectured
by Wyss in \cite[\S 4.]{Wys17b}. Let us recall how this works. When
$\rr=\mathbf{1}$ and $Q$ is 2-connected\footnote{We say that a quiver is 2-connected if the underlying graph is connected
and remains connected after removing any arrow of $Q$, see Section
\ref{Sect/Graph}.}, Wyss showed that the following two sequences converge and that their
limits are rational functions in $q$:\begin{align*}
& A_{Q}(q):=
\underset{\alpha\rightarrow +\infty}{\lim}
\left(q^{-\alpha (1-\langle\rr,\rr\rangle_Q)}\cdot A_{Q,\mathbf{1},\alpha}(q)\right), \\
& B_{\mu_{Q}}(q):=
\underset{\alpha\rightarrow +\infty}{\lim}
\left(q^{-\alpha(2\# Q_1-\# Q_0+1)}\cdot\#\mu_{Q,\mathbf{1},\alpha}^{-1}(0)(\FF_q)\right).
\end{align*} He further conjectured that $A_{Q}$ and $B_{\mu_{Q}}$ are directly
related. Indeed, this is implied by Theorem \ref{Thm/IntroExpFmlKacPol}
:

\begin{cor} \label{Cor/IntroRelAvsB}

Let $Q$ be a 2-connected quiver. Then: \[
\frac{B_{\mu_{Q}}(q)}{(1-q^{-1})^{\# Q_0}}=\frac{A_{Q}(q)}{1-q^{-1}}.
\]

\end{cor}

Another approach to obtain Kac polynomials from quiver moment maps
uses the cohomology of generic fibers of $\mu_{Q,\dd}$. This strategy
led to the first proof of Kac's conjectures for indivisible dimension
vectors by Crawley-Boevey and Van den Bergh \cite{CBVB04}. The relation
between Kac polynomials and the cohomology of fibers of $\mu_{Q,\dd}$
at non-zero parameters was recently extended to all dimension vectors
by Davison, in \cite{Dav23a}. We prove a generalization in higher
depth (i.e. for $\alpha>1$) of Crawley-Boevey and Van den Bergh's
result. Given $\rr\in\NN^{Q_{0}}$, we say that $\lambda\in\ZZ^{Q_{0}}$
is generic with respect to $\rr$ if $\lambda\cdot\rr=0$ and $\lambda\cdot\rr'\ne0$
for all $0<\rr'<\rr$. Such a $\lambda$ exists if, and only if, $\rr$
is indivisible.

\begin{thm} \label{Thm/IntroCountMomMapGenFib}

Let $\rr\in\NN^{Q_{0}}$ be an indivisible rank vector and $\lambda\in\ZZ^{Q_{0}}$
be generic with respect to $\rr$. Then: \[
\frac{\#\mu_{Q,\rr,\alpha}^{-1}(t^{\alpha-1}\cdot\lambda)(\FF_q)}{\#\GL(\rr,\mathcal{O}_{\alpha})}=q^{-\alpha\langle\rr,\rr\rangle_Q}\cdot\frac{A_{Q,\rr,\alpha}(q)}{1-q^{-1}},
\]where we view $\lambda$ as $(\lambda_{i}\cdot\mathrm{id})_{i\in Q_{0}}\in\mathfrak{gl}(\rr,\mathcal{O}_{\alpha})$.

\end{thm}

\paragraph*{Positivity results}

We also prove that counts of toric quiver representations in higher
depth enjoy positivity properties, both in fixed depth $\alpha\geq1$
and asymptotically. The asymptotic result was also conjectured by
Wyss in \cite[\S 4.]{Wys17b} and roughly states that the numerator
of $B_{\mu_{Q}}$ (or equivalently, of $A_{Q}$ by Corollary \ref{Cor/IntroRelAvsB})
has non-negative coefficients. Our approach to this conjecture relies
on Wyss' explicit formula for $A_{Q}$:\[
A_Q(q)=
(1-q^{-1})^{b(Q)}\cdot
\sum_{E_1\subsetneq E_2\ldots\subsetneq E_s=Q_1}\prod_{j=1}^{s-1}\frac{1}{q^{b(Q)-b(Q\vert_{E_j})}-1}.
\] Here $b(Q)$ stands for the Betti number of the graph underlying
$Q$ (see Section \ref{Sect/Graph} for details). We show that $A_{Q}$
is essentially the Hilbert series of the Stanley-Reisner ring of a
Cohen-Macaulay simplicial complex (see Section \ref{Sect/SRrings}
for background). This was inspired by \cite{MV24}, where similar
combinatorial formulas for local Igusa zeta functions are studied.
We obtain the following:

\begin{thm} \label{Thm/IntroPositivityA}

Let $Q$ be a 2-connected quiver. Consider the Stanley-Reisner ring
$\QQ[\Delta]$ associated to the order complex $\Delta$ of $(\Pi(Q_{1})\setminus\{Q_{1}\},\subseteq)$,
the poset of (non-empty, proper) subsets of $Q_{1}$ (see Section
\ref{Sect/SRrings} for definitions). Then:

\[
A_Q(q)=(1-q^{-1})^{b(Q)}\cdot\Hilb_{\Delta}\left(u_{E}=q^{-(b(Q)-b(Q\vert_{E}))}\right),
\] and $\Hilb_{\Delta}\left(u_{E}=q^{-(b(Q)-b(Q\vert_{E}))}\right)$
can be presented as a rational fraction in $q$ whose numerator has
non-negative coefficients.

\end{thm}

On the other hand, assuming that $\alpha$ is finite and fixed, we
also prove that $A_{Q,\rr,\alpha}$ has non-negative coefficients,
when $\rr=\mathbf{1}$. For toric Kac polynomials, there is a beautiful
graph-theoretic proof of Kac's positivity conjecture involving Tutte
polynomials \cite{AMRV22}. The key observation is that toric Kac
polynomials can be computed recursively, using a contraction-deletion
algorithm on arrows of $Q$. In higher depth, the contraction-deletion
argument must take into account valuations in the ring $\mathcal{O}_{\alpha}$
and the connection to Tutte polynomials is lost. Nevertheless, this
recursion on $Q$ is enough to prove:

\begin{thm} \label{Thm/IntroPosToricKacPol}

Let $\rr=\mathbf{1}$. Then $A_{Q,\rr,\alpha}$ has non-negative coefficients.

\end{thm}

Contraction-deletion of edges is also key to our computation of $\HH_{\mathrm{c}}^{\bullet}\left(\left[\mu_{Q,\rr,\alpha}^{-1}(0)/\GL(\rr,\mathcal{O}_{\alpha})\right]\right)$,
which we now explain.

\paragraph*{Cohomology of toric preprojective stacks in higher depth}

A fruitful strategy to prove that Kac polynomials have non-negative
coefficients is to extract $A_{Q,\dd}$ from the cohomology of pure
algebraic varieties (or stacks). This was done with quiver varieties
in \cite{CBVB04,HLRV13b} and preprojective stacks in \cite{Dav18,Dav23c}.

When $\rr=\mathbf{1}$, we show that this cohomological approach also
works in higher depth and compute the (compactly supported) cohomology
of preprojective stacks over $\KK=\CC$ (with rational coefficients).
We obtain that the mixed Hodge structure of $\HH_{\mathrm{c}}^{\bullet}\left(\left[\mu_{Q,\rr,\alpha}^{-1}(0)/\GL(\rr,\mathcal{O}_{\alpha})\right]\right)$
is pure and that its graded dimensions can be computed from $A_{Q,\rr',\alpha},\ \rr'\leq\mathbf{1}$.
This is a cohomological upgrade of Theorems \ref{Thm/IntroExpFmlKacPol}
and \ref{Thm/IntroPosToricKacPol}. Our proof relies on a stratification
of $\left[\mu_{Q,\rr,\alpha}^{-1}(t^{\alpha-1}\cdot\lambda)/\GL(\rr,\mathcal{O}_{\alpha})\right]$,
based on the contraction-deletion algorithm mentioned above, which
we further extend to $\left[\mu_{Q,\rr,\alpha}^{-1}(0)/\GL(\rr,\mathcal{O}_{\alpha})\right]$.

Let us call $\mathbb{L}:=\HH_{\mathrm{c}}^{\bullet}(\mathbb{A}^{1})$,
seen as a $\ZZ$-graded polarizable mixed Hodge structure over $\QQ$.
For a polynomial $P(T)=\sum_{i}a_{i}T^{i}\in\ZZ_{\geq0}[T]$, we use
the notation $P(\mathbb{L}):=\bigoplus_{i}(\mathbb{L}^{\otimes i})^{\oplus a_{i}}$.
We refer to Section \ref{Sect/EqCoh} for further details on cohomology
of stacks and its mixed Hodge structure. Our main cohomological results
are the following:

\begin{thm} \label{Thm/IntroCohMomMapGenFib}

Let $\rr=\mathbf{1}$. Suppose that $\lambda$ is generic with respect
to $\rr$. There is a non-canonical isomorphism of $\ZZ$-graded polarizable
mixed Hodge structures over $\QQ$: \[
\HH_{\mathrm{c}}^{\bullet}\left(\left[\mu_{Q,\rr,\alpha}^{-1}(t^{\alpha-1}\cdot\lambda)/\GL(\rr,\mathcal{O}_{\alpha})\right]\right)
\simeq
A_{Q,\rr,\alpha}(\mathbb{L})\otimes\mathbb{L}^{1-\alpha\langle\rr,\rr\rangle_Q}\otimes\HH_{\mathrm{c}}^{\bullet}\left(\mathrm{B}\mathbb{G}_{\mathrm{m}}\right)
.
\]In particular, $\HH_{\mathrm{c}}^{\bullet}\left(\left[\mu_{Q,\rr,\alpha}^{-1}(t^{\alpha-1}\cdot\lambda)/\GL(\rr,\mathcal{O}_{\alpha})\right]\right)$
carries a pure Hodge structure.

\end{thm}

\begin{thm} \label{Thm/IntroCohIntgr}

Let $\rr=\mathbf{1}$. There is a non-canonical isomorphism of $\ZZ$-graded
polarizable mixed Hodge structures over $\QQ$:\[
\HH_{\mathrm{c}}^{\bullet}\left(\left[\mu_{Q,\rr,\alpha}^{-1}(0)/\GL(\rr,\mathcal{O}_{\alpha})\right]\right)
\otimes\mathbb{L}^{\otimes\alpha\langle\rr,\rr\rangle_Q}
\simeq
\bigoplus_{Q_0=I_1\sqcup\ldots\sqcup I_s}
\bigotimes_{j=1}^s
\left(
A_{Q\vert_{I_j},\rr\vert_{I_j},\alpha}(\mathbb{L})\otimes\mathbb{L}\otimes \HH_{\mathrm{c}}^{\bullet}(\mathrm{B}\mathbb{G}_m)
\right).
\]In particular, $\HH_{\mathrm{c}}^{\bullet}\left(\left[\mu_{Q,\rr,\alpha}^{-1}(0)/\GL(\rr,\mathcal{O}_{\alpha})\right]\right)$
carries a pure Hodge structure.

\end{thm}

Theorem \ref{Thm/IntroCohIntgr} can be seen as an analog in higher
depth of the PBW isomorphism for preprojective cohomological Hall
algebras \cite{Dav23c}, restricted to $\rr=\mathbf{1}$. This is
a key ingredient in Davison's reproof of Kac's positivity conjecture\footnote{However, in our setting, Theorem \ref{Thm/IntroPosToricKacPol} is
not a consequence of Theorem \ref{Thm/IntroCohIntgr}, see Remark
\ref{Rmk/CohUpgrade=000026Positivity}.} \cite{Dav18}. We take this result as a first piece of evidence that
similar structures may exist in higher depth, which we hope to address
in future works.

\paragraph*{Beyond the toric case}

As mentioned above, beyond the case where $\rr=\mathbf{1}$, the counting
functions $A_{Q,\rr,\alpha}$ remain largely unknown. One reason for
this is that Kac and Stanley's computations of $A_{Q,\dd}$ \cite{Kac83}
(see also \cite{Hua00a}) rely on a good understanding of conjugacy
classes in $\GL(\dd)$. To the best of our knowledge, conjugacy classes
of $\GL(r,\mathcal{O}_{\alpha})$ are not classified for $r\geq4$.
When $r\leq3$, we exploit results of Avni, Klopsch, Onn and Voll
\cite{AKOV16} to prove the following explicit formulas:

\begin{prop} \label{Prop/IntroKacPolLowRk}

Let $Q$ be the quiver with one vertex and $g\geq1$ loops. Then:\[
\begin{split}
A_{Q,2,\alpha}= & \frac{q^{2\alpha g-1}(q^{2g}-1)(q^{\alpha(2g-3)}-1)}{(q^2-1)(q^{2g-3}-1)}, \\
A_{Q,3,\alpha}= &
\frac{q^{3\alpha g-2}(q^{2g}-1)(q^{2g-1}-1)}{(q^2-1)(q^3-1)(q^{2g-3}-1)(q^{6g-8}-1)(q^{4g-5}-1)} \\
 & \cdot
\left(
q^{\alpha(6g-8)-1}(q^{6g-7}-1)(q^{2g}+1)
-q^{\alpha(6g-8)+2g-4}(q^2-1)(q^{4g-3}+1) \right. \\
 & \left. +q^{\alpha(2g-3)-1}(q^2+q+1)(q^{2g-1}-1)(q^{6g-8}-1)+(q+1)(q^{8g-10}-1)+q^{2g-4}(q^4+1)(q^{4g-5}-1)
\right)
.
\end{split}
\]In particular, $A_{Q,2,\alpha}$ and $A_{Q,3,\alpha}$ are polynomials
and $A_{Q,2,\alpha}$ has non-negative coefficients.

\end{prop}

Although it is not clear from the formula above that $A_{Q,3,\alpha}\in\ZZ_{\geq0}[q]$,
we check with a computer that this is true for small values of $\alpha$
and $g$. This provides some evidence that $A_{Q,\rr,\alpha}$ is
a polynomial with non-negative coefficients for all $\rr\in\ZZ_{\geq0}^{Q_{0}}$
and $\alpha\geq1$. We record this in the following:

\begin{cj}

Let $Q$ be a quiver, $\rr\in\ZZ_{\geq0}^{Q_{0}}$ and $\alpha\geq1$.
Then $A_{Q,\rr,\alpha}\in\ZZ_{\geq0}[q]$ .

\end{cj}

\paragraph*{Plan of the paper}

In Section \ref{Sect/QuiverRep}, we introduce basic facts on locally
free representations of quivers in higher depth, generalising well-known
results over base fields, such as Krull-Schmidt decomposition and
Galois descent techniques. We then set notations concerning plethystic
formulas and graphs in Sections \ref{Sect/Plethysm} and \ref{Sect/Graph}.
In Section \ref{Sect/SRrings}, we collect relevant facts on Stanley-Reisner
rings and their Hilbert series, used in the proof of Theorem \ref{Thm/IntroPositivityA}.
Our last prerequisites concern mixed Hodge structures on equivariant
cohomology of algebraic varieties and their relation to point-counts
over finite fields. We also collect computational tools for the proof
of Theorem \ref{Thm/IntroCohIntgr}. This is done in Section \ref{Sect/EqCoh}.

We prove Theorems \ref{Thm/IntroExpFmlKacPol}, \ref{Thm/IntroCountMomMapGenFib}
in Section \ref{Sect/KacPolvsMomMap}, as well as Corollary \ref{Cor/IntroRelAvsB}.
Section \ref{Section/PosPurity} is devoted to the proof of our positivity
and purity results in the toric setting. Theorem \ref{Thm/IntroPositivityA}
is proved in Section \ref{Sect/PosA}, Theorem \ref{Thm/IntroPosToricKacPol}
in Section \ref{Sect/PosToricKacPol}, whereas Section \ref{Sect/CohPreprojStack}
contains the proofs of Theorems \ref{Thm/IntroCohMomMapGenFib} and
\ref{Thm/IntroCohIntgr}. Finally, we discuss higher-rank counts and
Proposition \ref{Prop/IntroKacPolLowRk} in Section \ref{Sect/HigherRk}.

\paragraph*{Acknowledgements}

I would like to warmly thank Dimitri Wyss for his constant support
throughout this project and Emmanuel Letellier for helpful discussions
and suggesting to look into \cite{AMRV22}. I would also like to thank
Ben Davison, Tam\'as Hausel, Anton Mellit, Olivier Schiffmann and
Sebastian Schlegel-Mejia for helpful conversations. Finally, I would
like to thank the anonymous referee for their careful reading and
suggesting many improvements. This work was supported by the Swiss
National Science Foundation {[}No. 196960{]}. The project has also
received funding from the European Union\textquoteright s Horizon
2020 research and innovation programme under the Marie Sk\l odowska-Curie
Grant Agreement No. 101034413. It was further funded by the Deutsche
Forschungsgemeinschaft (DFG, German Research Foundation) - Project
No. 585538853.

\section{Preliminaries}

\subsection{Quiver representations in higher depth \protect\label{Sect/QuiverRep}}

In this section, we set notations for locally free representations
of quivers over a base ring (typically $\mathbb{F}_{q}[t]/(t^{\alpha})$)
and their moduli. We also recall and adapt to our setting a few well-known
results on endomorphism rings of quiver representations and their
relation with quiver moment maps.

\paragraph*{Locally free representations of quivers}

A quiver is the datum $Q=(Q_{0},Q_{1},s,t)$ of a set of vertices
$Q_{0}$, a set of arrows $Q_{1}$ and maps $s,t:Q_{1}\rightarrow Q_{0}$
which assign to an arrow $a\in Q_{1}$ its source $s(a)$ and target
$t(a)$ respectively. A quiver may have parallel edges and loops.

Consider a base ring $R$. A representation $M$ of $Q$ over $R$
is the datum of $R$-modules $M_{i},\ i\in Q_{0}$ and $R$-linear
maps $f_{a}:M_{s(a)}\rightarrow M_{t(a)},\ a\in Q_{1}$. A morphism
of representations $\varphi:M\rightarrow M'$ is the datum of $R$-linear
maps $\varphi_{i}:M_{i}\rightarrow M'_{i}$ such that $f'_{a}\circ\varphi_{s(a)}=\varphi_{t(a)}\circ f_{a}$
for all $a\in Q_{1}$.

A representation $M$ is called locally free if all $M_{i},\ i\in Q_{0}$
are free $R$-modules. If $M_{i}$ is finitely generated for every
$i\in Q_{0}$, we call $\mathbf{rk}(M):=(\rk(M_{i}))_{i\in Q_{0}}$
the rank vector of $M$. In this paper, we will always work with locally
free representations of finite rank.

We denote by $\Rep_{R}(Q)^{\mathrm{l.f.}}$ the category of finite-rank,
locally free representations of $Q$ over $R$. When $R$ is a finite-dimensional
algebra over a field $\KK$, $\Rep_{R}(Q)^{\mathrm{l.f.}}$ is an
exact subcategory of the abelian category $\Rep_{R}(Q)$ of \emph{all}
representations of $Q$ over $R$, as it is closed under extensions
\cite[Lem. 3.8]{GLS17a}\footnote{The proof in \cite{GLS17a} is done for $R=\KK[t]/(t^{\alpha})$,
but it remains valid for any finite-dimensional $\KK$-algebra.}. See \cite{Bue10} for basics of exact categories.

Finally, we recall the definition of the Euler form of $Q$:\[
\begin{array}{cccc}
\langle\bullet,\bullet\rangle_Q: & \ZZ^{Q_0}\times\ZZ^{Q_0} & \rightarrow & \ZZ^{Q_0} \\
& (\dd,\ee) & \mapsto & \sum_{i\in Q_0}d_ie_i-\sum_{a:i\rightarrow j}d_ie_j.
\end{array}
\]

\paragraph*{Moduli of quiver representations in higher depth}

Let us now set a quiver $Q$ and $R=\mathcal{O}_{\alpha}:=\KK[t]/(t^{\alpha})$,
where $\alpha\geq1$ is an integer and $\KK$ is a field. Similarly
to the case where the base ring is a field (i.e. $\alpha=1$), locally
free representations of quivers of a given rank vector are parametrized
by a quotient stack (see \cite[\S 3]{Rei08a} for the corresponding
group action). Let us define:\[
R(Q,\rr,\mathcal{O}_{\alpha}):=
\prod_{\substack{a\in Q_1 \\ a:i\rightarrow j}}
\Hom_{\mathcal{O}_{\alpha}}(\mathcal{O}_{\alpha}^{\oplus r_i},\mathcal{O}_{\alpha}^{\oplus r_j}),
\]
\[
\GL(\rr,\mathcal{O}_{\alpha}):=\prod_{i\in Q_0}\GL(r_i,\mathcal{O}_{\alpha}),
\] which we view respectively as an algebraic variety and an algebraic
group over $\KK$. Then $\GL(\rr,\mathcal{O}_{\alpha})$ acts on $R(Q,\rr,\mathcal{O}_{\alpha})$
similarly to the case where $\alpha=1$ and the quotient stack $\left[R(Q,\rr,\mathcal{O}_{\alpha})/\GL(\rr,\mathcal{O}_{\alpha})\right]$
parametrizes locally free representations of $Q$ of rank vector $\rr$.
Given $x\in R(Q,\rr,\mathcal{O}_{\alpha})$, we call $M(x)$ the associated
locally free representation of $Q$.

\begin{rmk} \label{Rmk/JetSpaces}

Recall that the $(\alpha-1)$-th jet space of a $\KK$-scheme $X$
is the scheme $J_{\alpha-1}(X)$ determined by $J_{\alpha-1}(X)(A)=X(A[t]/(t^{\alpha}))$
for any $\KK$-algebra $A$. Then $R(Q,\rr,\mathcal{O}_{\alpha})$
and $\GL(\rr,\mathcal{O}_{\alpha})$ are the $(\alpha-1)$-th jet
spaces of $R(Q,\rr,\KK)$ and $\GL(\rr,\KK)$ respectively. The above
action can be obtained by extending to jet spaces the classical action
of $\GL(\rr,\KK)$ on $R(Q,\rr,\KK)$ (the construction of jet spaces
is functorial - see \cite[Ch. 3]{CLNS18}).

\end{rmk}

\paragraph*{Moment map}

The action of $\GL(\rr,\mathcal{O}_{\alpha})$ on $R(Q,\rr,\mathcal{O}_{\alpha})$
also extends to its cotangent bundle and gives rise to a moment map
$\mu_{Q,\rr,\alpha}:R(\overline{Q},\rr,\mathcal{O}_{\alpha})\rightarrow\mathfrak{gl}(\rr,\mathcal{O}_{\alpha})$,
where $\overline{Q}$ is the double quiver of $Q$. Let us first clarify
how $R(\overline{Q},\rr,\mathcal{O}_{\alpha})$ is identified to $\mathrm{T}^{*}R(Q,\rr,\mathcal{O}_{\alpha})$
- see also \cite[\S 6.1.]{HWW23}.

The $\KK$-algebra $\mathcal{O}_{\alpha}$ is a Frobenius algebra,
that is, it admits a $\KK$-linear functional $\pi:\mathcal{O}_{\alpha}\rightarrow\KK$
which induces an isomorphism of $\mathcal{O}_{\alpha}$-modules $\mathcal{O}_{\alpha}\simeq\Hom_{\KK}(\mathcal{O}_{\alpha},\KK)$,
$a\mapsto\pi(a\bullet)$ (see \cite[Ch. 16]{Lam99} for background).
Concretely, $\pi$ is given by $\pi(t^{k})=\delta_{k,\alpha-1}$.
Then, we have the following non-degenerate, $\KK$-linear pairing
on $\mathcal{O}_{\alpha}$:\[
\begin{array}{ccc}
\mathcal{O}_\alpha\times\mathcal{O}_\alpha & \rightarrow & \KK \\
(a,b) & \mapsto & \langle a,b\rangle:=\pi(ab).
\end{array}
\] 

Generalizing this, for integers $m,n\geq1$, we can identify the dual
$\KK$-vector space $\Hom_{\mathcal{O}_{\alpha}}(\mathcal{O}_{\alpha}^{\oplus n},\mathcal{O}_{\alpha}^{\oplus m})^{\vee}$
to $\Hom_{\mathcal{O}_{\alpha}}(\mathcal{O}_{\alpha}^{\oplus m},\mathcal{O}_{\alpha}^{\oplus n})$
using the following trace pairing, which we also denote by $\langle\bullet,\bullet\rangle$:\[
\begin{array}{ccc}
\Hom_{\mathcal{O}_{\alpha}}(\mathcal{O}_{\alpha}^{\oplus n},\mathcal{O}_{\alpha}^{\oplus m})
\times
\Hom_{\mathcal{O}_{\alpha}}(\mathcal{O}_{\alpha}^{\oplus m},\mathcal{O}_{\alpha}^{\oplus n})
& \rightarrow & \KK \\
(A,B) & \mapsto & \langle A,B\rangle:=\pi\left(\Tr(AB)\right)
.
\end{array}
\] This pairing was considered in \cite[Prop. 5.1]{HLRV24} in the context
of quiver representations over Frobenius algebras.

Using the trace pairing, we obtain the following isomorphism of $\KK$-vector
spaces: \[
\begin{split}
\mathrm{T}^*R(Q,\rr,\mathcal{O}_{\alpha})
& =
\prod_{\substack{a\in Q_1 \\ a:i\rightarrow j}}
\left(
\Hom_{\mathcal{O}_{\alpha}}(\mathcal{O}_{\alpha}^{\oplus r_i},\mathcal{O}_{\alpha}^{\oplus r_j})
\times
\Hom_{\mathcal{O}_{\alpha}}(\mathcal{O}_{\alpha}^{\oplus r_i},\mathcal{O}_{\alpha}^{\oplus r_j})^{\vee}
\right) \\
& \simeq
\prod_{\substack{a\in Q_1 \\ a:i\rightarrow j}}
\left(
\Hom_{\mathcal{O}_{\alpha}}(\mathcal{O}_{\alpha}^{\oplus r_i},\mathcal{O}_{\alpha}^{\oplus r_j})
\times
\Hom_{\mathcal{O}_{\alpha}}(\mathcal{O}_{\alpha}^{\oplus r_j},\mathcal{O}_{\alpha}^{\oplus r_i})
\right)
= R(\overline{Q},\rr,\mathcal{O}_{\alpha}),
\end{split}
\] where $\overline{Q}$ is the double quiver of $Q$, i.e. the quiver
obtained by adding to $Q$ an arrow $a^{*}:j\rightarrow i$ for each
arrow $Q_{1}\ni a:i\rightarrow j$. Likewise, using the trace pairing,
we obtain an isomorphism of $\KK$-vector spaces $\mathfrak{gl}(\rr,\mathcal{O}_{\alpha})^{\vee}\simeq\mathfrak{gl}(\rr,\mathcal{O}_{\alpha})$.

The action $\GL(\rr,\mathcal{O}_{\alpha})\curvearrowright\mathrm{T}^{*}R(Q,\rr,\mathcal{O}_{\alpha})$
is hamiltonian, i.e. it admits a moment map $\mu_{Q,\rr,\alpha}:\mathrm{T}^{*}R(Q,\rr,\mathcal{O}_{\alpha})\rightarrow\mathfrak{gl}(\rr,\mathcal{O}_{\alpha})^{\vee}$
(see \cite[\S 1.4]{CG97} for background on moment maps). Under the
above identifications, the moment map has the following expression,
which is similar to the case where $\alpha=1$. We denote a point
of $R(\overline{Q},\rr,\mathcal{O}_{\alpha})$ by $(x_{a},y_{a})_{a\in Q_{1}}$,
where $x_{a}\in\Hom_{\mathcal{O}_{\alpha}}(\mathcal{O}_{\alpha}^{\oplus r_{s(a)}},\mathcal{O}_{\alpha}^{\oplus r_{t(a)}})$
and $y_{a}\in\Hom_{\mathcal{O}_{\alpha}}(\mathcal{O}_{\alpha}^{\oplus r_{t(a)}},\mathcal{O}_{\alpha}^{\oplus r_{s(a)}})$.
Then:\[
\begin{array}{cccc}
\mu_{Q,\rr,\alpha}: & R(\overline{Q},\rr,\mathcal{O}_{\alpha}) & \rightarrow & \mathfrak{gl}(\rr,\mathcal{O}_{\alpha}) \\
& (x_{a},y_{a})_{a\in Q_{1}} & \mapsto & \left(\sum_{t(a)=i}x_ay_a-\sum_{s(a)=i}y_ax_a\right)_{i\in Q_0}
\end{array}
.
\]

Finally, we record the following proposition, which characterizes
fibers of the forgetful map $\mu_{Q,\rr,\alpha}^{-1}(0)\rightarrow R(Q,\rr,\mathcal{O}_{\alpha})$,
$(x,y)\mapsto x$. The proof closely follows \cite[Lem. 4.2.]{CBH98}.

\begin{prop} \label{Prop/MomMapExactSeq}

Let $x\in R(Q,\rr,\mathcal{O}_{\alpha})$ and $M$ the corresponding
locally free representation of $Q$. Then there is an exact sequence:
\[
\begin{tikzcd}[ampersand replacement = \&, column sep=small, row sep=large]
0 \arrow[r] \& \Ext^1(M,M)^{\vee} \arrow[r] \& R(Q^{\mathrm{op}},\rr,\mathcal{O}_{\alpha}) \arrow[r, "{\mu_{Q,\rr,\alpha}(x,\bullet)}"] \&[4em] \mathfrak{gl}(\rr,\mathcal{O}_{\alpha}) \arrow[r] \& \End(M)^{\vee} \arrow[r] \& 0,
\end{tikzcd}
\]where $Q^{\mathrm{op}}$ is the quiver obtained from $Q$ by reversing
all arrows and $\mathfrak{gl}(\rr,\mathcal{O}_{\alpha})\rightarrow\End(M)^{\vee}$
is the dual of the inclusion $\End(M)\hookrightarrow\mathfrak{gl}(\rr,\mathcal{O}_{\alpha})$
(using the identification $\mathfrak{gl}(\rr,\mathcal{O}_{\alpha})^{\vee}\simeq\mathfrak{gl}(\rr,\mathcal{O}_{\alpha})$).

\end{prop}

\begin{proof}

Let $\mathcal{O}_{\alpha}Q:=\mathcal{O}_{\alpha}\otimes_{\KK}\KK Q$
be the path algebra of $Q$ with coefficients in $\mathcal{O}_{\alpha}$
(see \cite[\S 1]{CB92}) and denote by $e_{i}\in\mathcal{O}_{\alpha}Q$
the path starting and ending at vertex $i\in Q_{0}$. The category
of representations of $Q$ over $\mathcal{O}_{\alpha}$ is equivalent
to $\mathcal{O}_{\alpha}Q-\mathrm{Mod}$. By \cite[Cor. 7.2.]{GLS17a},
since $M$ is locally free, it admits the following projective resolution:\[
0\rightarrow
\bigoplus_{\substack{a\in Q_1 \\ a:i\rightarrow j}}\mathcal{O}_{\alpha}Q e_j \otimes_{\mathcal{O}_{\alpha}}M_i\overset{\varphi}{\longrightarrow}
\bigoplus_{i\in Q_0}\mathcal{O}_{\alpha}Q e_i\otimes_{\mathcal{O}_{\alpha}}M_i\overset{\psi}{\longrightarrow}
M\rightarrow0,
\] where, for $\rho=a_{1}\ldots a_{s}$ a path in $Q$, $\varphi(\rho e_{t(a)}\otimes m)=\rho ae_{s(a)}\otimes m-\rho e_{t(a)}\otimes f_{a}(m)$
and $\psi(\rho e_{i}\otimes m)=f_{\rho}(m)$ (here $f_{\rho}$ is
a shorthand for $f_{a_{1}}\circ\ldots\circ f_{a_{s}}$). Indeed, in
the notation of \cite{GLS17a}, $\mathcal{O}_{\alpha}Q\simeq H(C,D,\Omega)$,
where $C$ is the symmetric Cartan matrix associated to $Q$, $D=\alpha\cdot\mathrm{Id}$
and $\Omega$ is the orientation of $Q$\footnote{Note that Geiss, Leclerc and Schr\"{o}er work with quivers without oriented
cycles, but the proof of \cite[Cor. 7.2.]{GLS17a} works for arbitrary
quivers.}.

The required exact sequence is then obtained by applying the functor
$\Hom_{\mathcal{O}_{\alpha}Q}(\bullet,M)$ to the above resolution
and dualizing, as in \cite[Lem. 4.2.]{CBH98}. Indeed, for every $i,j\in Q_{0}$
and $a\in Q_{1}$, such that $a:i\rightarrow j$, we have:\[
\Hom_{\mathcal{O}_{\alpha}Q}\left(\mathcal{O}_{\alpha}Q e_i\otimes_{\mathcal{O}_{\alpha}}M_i,M\right)
\simeq
\Hom_{\mathcal{O}_{\alpha}}\left(M_i,\Hom_{\mathcal{O}_{\alpha}Q}(\mathcal{O}_{\alpha}Q e_i,M)\right)
\simeq
\Hom_{\mathcal{O}_{\alpha}}(M_i,M_i)
,
\]
\[
\Hom_{\mathcal{O}_{\alpha}Q}\left(\mathcal{O}_{\alpha}Q e_j\otimes_{\mathcal{O}_{\alpha}}M_i,M\right)
\simeq
\Hom_{\mathcal{O}_{\alpha}}\left(M_i,\Hom_{\mathcal{O}_{\alpha}Q}(\mathcal{O}_{\alpha}Q e_j,M)\right)
\simeq
\Hom_{\mathcal{O}_{\alpha}}(M_i,M_j)
,
\]where an $\mathcal{O}_{\alpha}$-linear homomorphism $h:M_{i}\rightarrow M_{i}$
(resp. $h:M_{i}\rightarrow M_{j}$) is mapped to the homomorphism
of $\mathcal{O}_{\alpha}Q$-modules $\rho e_{i}\otimes m\mapsto\rho\cdot h(m)$
(resp. $\rho e_{j}\otimes m\mapsto\rho\cdot h(m)$), for a path $\rho\in\mathcal{O}_{\alpha}Q$
and $m\in M_{i}$. Under these isomorphisms, the long exact sequence
obtained from applying $\Hom_{\mathcal{O}_{\alpha}Q}(\bullet,M)$
to the above projective resolution is:\[
0\rightarrow\End_{\mathcal{O}_{\alpha}Q}(M)
\overset{\subset}{\rightarrow}
\mathfrak{gl}(\rr,\mathcal{O}_\alpha)
\overset{\xi\mapsto\xi\cdot x}{\rightarrow}
R(Q,\rr,\mathcal{O}_\alpha)
\rightarrow
\Ext_{\mathcal{O}_{\alpha}Q}(M,M)
\rightarrow
0
,
\]where $\xi\cdot x$ denotes the infinitesimal action of $\xi\in\mathfrak{gl}(\rr,\mathcal{O}_{\alpha})$
on $x\in R(Q,\rr,\mathcal{O}_{\alpha})$. By dualizing the above sequence
of $\KK$-vector spaces, we obtain the required exact sequence, since
for all $\xi\in\mathfrak{gl}(\rr,\mathcal{O}_{\alpha})$ and $(x,y)\in R(\overline{Q},\rr,\mathcal{O}_{\alpha})$,
we have by definition $\langle\xi\cdot x,y\rangle=\langle\xi,\mu_{Q,\rr,\alpha}(x,y)\rangle$.
\end{proof}

\paragraph*{Krull-Schmidt theorem and absolutely indecomposable representations}

Let us now assume that $\KK=\FF_{q}$. We collect some standard results
on absolutely indecomposable representations and their rings of endomorphisms.
We also check that standard Galois descent techniques apply to $\Rep_{\mathcal{O}_{\alpha}}(Q)^{\mathrm{l.f.}}$,
following \cite{Moz19}.

We start with some definitions. Let $M$ be a locally free representation
of $Q$ over $\mathcal{O}_{\alpha}$. We say that $M$ is indecomposable
if it cannot be split into a direct sum of non-zero (locally free\footnote{Note that a direct summand of a locally free representation of $Q$
over $\mathcal{O}_{\alpha}$ is also locally free. Indeed, over a
local $\KK$-algebra, a direct summand of a free module is free, as
it is projective.}) subrepresentations. Given a field extension $\KK'/\KK$, we denote
by $M_{\KK'}=M\otimes_{\KK}\KK'$ the representation obtained from
$M$ by base change, i.e. $\left(M_{\KK'}\right)_{i}=M_{i}\otimes_{\KK}\KK'\simeq M_{i}\otimes_{\mathcal{O}_{\alpha}}(\mathcal{O}_{\alpha}\otimes_{\KK}\KK')$
for $i\in Q_{0}$ and $\left(f_{\KK'}\right)_{a}=f_{a}\otimes\mathrm{id}_{\KK'}$
for $a\in Q_{1}$. Note that $M_{\KK'}$ is a locally free representation
of $Q$ over the $\KK'$-algebra $\mathcal{O}_{\alpha}\otimes_{\KK}\KK'$.
Let $\overline{\KK}$ be an algebraic closure of $\KK$. A locally
free representation $M$ is called absolutely indecomposable if $M_{\overline{\KK}}$
is indecomposable.

In this paper, we use well-known properties of absolutely indecomposable
representations, which rely on standard Galois descent arguments.
These arguments were formalized by Mozgovoy in \cite[\S 3]{Moz19}
through the notion of a linear stack. In order to apply his results,
we show that the assignment $\KK'\mapsto\Rep_{\mathcal{O}_{\alpha}\otimes\KK'}(Q)^{\mathrm{l.f.}}$
gives a linear stack on the small \'{e}tale site over $\KK$.

We check that Galois descent holds for objects and morphisms. Descent
for objects follows from the fact that $\left[R(Q,\rr,\mathcal{O}_{\alpha})/\GL(\rr,\mathcal{O}_{\alpha})\right]$
is an Artin stack. Note that $\left[R(Q,\rr,\mathcal{O}_{\alpha})/\GL(\rr,\mathcal{O}_{\alpha})\right](\KK')$
is equivalent to the groupoid of $\Rep_{\mathcal{O}_{\alpha}\otimes\KK'}(Q)^{\mathrm{l.f.}}$.
Indeed, since $\GL(\rr,\mathcal{O}_{\alpha})$ is a connected algebraic
group over $\KK=\FF_{q}$, Lang's theorem \cite[Thm. 4.4.17]{Spr98}
implies that all principal $\GL(\rr,\mathcal{O}_{\alpha})$-bundles
over $\Spec(\KK')$ are trivial \cite[Exmp. 12.3.5.6]{Spr98}. Descent
for morphisms follows from the following Galois-equivariant isomorphism:
given a field extension $\KK''/\KK'$ and $M,N\in\Rep_{\mathcal{O}_{\alpha}\otimes\KK'}(Q)^{\mathrm{l.f.}}$,
we have\[
\Hom(M_{\KK''},N_{\KK''})\simeq\Hom(M,N)\otimes_{\KK'}\KK''
.
\] Then we obtain $\Hom(M,N)\simeq\Hom(M_{\KK''},N_{\KK''})^{\mathrm{Gal}(\KK''/\KK')}$,
i.e. Galois descent for morphisms. Lastly, we check that $\Rep_{\mathcal{O}_{\alpha}\otimes\KK'}(Q)^{\mathrm{l.f.}}$
is Krull-Schmidt (see \cite[\S 2.3.]{Moz19}):

\begin{prop} \label{Prop/Krull-Schmidt}

$\Rep_{\mathcal{O}_{\alpha}}(Q)^{\mathrm{l.f.}}$ is a Krull-Schmidt
category, i.e. any objects splits into a direct sum of indecomposable
objects and the ring of endomorphisms of an indecomposable object
is local.

\end{prop}

\begin{proof}

It is clear, by induction on rank vectors, that any locally free,
finite-rank representation of $Q$ over $\mathcal{O}_{\alpha}$ can
be decomposed into a direct sum of indecomposable representations
(which are also locally free, as mentioned above).

Let $M\in\Rep_{\mathcal{O}_{\alpha}}(Q)^{\mathrm{l.f.}}$ be indecomposable.
We prove that any endomorphism of $M$ is either invertible or nilpotent.
Then nilpotent elements form a maximal double-sided ideal, which is
the unique maximal ideal of $\End(M)$. Let $\xi\in\End(M)$. For
each $i\in Q_{0}$, we decompose the $\KK$-vector space $M_{i}$
into a direct sum of characteristic subspaces of the $\KK$-linear
map $\xi_{i}$. Given a factorization of the characteristic polynomial
$\chi_{\xi_{i}}(T)$ of $\xi_{i}$ into irreducible factors\[
\chi_{\xi_i}(T)=\prod_PP(T)^{n_{i,P}},
\]where the product runs over the set of irreducible monic polynomials
in $\KK[T]$ and $n_{i,P}=0$ for all but finitely many $P$, we have
a decomposition \[
M_i=\bigoplus_PM_i^P:=\bigoplus_P\Ker(P(\xi_i)^{n_{i,P}}).
\] Since $\xi$ commutes with the action of $t\in\mathcal{O}_{\alpha}$
and of all arrows of $Q$, its characteristic subspaces $M_{i}^{P},\ i\in Q_{0}$
($P$ an irreducible polynomial over $\mathbb{F}_{q}$) are $\mathcal{O}_{\alpha}$-submodules,
which are preserved by the action of arrows. Therefore, the submodules
$M_{i}^{P},\ i\in Q_{0}$ form a subrepresentation $M^{P}\subset M$
of $Q$ over $\mathcal{O}_{\alpha}$. Moreover, since $M_{i}=\bigoplus_{P}M_{i}^{P}$,
the subrepresentations $M^{P}$ are locally free. Finally, since $M$
is indecomposable, $M=M^{P}$ for some $P$. Thus $\xi$ is either
invertible ($P\ne T$) or nilpotent ($P=T$). \end{proof}

Now, the following well-known properties follow from the general framework
of \cite{Moz19}. Given $M\in\Rep_{\mathcal{O}_{\alpha}}(Q)^{\mathrm{l.f.}}$,
its ring of endomorphisms $\End(M)$ is a finite-dimensional algebra.
We call $\mathrm{Rad}(M)\subseteq\End(M)$ its Jacobson radical and
$\mathrm{top}(\End(M)):=\End(M)/\mathrm{Rad}(M)$.

\begin{prop} \label{Prop/EndRings}

Let $M\in\Rep_{\mathcal{O}_{\alpha}}(Q)^{\mathrm{l.f.}}$ and $\rr=\mathbf{rk}(M)$.
\begin{enumerate}
\item If $M$ is indecomposable, then $\mathrm{Rad}(M)$ consists of nilpotent
elements.
\item $M$ is absolutely indecomposable if, and only if, $\mathrm{top}(\End(M))=\KK$.
\item If $M$ is indecomposable and $\rr$ is indivisible, then $M$ is
absolutely indecomposable.
\end{enumerate}
\end{prop}

\begin{proof}

1. follows from the proof of Proposition \ref{Prop/Krull-Schmidt},
since $\mathrm{Rad}(M)$ is the maximal ideal of $\End(M)$. 2. is
\cite[Lemma 3.8.2.]{Moz19}.

3. follows from \cite[Thm. 3.9.1.]{Moz19}. Indeed, if $M$ is indecomposable,
then $\mathrm{top}(\End(M))=\KK'$ is a finite field extension of
$\KK$ and $M_{\KK'}$ splits as a direct sum of $[\KK':\KK]$ Galois
twists of some absolutely indecomposable representation over $\mathcal{O}_{\alpha}\otimes_{\KK}\KK'$,
hence $[\KK':\KK]\vert\rr$. If $\rr$ is indivisible, this implies
that $\mathrm{top}(\End(M))=\KK$, hence $M$ is absolutely indecomposable
by 2. \end{proof}

\subsection{Plethystic notation \protect\label{Sect/Plethysm}}

In this section, we introduce the $\lambda$-ring formalism required
to formulate Theorem \ref{Thm/IntroExpFmlKacPol}, again following
\cite{Moz19}. In particular, we define the plethystic exponential
and recall a well-known formula for counting objects with endomorphisms
in $\Rep_{\mathcal{O}_{\alpha}}(Q)^{\mathrm{l.f.}}$. We assume the
ground field to be $\KK=\FF_{q}$.

Let us first define $\mathcal{V}$, the $\lambda$-ring of counting
functions (see \cite[\S 2.1-2.]{Moz19}). As a ring, $\mathcal{V}=\prod_{n\geq1}\mathbb{Q}$.
Here are a few examples of counting functions. Given a quiver $Q$
and a rank vector $\rr$, we define $A_{Q,\rr,\alpha}:=\left(A_{Q,\rr,\alpha}(q^{n})\right)_{n\geq1}\in\mathcal{V}$,
where $A_{Q,\rr,\alpha}(q^{n})$ is the number of isomorphism classes
of absolutely indecomposable, locally free representations of $Q$
over $\FF_{q^{n}}[t]/(t^{\alpha})$, in rank $\rr$. If $X$ is an
algebraic variety over $\FF_{q}$, then $\#X:=\left(\#X(\FF_{q^{n}})\right)_{n\geq1}\in\mathcal{V}$
is also a counting function. A last, easy example is given by rational
functions: if $R\in\QQ(T)$ does not vanish on powers of $q$, then
we obtain a counting function $R(q):=\left(R(q^{n})\right)_{n\geq1}\in\mathcal{V}$.

The $\lambda$-ring structure on $\mathcal{V}$ is given by the Adams
operators $\psi_{m}:\mathcal{V}\rightarrow\mathcal{V},\ m\geq1$,
defined by: \[
\psi_m(a)=(a_{mn})_{n\geq1},\text{ where } a=(a_n)_{n\geq1}\in\mathcal{V}.
\]

Our main counting results are expressed in terms of power series.
For instance, we consider the following power series, which encodes
all counting functions $A_{Q,\rr,\alpha}$ at once: \[
\sum_{\rr\in\NN^{Q_0}\setminus\{0\}}A_{Q,\rr,\alpha}\cdot t^{\rr}\in\mathcal{V}[[t_i,\ i\in Q_0]].
\] The power series ring $\mathcal{V}[[t_{i},\ i\in Q_{0}]]$ can also
be endowed with a $\lambda$-ring structure, given by the Adams operators:
\[
\psi_m\left(\sum_{\rr\in\NN^{Q_0}}f_{\rr}\cdot t^{\rr}\right):=\sum_{\rr\in\NN^{Q_0}}\psi_m(f_{\rr})\cdot t^{m\rr}.
\] Let $\mathcal{V}[[t_{i},\ i\in Q_{0}]]_{+}=\{f_{0}=0\}\subset\mathcal{V}[[t_{i},\ i\in Q_{0}]]$
be the augmentation ideal. The plethystic exponential is the operator
$\Exp_{q,t}:\mathcal{V}[[t_{i},\ i\in Q_{0}]]_{+}\rightarrow\mathcal{V}[[t_{i},\ i\in Q_{0}]]$
defined by:\[
\Exp_{q,t}(F)=\exp\left(\sum_{m\geq1}\frac{\psi_m(F)}{m}\right).
\] 

The following result is a direct application of \cite[Thm. 4.6.]{Moz19},
on the weighted volume of the stack of objects with endomorphisms
(associated to a linear stack).

\begin{prop} \label{Prop/ExpFml}

Let $Q$ be a quiver and $\alpha\geq1$. Let us define:\[
\vol_{\End}(\Rep_{\mathcal{O}_{\alpha}}(Q,\rr)^{\mathrm{l.f.}})
=
\left(
\sum_{[M]}
\frac{\#\End(M)}{\#\Aut(M)}
\right)_{n\geq1}\in\mathcal{V},
\]where the sum runs over isomorphism classes of $\Rep_{\mathcal{O}_{\alpha}\otimes\FF_{q^{n}}}(Q,\rr)^{\mathrm{l.f.}}$.
Then:\[
\sum_{\rr\in\NN^{Q_0}}
\vol_{\End}(\Rep_{\mathcal{O}_{\alpha}}(Q,\rr)^{\mathrm{l.f.}})\cdot t^{\rr}
=
\Exp_{q,t}\left(\sum_{\rr\in\NN^{Q_0}\setminus\{0\}}\frac{A_{Q,\rr,\alpha}}{1-q^{-1}}\cdot t^{\rr}\right).
\]

\end{prop}

\subsection{Graph-theoretic notations \protect\label{Sect/Graph}}

In this section, we set notations for the operations on quivers which
appear in the proofs of our main counting results. These include restricting
to subquivers, contracting and deleting arrows. We also recall basic
notions on Betti numbers of graphs and spanning trees.

Let $Q$ be a quiver. We define subquivers of $Q$ obtained by restricting
to a subset of vertices or arrows of $Q$.

\begin{df}

Let $I\subseteq Q_{0}$ be a subset of vertices and $J\subseteq Q_{1}$
be a subset of arrows of $Q$.
\begin{enumerate}
\item $Q\vert_{I}$ is the quiver with the set of vertices $I$ and the
set of arrows $Q_{1,I}:=\{a\in Q_{1}\ \vert\ s(a),t(a)\in I\}$. The
source and target maps of $Q\vert_{I}$ are obtained from those of
$Q$ by restriction.
\item $Q\vert_{J}$ is the quiver with the set of vertices $Q_{0}$ and
the set of arrows $J$. The source and target maps of $Q\vert_{J}$
are obtained from those of $Q$ by restriction.
\end{enumerate}
\end{df}

We also define quivers obtained from $Q$ by contracting or deleting
a subset of arrows.

\begin{df}

Let $E\subseteq Q_{1}$ be a subset of arrows of $Q$. Consider the
equivalence relation $\sim_{E}$ on $Q_{0}$ generated by the relations
$s(a)\sim_{a}t(a)$, for $a\in E$. Given a vertex $i\in Q_{0}$,
we denote by $[i]$ the equivalence class of $i$ under $\sim_{E}$.
\begin{enumerate}
\item $Q/E$ is the quiver obtained from $Q$ by contracting arrows in $E$,
i.e. its set of vertices is $Q_{0}/\sim_{E}$ and its set of arrows
is $(Q/E)_{1}:=Q_{1}\setminus\{E\}$. The source and target maps are
obtained from those of $Q$ by composing with $Q_{0}\rightarrow Q_{0}/\sim_{E}$.
\item $Q\setminus E$ is the quiver obtained from $Q$ by deleting arrows
in $E$, i.e. its set of vertices is $Q_{0}$ and its set of arrows
is $(Q\setminus E)_{1}:=Q_{1}\setminus E$. The source and target
maps are obtained from those of $Q$ by restriction.
\end{enumerate}
Likewise, given $\lambda\in\ZZ^{Q_{0}}$, we define:
\begin{enumerate}
\item $\lambda/E\in\ZZ^{(Q/a)_{0}}$ is given by $(\lambda/E)_{[i]}=\sum_{j\sim_{E}i}\lambda_{j}$,
for $i\in Q_{0}$;
\item $\lambda\setminus E\in\ZZ^{(Q\setminus E)_{0}}=\ZZ^{Q_{0}}$ is simply
given by $\lambda$.
\end{enumerate}
\end{df}

We will also need a basic fact on the Betti number of a graph. Given
a graph $\Gamma$ (not necessarily oriented) with $C$ connected components,
$V$ vertices and $E$ edges, the Betti number of $\Gamma$ is $b(\Gamma):=C-V+E$.
This should be thought as the number of independent cycles of $\Gamma$.
Note that $C=\dim\mathrm{H}_{0}(\Gamma,\QQ)$ and $b(\Gamma)=\dim\mathrm{H}_{1}(\Gamma,\QQ)\geq0$.
Given a quiver $Q$, we define its Betti number $b(Q)$ as the Betti
number of the underlying graph. We also denote by $c(Q)$ the number
of connected components of $Q$. Recall that a quiver $Q$ (or its
underlying graph $\Gamma$) is said to be 2-connected if $\Gamma$
is connected and removing one edge from $\Gamma$ does not disconnect
it.

\begin{prop} \label{Prop/BettiNb}

Let $Q$ be a quiver.
\begin{enumerate}
\item If $Q$ is 2-connected, then $b(Q)>0$.
\item For any subset $E\subset Q_{1}$, we have $b(Q)=b(Q/E)+b(Q\vert_{E})$.
In particular, we have $b(Q)\geq b(Q\vert_{E})$ and the inequality
is strict if $Q$ is 2-connected and $E\subsetneq Q_{1}$.
\item Given $Q_{0}=I_{1}\sqcup\ldots\sqcup I_{s}$ a partition of the set
of vertices and $Q'$ the quiver obtained from $Q$ by contracting
the arrows of subquivers $Q\vert_{I_{1}},\ldots,Q\vert_{I_{s}}$,
we have $b(Q)=b(Q')+b(Q\vert_{I_{1}})+\ldots+b(Q\vert_{I_{s}})$.
\item $Q$ is 2-connected if, and only if, for all partitions $Q_{0}=I_{1}\sqcup\ldots\sqcup I_{s}$,
$b(Q)>b(Q\vert_{I_{1}})+\ldots+b(Q\vert_{I_{s}})$.
\end{enumerate}
\end{prop}

\begin{proof}

1. By contraposition (and leaving out the trivial case where $Q$
is not connected), suppose that $b(Q)=0$ and $Q$ is connected. Then
$\#Q_{1}=\#Q_{0}-1$, which is the minimum number of arrows required
to get a connected quiver with $\#Q_{0}$ vertices. Hence removing
one arrow disconnects $Q$, i.e. it is not 2-connected.

2. The case where $E=\emptyset$ (resp. $E=Q_{1}$) being trivial,
let us assume that $\emptyset\subsetneq E\subsetneq Q_{1}$. We reason
by induction on $\#Q_{1}$. Consider $a\in E$ and let $E'=E\setminus\{a\}$.
On the one hand, if $a$ is a loop, then $b(Q/\{a\})=b(Q)-1$ and
$b(Q\vert_{\{a\}})=1$; on the other hand, if $a$ is not a loop,
then $b(Q/\{a\})=b(Q)$ and $b(Q\vert_{\{a\}})=0$. In all cases,
we obtain that $b(Q)=b(Q/\{a\})+b(Q\vert_{\{a\}})$. Moreover, by
induction, we have $b(Q/\{a\})=b((Q/\{a\})/E')+b((Q/\{a\})\vert_{E'})$.
Since $(Q/\{a\})/E'$ (resp. $(Q/\{a\})\vert_{E'}$) is isomorphic
to $Q/E$ (resp. $(Q\vert_{E})/\{a\}$), we obtain\[
b(Q)=b(Q/E)+b((Q\vert_{E})/\{a\})+b((Q\vert_{E})\vert_{\{a\}})=b(Q/E)+b(Q\vert_E),
\]by applying the above argument to $Q\vert_{E}$.

3. This follows from a straightforward computation. Suppose that $Q$
(resp. $Q\vert_{I_{k}}$) has $C$ (resp. $C_{k}$) connected components,
$V$ (resp. $V_{k}$) vertices and $A$ (resp. $A_{k}$) arrows. Then
$Q'$ has $C$ connected components, $C_{1}+\ldots+C_{s}$ vertices
and $A-A_{1}-\ldots-A_{s}$ arrows.

4. Suppose that $Q$ is 2-connected. Let $Q_{0}=I_{1}\sqcup\ldots\sqcup I_{s}$
be a partition of the set of vertices and $Q'$ be the quiver obtained
from $Q$ by contracting the arrows of subquivers $Q\vert_{I_{1}},\ldots,Q\vert_{I_{s}}$.
Then $Q'$ is also 2-connected and the required inequality follows
by 1. and 3. On the other hand, if $Q$ is not 2-connected, consider
an arrow $a\in Q_{1}$ which disconnects $Q$ when removed. Then $Q\setminus\{a\}$
splits into two connected components with sets of vertices $I_{1},I_{2}$
and one can check that $b(Q)=b(Q\vert_{I_{1}})+b(Q\vert_{I_{2}})$.
\end{proof}

Finally, we will use labeled trees to index strata of the stack $\left[\mu_{Q,\rr,\alpha}^{-1}(0)/\GL(\rr,\mathcal{O}_{\alpha})\right]$,
in order to compute its cohomology (when $\rr=\mathbf{1}$). A spanning
tree of a connected quiver $Q$ is a connected subquiver $Q\vert_{J}$,
where $J\subseteq Q_{1}$ has cardinality $\#Q_{0}-1$. In other words,
$b(Q\vert_{J})=0$, which means that $Q\vert_{J}$ has no cycles,
i.e. it is a tree.

\begin{df}

A valued spanning tree $T$ of $Q$ is the datum of a spanning tree
$Q\vert_{J}$ ($J\subseteq Q_{1}$), along with a labeling $v:J\rightarrow\mathbb{Z}$.

\end{df}

The above labeling will be used to keep track of the integers $\ord(x_{a}):=\max\{k\in\ZZ_{\geq0}\ \vert\ t^{k}\vert x_{a}\}$,
$a\in J$, for a quiver representation $x\in R(Q,\mathbf{1},\mathcal{O}_{\alpha})$.

\subsection{Stanley-Reisner rings and their Hilbert series \protect\label{Sect/SRrings}}

In this section, we recall some results on Stanley-Reisner rings of
simplicial complexes and their Hilbert series. We focus on shellable
and Cohen-Macaulay simplicial complexes, as their Hilbert series enjoy
positivity properties that we exploit in our proof of Theorem \ref{Thm/IntroPositivityA}.
More specifically, we state shellability results by Bj\"{o}rner on order
complexes of posets \cite{Bjo80}. We refer to \cites{BH98}{Sta07}
for more details on Hilbert series of simplicial complexes.

Let us first define simplicial complexes and their Stanley-Reisner
rings. An (abstract) simplicial complex $\Delta$ on a set of vertices
$V=\{x_{1},\ldots,x_{n}\}$ is a collection of subsets of $V$ (called
faces of $\Delta$) such that (i) for all $x\in V$, $\{x\}\in\Delta$
and (ii) for any $F\in\Delta$, we have $G\subseteq F\Rightarrow G\in\Delta$.
The dimension of a face $F\in\Delta$ is $\#F-1$ and faces of maximal
dimension are called facets. The Stanley-Reisner ring $\QQ[\Delta]$
associated to $\Delta$ is:\[
\QQ[x_1,\ldots,x_n]/I_{\Delta},
\] where $I_{\Delta}$ is the ideal generated by $x_{i_{1}}\ldots x_{i_{r}}$,
for $i_{1}<\ldots<i_{r}$ and $\{i_{1},\ldots,i_{r}\}\notin\Delta$.
Note that a $\QQ$-basis of $\QQ[\Delta]$ is given by monomials $x^{\mm}=x_{1}^{m_{1}}\ldots x_{n}^{m_{n}}$
whose support (i.e. $\supp(\mm):=\{1\leq i\leq n\ \vert\ m_{i}\ne0\}$)
is a face of $\Delta$.

Given a $\ZZ_{\geq0}$-grading of $\QQ[\Delta]$, one can consider
its Hilbert series. We will consider $\ZZ_{\geq0}$-gradings for which
the Hilbert series is a specialization of the fine Hilbert series
$\Hilb_{\Delta}\in\ZZ[[u_{1},\ldots,u_{n}]]$, associated to the natural
$\ZZ_{\geq0}^{n}$-grading of $\QQ[\Delta]$ given by $\deg(x_{1}^{m_{1}}\ldots x_{n}^{m_{n}})=(m_{1},\ldots,m_{n})$.
The fine Hilbert series is defined as:\[
\Hilb_{\Delta}(u):=\sum_{\mm\in\ZZ_{\geq0}^{n}}\dim_{\QQ}\QQ[\Delta]_{\mm}\cdot u^{\mm}
=\sum_{F\in\Delta}\prod_{x_i\in F}\frac{u_i}{1-u_i}.
\] By \cite[Ch. I, Thm. 2.3.]{Sta07}, there exists $P(u)\in\ZZ[u_{1},\ldots,u_{n}]$
such that:\[
\Hilb_{\Delta}(u)=\frac{P(u)}{\prod_{i=1}^n(1-u_i)}.
\]

We now discuss some properties of simplicial complexes which imply
that, when specialized according to any $\ZZ_{\geq0}$-grading of
$\QQ[\Delta]$, $\Hilb_{\Delta}$ admits a numerator with non-negative
coefficients. Consider the $\ZZ_{\geq0}$-grading of $\QQ[\Delta]$
given by $\deg(x_{i})=d_{i}>0$. A simplicial complex $\Delta$ is
called Cohen-Macaulay if the ring $\QQ[\Delta]$ is Cohen-Macaulay.
Note that the Cohen-Macaulay property is independent of the choice
of the grading (see \cite[Ch. 1-2]{BH98}\cite[Ch. I, \S5]{Sta07}
for background). The Cohen-Macaulay property implies that the specialized
Hilbert series can be presented as a rational function whose numerator
has non-negative coefficients:

\begin{prop} \label{Prop/CMHilb}

\emph{\cite[Ch. I, \S 5.2. - Thm. 5.10.]{Sta07}} Let $\Delta$ be
a simplicial complex with set of vertices $V=\{x_{1},\ldots,x_{n}\}$
and fix a $\ZZ_{\geq0}$-grading of $\QQ[\Delta]$ given by $\deg(x_{i})=d_{i}>0$.
If $\Delta$ is Cohen-Macaulay, then there exist $Q(q)\in\ZZ_{\geq0}[q]$
and $e_{1},\ldots,e_{s}\geq1$ such that:\[
\Hilb_{\Delta}(q^{d_1},\ldots,q^{d_n})=\frac{Q(q)}{\prod_{i=1}^s(1-q^{e_i})}.
\]

\end{prop}

\begin{rmk}

Note that, in Proposition \ref{Prop/CMHilb}, $Q(q)$ may not be a
specialization of $P(u)$. Indeed, such a presentation of $\Hilb_{\Delta}$
as a rational fraction depends on the choice of some elements in $\QQ[\Delta]$.
In the first presentation (with numerator $P(u)$), we choose generators
$x_{i},\ 1\leq i\leq n$. In the second presentation (with numerator
$Q(q)$), we consider a regular sequence of homogeneous elements (of
degrees $e_{i},\ 1\leq i\leq s$) in $\QQ[\Delta]$. It may be the
case that generators $x_{i},\ 1\leq i\leq n$ do not form a regular
sequence (the depth of $\QQ[\Delta]$ may even be smaller than $n$).
See \cite[Ch. I, \S5]{Sta07} for the relation between $Q(q)$ and
the choice of a regular sequence.

\end{rmk}

While the Cohen-Macaulay condition has an abstract algebraic formulation
in terms of $\QQ[\Delta]$, there is a combinatorial condition on
$\Delta$ itself which implies that $\Delta$ is Cohen-Macaulay. A
simplicial complex $\Delta$ is called shellable if:
\begin{enumerate}
\item $\Delta$ is pure, i.e. its facets all have the same dimension $d$;
\item Facets of $\Delta$ can be ordered (let us call them $F_{1},\ldots,F_{s}$)
in such a way that for $2\leq i\leq s$, $F_{i}\cap\left(\bigcup_{j=1}^{i-1}F_{j}\right)$
is a nonempty union of faces of dimension $d-1$.
\end{enumerate}
The sequence of facets $F_{1},\ldots,F_{s}$ is called a shelling
of $\Delta$.

\begin{prop} \label{Prop/ShellCM}

\emph{\cite[Ch. III, Thm. 2.5.]{Sta07}} If $\Delta$ is shellable,
then it is Cohen-Macaulay.

\end{prop}

There is a particular class of simplicial complexes with well-studied
shellings: order complexes. Given a poset $\Pi$, the associated order
complex $\Delta(\Pi)$ has vertex set $\Pi$ and its faces are chains
$\{x_{1}<x_{2}<\ldots<x_{r}\}\subseteq\Pi$. In \cite{Bjo80}, Bj\"{o}rner
studies shellings of order complexes through the stronger notion of
lexicographic shellability. We recall some of his results on order
complexes of lattices.

Recall that a lattice is a poset where any pair $x,y\in\Pi$ admits
a greatest lower bound $x\wedge y$ and a least upper bound $x\vee y$
(see \cite[Ch. I, \S 4]{Bir73} for background). A lattice is called
modular if, for all $x,y,z\in\Pi$ such that $x\leq z$, $x\vee(y\wedge z)=(x\vee y)\wedge z$.
A lattice is called graded if there exists a strictly increasing function
$\rho:\Pi\rightarrow\ZZ_{\geq0}$ such that for $x,y\in\Pi$, $y$
covers\footnote{In a poset $\Pi$, $y$ is said to cover $x$ if $x\leq y$ and there
exist no $z\in\Pi$ such that $x<z<y$.} $x\ \Rightarrow\ \rho(y)=\rho(x)+1$. A modular lattice is graded
(see \cite[Ch II, \S 8]{Bir73}).

\begin{exmp} \label{Exmp/OrderCpx}

Given a finite set $S$, let us call $\Pi(S)$ the poset of subsets
of $S$, ordered by inclusion. $\Pi(S)$ is a modular lattice. Indeed,
for $A,B\in\Pi(S)$, $A\wedge B=A\cap B$ and $A\vee B=A\cup B$;
moreover, if $A\subseteq C$, then $A\cup(B\cap C)=(A\cup B)\cap C$.
Note that $\Pi(S)$ is graded by cardinality $\rho:A\mapsto\#A$.

\end{exmp}

We collect the following results from \cite{Bjo80}:

\begin{prop} \label{Prop/LattShell}

Let $L$ be a finite lattice. If $L$ is graded by $\rho$, define,
for $S\subset\ZZ_{\geq0}$, $L_{S}:=\{x\in L\ \vert\ \rho(x)\in S\}$.
Then:
\begin{enumerate}
\item \emph{\cite[Thm. 3.1.]{Bjo80}} If $L$ is modular, then $\Delta(L)$
is shellable.
\item \emph{\cite[Thm. 4.1.]{Bjo80}} If $L$ is graded and $\Delta(L)$
is shellable, then $\Delta(L_{S})$ is shellable for any $S\subset\ZZ_{\geq0}$.
\end{enumerate}
\end{prop}

\begin{exmp} \label{Exmp/OrderCpxShell}

Given a finite set $S$, Proposition \ref{Prop/LattShell} implies
that $\Delta(\Pi(S))$ and $\Delta(\Pi(S)\setminus\{S\})$ are shellable,
hence Cohen-Macaulay.

\end{exmp}

\subsection{Equivariant cohomology and mixed Hodge structures \protect\label{Sect/EqCoh}}

In this section, we recall elementary computational facts on equivariant
cohomology of complex algebraic varieties. We also recall the connection
between Hodge theory of algebraic varieties and point-counts over
finite fields, as discussed, for instance, in \cite[Appendix]{HRV08}.

\paragraph*{Mixed Hodge structures on equivariant cohomology}

A mixed Hodge structure on a $\QQ$-vector space $V$ is the datum
of an increasing filtration $W_{\bullet}\subseteq V$ (the weight
filtration) and a decreasing filtration $F^{\bullet}\subseteq V_{\CC}$
(the Hodge filtration) satisfying certain compatibility conditions
(see \cite[Def. 2.1-3.1]{PS08}). In particular, for any $n\in\ZZ$,
there is a splitting\[
\mathrm{gr}_{n}^{W}(V)_{\CC}=\bigoplus_{p+q=n}V^{p,q},
\] such that $F^{p}\mathrm{gr}_{n}^{W}(V)_{\CC}=\bigoplus_{p'\geq p}V^{p',n-p'}$.
By definition, the Hodge numbers of $V$ are\[
h^{p,q}:=\dim_{\CC}V^{p,q}.
\] A mixed Hodge structure is called pure of weight $n$ if $\mathrm{gr}_{n'}^{W}(V)_{\CC}=0$
for $n'\ne n$. A polarization on a pure Hodge structure $V$ is a
$\QQ$-valued bilinear form satisfying certain positivity and orthogonality
conditions with respect to the Hodge filtration (see \cite[Def. 2.9]{PS08}).
A mixed Hodge structure is called polarizable if, for all $n\in\ZZ$,
the pure Hodge structure $\mathrm{gr}_{n}^{W}(V)$ admits a polarization.
Given two vector spaces $V_{1}$, $V_{2}$ endowed with polarizable
mixed Hodge structures, the tensor product $V_{1}\otimes V_{2}$ also
carries a polarizable mixed Hodge structure.

In this paper, we always work with $\ZZ$-graded polarizable mixed
Hodge structures (in all examples below, the grading will keep track
of cohomological degree). We sometimes omit the adjective ``polarizable''
for the sake of brevity. A $\ZZ$-graded mixed Hodge structure $V^{\bullet}$
is called pure of weight $n$ if $V^{k}$ is pure of weight $n+k$
for all $k\in\ZZ$. Given two $\ZZ$-graded mixed Hodge structures
$V_{1}^{\bullet}$, $V_{2}^{\bullet}$, we define the tensor product
$(V_{1}\otimes V_{2})^{\bullet}$ as follows:\[
(V_{1}\otimes V_{2})^k=\bigoplus_{k_1+k_2=k}V_1^{k_1}\otimes V_2^{k_2}.
\]We let $\QQ(n)$ be the pure Hodge structure of dimension 1 and type
$(-n,-n)$, concentrated in degree 0. We also denote by $\mathbb{L}$
the graded mixed Hodge structure on $\HH_{\mathrm{c}}^{\bullet}(\mathbb{A}^{1},\QQ)$
(see \cite[Ch. 4.]{PS08}). It is concentrated in degree $2$ and
$\HH_{\mathrm{c}}^{2}(\mathbb{A}^{1},\QQ)$ is pure of type $(1,1)$.
In other words, $\mathbb{L}\simeq\QQ(-1)[-2]$. Following \cite{Dav23c},
we say that a graded mixed Hodge structure $V^{\bullet}$ is of Tate
type if there exist $a_{m,n}\in\ZZ_{\geq0},\ m,n\in\ZZ$ such that
$V^{\bullet}=\bigoplus_{m,n}\left(\mathbb{L}^{\otimes n}[m]\right)^{\oplus a_{m,n}}$.
Given a polynomial $P=\sum_{i}a_{i}T^{i}\in\ZZ_{\geq0}[T]$, we define
$P(\mathbb{L}):=\bigoplus_{i}(\mathbb{L}^{\otimes i})^{\oplus a_{i}}$.

A natural source of graded mixed Hodge structures is the (equivariant,
compactly supported) cohomology of algebraic varieties over $\CC$.
Given an algebraic variety $X$ acted on by a linear algebraic group
$G$, the compactly supported cohomology (with rational coefficients)
of the quotient stack $[X/G]$ can be defined and endowed with a polarizable
mixed Hodge structure, using approximations of Borel's construction
$(X(\CC)\times E_{G})/G(\CC)$ by algebraic varieties \cite{Tot99,EG98a}
and Deligne's mixed Hodge theory \cite{Del74a}. Alternatively, the
cohomology of $(X(\CC)\times E_{G})/G(\CC)$ can be presented as the
cohomology of a complex simplicial scheme, which carries a mixed Hodge
structure by \cite{Del74a}. These constructions can be recovered
using Saito's theory of mixed Hodge modules \cite{Sai88,Sai90}, which
was extended to quotient stacks by Achar \cite{Ach13} and more recently
to algebraic stacks by Tubach \cite{Tub25b}. We will refer to the
latter works for certain well-known facts, for the sake of conciseness.
See \cite[Rmk. 2.1.2, \S4.5]{Tub25b} for a comparison of the above
constructions.

We will use below the following definition of compactly supported
cohomology. Let $p:[X/G]\rightarrow\mathrm{pt}$ be the structure
morphism and $\underline{\QQ}_{[X/G]}:=p^{*}\QQ(0)[0]\in\mathrm{D}_{\mathrm{H}}([X/G])$,
where $\mathrm{D}_{\mathrm{H}}([X/G])$ is the (infinity-categorical)
derived category of mixed Hodge modules on $[X/G]$ introduced by
Tubach \cite{Tub25b}, based on Saito's derived category $\mathrm{D}^{\mathrm{b}}(\mathrm{MHM}(-))$
for algebraic varieties. Then we have the following $\ZZ$-graded
polarizable mixed Hodge structure:\[
\HH_{\mathrm{c}}^\bullet([X/G]):=\HH^\bullet(p_!\underline{\QQ}_{[X/G]}).
\]

\begin{exmp}

A recurring example of equivariant cohomology groups appearing in
this paper is $\HH_{\mathrm{c}}^{\bullet}(\mathrm{B}\mathbb{G}_{\mathrm{m}})$.
By Borel approximation (see for instance \cite[Thm. 9.1.1]{Del74a}),
one can compute:\[
\HH_\mathrm{c}^{\bullet}(\mathrm{B}\mathbb{G}_{\mathrm{m}})
\simeq
\bigoplus_{n\geq1}\mathbb{L}^{\otimes (-n)}.
\]

\end{exmp}

\paragraph*{Computational lemmas}

We now collect some lemmas on equivariant cohomology, which will prove
useful in computing $\HH_{\mathrm{c}}^{\bullet}([\mu_{Q,\rr,\alpha}^{-1}(0)/\GL(\rr,\mathcal{O}_{\alpha})])$.
In what follows, $G$ denotes a complex linear algebraic group. We
use mixed Hodge modules in some proofs - see \cite{Sai89,Tub25b}
for background. Unless stated otherwise, the functors used below are
all derived.

\begin{lem} \label{Lem/AffFib}

Suppose that $f:X\rightarrow Y$ is a $G$-equivariant (\'{e}tale-locally
trivial) affine fibration of dimension $d$. Then:\[
\HH_{\mathrm{c}}^{\bullet}([X/G])
\simeq
\mathbb{L}^{\otimes d}
\otimes
\HH_{\mathrm{c}}^{\bullet}([Y/G]).
\]

\end{lem}

\begin{proof}

Since $f$ induces an affine fibration $\overline{f}:[X/G]\rightarrow[Y/G]$,
the result can be deduced from the fact that the adjunction morphism
for smooth maps \cite[Thm. 4.1.1(7)]{Tub25b} $\overline{f}_{!}\underline{\QQ}_{[X/G]}\rightarrow\underline{\QQ}_{[Y/G]}\otimes\mathbb{L}_{[Y/G]}^{\otimes d}$
is an isomorphism. Indeed, this can be checked pointwise on $[Y/G]$,
using base change \cite[Thm. 4.1.1(4)]{Tub25b}.\end{proof}

\begin{lem} \label{Lem/DepthChg}

Let $G\twoheadrightarrow H$ be a surjective morphism of complex linear
algebraic groups with unipotent kernel $K$, of dimension $k$. Let
$X$ be a complex $H$-variety. Then:

\[
\HH_{\mathrm{c}}^{\bullet}([X/G])
\simeq
\mathbb{L}^{\otimes (-k)}
\otimes
\HH_{\mathrm{c}}^{\bullet}([X/H]).
\]

\end{lem}

\begin{proof}

We argue similarly to the proof of Lemma \ref{Lem/AffFib}. The map
$r:[X/G]\rightarrow[X/H]$ is a $K$-gerbe, so it is smooth of dimension
$-k$ (in the sense of \cite[Def. 5.4.2]{LZ24}) and we have the adjunction
morphism $r_{!}\underline{\QQ}_{[X/G]}\rightarrow\underline{\QQ}_{[X/H]}\otimes\mathbb{L}_{[X/H]}^{\otimes(-k)}$
\cite[Thm. 4.1.1(7)]{Tub25b}. We argue that this adjunction morphism
is an isomorphism. By base change \cite[Thm. 4.1.1(4)]{Tub25b}, this
can be checked pointwise on $[X/H]$, by showing that $\HH_{\mathrm{c}}^{\bullet}(\mathrm{B}K)\simeq\mathbb{L}^{\otimes(-k)}$.
Since $K$ is unipotent, it is isomorphic to an affine space of dimension
$k$ \cite[Prop. 14.32]{Mil17}. Therefore, the adjunction morphism
for the affine fibration $\mathrm{pt}\rightarrow\mathrm{B}K$ yields
$\HH_{\mathrm{c}}^{\bullet}(\mathrm{B}K)\simeq\HH_{\mathrm{c}}^{\bullet}(\mathrm{pt})\otimes\mathbb{L}^{\otimes(-k)}\simeq\mathbb{L}^{\otimes(-k)}$,
by the same argument. \end{proof}

\begin{lem} \label{Lem/GrpChg}

Let $G\subset H$ be an inclusion of complex linear algebraic groups
and $X$ a complex $G$-variety. Then:\[
\HH_{\mathrm{c}}^{\bullet}([X/G])
\simeq
\HH_{\mathrm{c}}^{\bullet}([(X\times^GH)/H])
.
\]

\end{lem}

\begin{proof}

This an immediate consequence of the isomorphism of algebraic stacks
$[X/G]\simeq[(X\times^{G}H)/H]$ and the extension of mixed Hodge
modules to algebraic stacks \cite{Tub25b}.\end{proof}

\begin{lem} \label{Lem/Kunneth}

Let $X$ (resp. $Y$) be a complex $G$-variety (resp. $H$-variety).
Then:

\[
\HH_{\mathrm{c}}^{\bullet}([(X\times Y)/(G\times H)])
\simeq
H_{\mathrm{c}}^{\bullet}([X/G])
\otimes
H_{\mathrm{c}}^{\bullet}([Y/H]).
\]

\end{lem}

\begin{proof}

Let $p_{X}:[X/G]\rightarrow\mathrm{pt}$ (resp. $p_{Y}:[Y/G]\rightarrow\mathrm{pt}$)
be the structure morphisms. The Lemma follows from the following quasi-isomorphism:\[
(p_X\times p_Y)_!\underline{\QQ}_{[(X\times Y)/(G\times H)]}
\simeq
(p_X)_!\underline{\QQ}_{[X/G]}
\otimes
(p_Y)_!\underline{\QQ}_{[Y/H]}
,
\]which is itself a consequence of proper base change and the projection
formula \cite[Thm. 4.1.1(4)]{Tub25b}.\end{proof}

\begin{lem} \label{Lem/Strat}

Let $X$ be a complex $G$-variety, $Z\subseteq X$ a $G$-invariant
closed subvariety and $U=X\setminus Z$. Then there is a long exact
sequence of polarizable mixed Hodge structures:\[
\ldots\rightarrow
\HH_{\mathrm{c}}^{j-1}([Z/G])\rightarrow
\HH_{\mathrm{c}}^j([U/G])\rightarrow
\HH_{\mathrm{c}}^j([X/G])\rightarrow
\HH_{\mathrm{c}}^j([Z/G])\rightarrow
\HH_{\mathrm{c}}^{j+1}([U/G]) \rightarrow \ldots
\] Moreover, if both $\HH_{\mathrm{c}}^{\bullet}([U/G])$ and $\HH_{\mathrm{c}}^{\bullet}([Z/G])$
are pure, then $\HH_{\mathrm{c}}^{\bullet}([X/G])$ is also pure.

\end{lem}

\begin{proof}

The long exact sequence can be derived from the analogous long exact
sequence in ordinary compactly supported cohomology, using Borel approximation.
That long exact sequence is compatible with mixed Hodge structures
(see \cite[\S 5.5.2.]{PS08}). Alternatively, one may use \cite[Thm. 4.1.1(5)]{Tub25b}.

If $\HH_{\mathrm{c}}^{\bullet}([U/G])$ and $\HH_{\mathrm{c}}^{\bullet}([Z/G])$
are pure, then for all $j\in\ZZ$, $\HH_{\mathrm{c}}^{j}([U/G])$
and $\HH_{\mathrm{c}}^{j}([Z/G])$ are pure of weight $j$. This implies
that the connecting morphisms of the long exact sequence vanish and
the short exact sequences obtained in each cohomological degree split
(see \cite[Cor. 2.12.]{PS08}). Therefore, $\HH_{\mathrm{c}}^{j}([X/G])\simeq\HH_{\mathrm{c}}^{j}([U/G])\oplus\HH_{\mathrm{c}}^{j}([Z/G])$
is pure of weight $j$, which proves the claim. \end{proof}

\paragraph*{E-series and counts of $\mathbb{F}_{q}$-points}

Finally, we recall some results relating Hodge numbers of an algebraic
variety to counts of its points over finite fields. These results
rely on a theorem by Katz \cite[Appendix, Thm. 6.1.2.]{HRV08}. Throughout,
we assume that $G$ is a complex linear algebraic group.

Let $X$ be a complex $G$-variety. The E-series of $[X/G]$ is defined
as:\[
E([X/G];x,y):=\sum_{p,q\in\ZZ}\left(\sum_{k\in\ZZ}(-1)^k\cdot h^{p,q}\left(\HH_{\mathrm{c}}^{k}([X/G])\right)\right)\cdot x^py^q\in\ZZ((x^{-1},y^{-1})).
\] Let us briefly justify that the E-series is a well-defined formal
Laurent series in $x^{-1},y^{-1}$. For a fixed ordered pair $(p,q)$,
we have that $h^{p,q}\left(\HH_{\mathrm{c}}^{k}([X/G])\right)\ne0$
only if $p+q\leq k$ (see \cite[Ch. 5, Prop. 5.54.]{PS08}). As $\HH_{\mathrm{c}}^{\bullet}([X/G])$
is supported in degree at most $\dim\left([X/G]\right)$, the coefficient
of $x^{p}y^{q}$ boils down to a finite sum. Moreover, one can check
from \cite[Thm. 8.2.4.]{Del74a} and \cite[Ch. 5, Def. 5.52.]{PS08}
that $\sum_{k}(-1)^{k}\cdot h^{p,q}\left(\HH_{\mathrm{c}}^{k}([X/G])\right)\ne0$
only if $p,q\leq\dim\left([X/G]\right)$, so we indeed obtain a Laurent
series. A similar reasoning shows that the $E$-series of a variety
$X$ is actually a polynomial in $x,y$.

\begin{exmp}

Since $\mathrm{H}_{\mathrm{c}}^{\bullet}(\mathbb{A}^{1})\simeq\QQ(-1)[-2]$,
we have $E(\mathbb{A}^{1})=xy$. Likewise, since $\mathrm{H}_{\mathrm{c}}^{\bullet}(\mathrm{B}\mathbb{G}_{\mathrm{m}})\simeq\bigoplus_{n\geq1}\mathbb{L}^{\otimes(-n)}$,
we have \[
E(\mathrm{B}\mathbb{G}_{\mathrm{m}})=\sum_{n\geq1}(xy)^{-n}=\frac{(xy)^{-1}}{1-(xy)^{-1}}=\frac{1}{xy-1}.
\]

\end{exmp}

Given a complex variety $X$, we may count points over finite fields
of some spreading-out of $X$, i.e. an $R$-scheme $\mathcal{X}$,
where $R\subseteq\CC$ is a finitely generated $\ZZ$-algebra and
$\mathcal{X}\otimes_{R}\CC\simeq X$. Such a ring $R$ admits ring
homomorphisms $R\rightarrow\FF_{q}$ for finite fields of large enough
characteristics (see for instance \cite[Ch. 1, Lem. 2.2.6.]{CLNS18}).
Following \cite[Appendix]{HRV08}, we call $X$ polynomial-count if
there exists a spreading-out $\mathcal{X}$ and a polynomial $P\in\mathbb{Q}[T]$
such that for any ring homomorphism $\varphi:R\rightarrow\FF_{q}$,
we have $\#\mathcal{X}_{\varphi}(\FF_{q^{r}})=P(q^{r})$ (for all
$r\geq1$). It follows from the discussion in \cite[Appendix]{HRV08}
that $P$ only depends on $X$. Note that, if $G$ is connected, then
it is polynomial-count \cite[Prop. 3.19]{DM91b}.

\begin{exmp}

Set $G=\GL(\rr,\mathcal{O}_{\alpha})$. The counting polynomial of
$G$ is $P_{G}(T)=\prod_{i\in Q_{0}}T^{\alpha r_{i}^{2}}(1-T^{-1})\ldots(1-T^{-r_{i}})$.
Indeed, for every $i\in Q_{0}$, we have a short exact sequence of
algebraic groups:\[
1\rightarrow\mathrm{R}_u\GL(r_i,\mathcal{O}_\alpha)\rightarrow\GL(r_i,\mathcal{O}_\alpha)\rightarrow\GL(r_i,\CC)\rightarrow 1,
\]where $\mathrm{R}_{u}\GL(r_{i},\mathcal{O}_{\alpha})$ is the unipotent
radical, which is isomorphic to $\mathbb{A}^{(\alpha-1)r_{i}^{2}}$
as an algebraic variety. Moreover, it is well-known that, for $r\geq1$,
we have $P_{\GL(\rr,\CC)}=T^{r^{2}}(1-T^{-1})\ldots(1-T^{-r})$ (this
can be simply seen by counting $r$-tuples of linearly independent
vectors in $\FF_{q}^{\oplus r}$). Using these two facts, we obtain
the required value of $P_{G}$.

\end{exmp}

The following proposition is a straightforward generalization of \cite[Thm. 6.1.2.]{HRV08}.
See also \cite[\S 2.5.]{SV20} and \cite[\S 2.2]{LRV23} for analogous
results in Borel-Moore homology and l-adic cohomology.

\begin{prop} \label{Prop/E-seriesVSCountF_q}

Let $G$ be a connected complex algebraic group and $X$ a complex
$G$-variety. Then we have:\[
E([X/G];x,y)=\frac{E(X;x,y)}{E(G;x,y)}.
\] In particular, if $X$ is polynomial-count, with counting polynomial
$P_{X}$, then:\[
E([X/G];x,y)=\frac{P_X(xy)}{P_G(xy)}.
\]

\end{prop}

\begin{proof}

Let us prove the first equality. The second equality is then a direct
consequence of \cite[Thm. 6.1.2.]{HRV08}. Consider the morphism of
quotient stacks $q:[X/G]\rightarrow\mathrm{B}G$ and the structure
morphism $p:\mathrm{B}G\rightarrow\mathrm{pt}$. Then the graded mixed
Hodge structure $\HH_{\mathrm{c}}^{\bullet}([X/G])$ is computed by
the following spectral sequence (see \cite[Ch. II.4]{Ver96} for a
formulation of spectral sequences in the setting of triangulated categories),
starting with page 2:\[
\HH^i(p_!\mathcal{H}^jq_!\underline{\QQ}_{[X/G]})
\Rightarrow
\HH_{\mathrm{c}}^{i+j}([X/G])
,
\]where $\HH^{\bullet}:\mathrm{D}^{\mathrm{b}}(\mathrm{MHM}(\mathrm{pt}))\rightarrow\mathrm{MHM}(\mathrm{pt})$
and $\mathcal{H}^{\bullet}:\mathrm{D}_{\mathrm{H}}(\mathrm{B}G)\rightarrow\mathrm{MHM}(\mathrm{B}G)$
are the cohomological functors associated to the t-structures with
respective hearts $\mathrm{MHM}(\mathrm{pt})$ and $\mathrm{MHM}(\mathrm{B}G)$
(see \cite{Ach13}). Since $G$ is connected, we obtain that $\mathcal{H}^{j}q_{!}\underline{\QQ}_{[X/G]}$
is constant for all $j\in\ZZ$ \cite[Thm. 5.2]{Ach13}; in other words,
we have $\mathcal{H}^{j}q_{!}\underline{\QQ}_{[X/G]}\simeq\HH_{\mathrm{c}}^{j}(X)$,
where the right-hand side is seen as a mixed Hodge structure with
trivial $G$-equivariant structure (see \cite[\S 5]{Ach13}). Therefore,
the above spectral sequence boils down to:\[
\HH_\mathrm{c}^i(\mathrm{B}G)\otimes\HH_\mathrm{c}^j(X)
\Rightarrow
\HH_{\mathrm{c}}^{i+j}([X/G])
,
\]which yields $E([X/G];x,y)=E(X;x,y)E(\mathrm{B}G;x,y)$. Moreover,
the same argument, applied to the quotient morphism $\mathrm{pt}\rightarrow\mathrm{B}G$
and the structure morphism $\mathrm{B}G\rightarrow\mathrm{pt}$, gives
a spectral sequence:\[
\HH_\mathrm{c}^i(\mathrm{B}G)\otimes\HH_\mathrm{c}^j(G)
\Rightarrow
\HH_{\mathrm{c}}^{i+j}(\mathrm{pt})
,
\] hence $E(G;x,y)E(\mathrm{B}G;x,y)=E(\mathrm{pt};x,y)=1$, which concludes
the proof. \end{proof}

Finally, if $X$ is polynomial-count and $\HH_{\mathrm{c}}^{\bullet}([X/G])$
is pure, then it follows from Proposition \ref{Prop/E-seriesVSCountF_q}
that $\HH_{\mathrm{c}}^{\bullet}([X/G])$ is of Tate type, concentrated
in even degrees and:\[
E([X/G];x,y)=\sum_{k\in\ZZ}\dim\left(\HH_{\mathrm{c}}^{2k}([X/G])\right)\cdot (xy)^k.
\] In other words, $\HH_{\mathrm{c}}^{\bullet}([X/G])$ is determined
by its E-series. Let us give some details. Given a Laurent series
$F(T)=\sum_{i}a_{i}T^{i}\in\ZZ_{\geq0}((T^{-1}))$, define $F(\mathbb{L}):=\bigoplus_{i}\left(\mathbb{L}^{\otimes i}\right)^{\oplus a_{i}}$.
Let further:\[
F_{G}(xy):=E([\mathrm{pt}/G])=\frac{1}{P_G(xy)}
.
\] Then we have an isomorphism of $\ZZ$-graded polarizable mixed Hodge
structures:\[
\HH_{\mathrm{c}}^{\bullet}([X/G])\simeq P_{X}(\mathbb{L})\otimes F_{G}(\mathbb{L}).
\]Note that $\HH_{\mathrm{c}}^{\bullet}([\mathrm{pt}/G])$ is pure,
by \cite[Thm. 9.1.1]{Del74a}. Purity accounts for the fact that many
counting polynomials have non-negative coefficients. For instance
Theorem \ref{Thm/IntroPosToricKacPol} admits a cohomological upgrade
in the form of Theorem \ref{Thm/IntroCohIntgr}.

\section{Jet spaces of quiver moment maps and absolutely indecomposable representations
\protect\label{Sect/KacPolvsMomMap}}

In this section, we prove several results relating counts of jets
over fibers of quiver moment maps and counts of absolutely indecomposable,
locally free representations of quivers over $\FF_{q}[t]/(t^{\alpha})$.
These results are analogous to the now well-understood relation between
Kac polynomials and counts of points on fibers of quiver moment maps,
see \cite{CBVB04,Moz11a,Dav23a}. Throughout this section, $\KK=\FF_{q}$.

Let $Q$ be a quiver. Our first result concerns counts of jets over
zero-fibers of quiver moment maps. We obtain a formula in the $\lambda$-ring
$\mathcal{V}[[t_{i},\ i\in Q_{0}]]$, which directly generalizes \cite[Thm. 5.1.]{Moz11a},
with a similar proof.

\begin{thm} \label{Thm/ExpFmlKacPol}

Let $Q$ be a quiver and $\alpha\geq1$. Then: \[
\sum_{\rr\in\NN^{Q_0}}
\frac{\#\mu_{Q,\rr,\alpha}^{-1}(0)}{\#\GL(\rr,\mathcal{O}_{\alpha})}
\cdot q^{\alpha\langle\rr,\rr\rangle_Q}t^{\rr}
=
\Exp_{q,t}\left(
\sum_{\rr\in\NN^{Q_0}\setminus\{0\}}
\frac{A_{Q,\rr,\alpha}}{1-q^{-1}}\cdot t^{\rr}
\right).
\]

\end{thm}

\begin{proof}

By Proposition \ref{Prop/MomMapExactSeq}, for a given $x\in R(Q,\rr,\mathcal{O}_{\alpha})$
corresponding to a locally free representation $M$, we have $\#\{y\in R(Q^{\mathrm{op}},\rr,\mathcal{O}_{\alpha}))\ \vert\ \mu_{Q,\rr,\alpha}(x,y)=0\}=\#\Ext^{1}(M,M)$.
Summing over all isomorphism classes $[M]$ of locally free representations
in rank $\rr$, we obtain:\[
\frac{\#\mu_{Q,\rr,\alpha}^{-1}(0)}{\#\GL(\rr,\mathcal{O}_{\alpha})}
=
\sum_{[M]}\frac{\#\Ext^{1}(M,M)}{\#\Aut(M)}.
\] Using Proposition \ref{Prop/MomMapExactSeq} again, we obtain that
$\#\Ext^{1}(M,M)=q^{-\alpha\langle\rr,\rr\rangle_{Q}}\cdot\#\End(M)$
and so:\[
q^{\alpha\langle\rr,\rr\rangle_Q}\cdot\frac{\#\mu_{Q,\rr,\alpha}^{-1}(0)}{\#\GL(\rr,\mathcal{O}_{\alpha})}
=
\sum_{[M]}\frac{\#\End(M)}{\#\Aut(M)}=\vol_{\End}(\Rep_{\mathcal{O}_{\alpha}}(Q,\rr)^{\mathrm{l.f.}}).
\] Thus the formula follows from Proposition \ref{Prop/ExpFml}. \end{proof}

\begin{rmk}\label{Rmk/plethysticFmlFrobeniusAlg}

Theorem \ref{Thm/ExpFmlKacPol} can be generalized to locally free
representations of quivers over a \emph{split} local (finite-dimensional)
Frobenius $\KK$-algebra $R$, i.e. a local Frobenius algebra whose
residue field is isomorphic to $\KK$. Indeed, for any finite-dimensional
$\KK$-algebra $R$, one can define an algebraic group $\GL(\rr,R)$
and a $\GL(\rr,R)$-variety $R(\overline{Q},\rr,R)$ over $\KK$,
as in Remark \ref{Rmk/JetSpaces}. The construction of the moment
map $\mu_{Q,\rr,R}$ and Proposition \ref{Prop/MomMapExactSeq} hold
over any Frobenius $\KK$-algebra. Finally, the plethystic formula
for counts of absolutely indecomposable representations (Proposition
\ref{Prop/ExpFml}) is also valid if one replaces $\mathcal{O}_{\alpha}$
by an arbitrary split local Frobenius $\KK$-algebra $R$. We require
that $R$ be local, so that $R$-modules have a well-defined rank.
The fact that $R$ is split guarantees that, for any finite extension
$\KK'/\KK$, the Frobenius $\KK'$-algebra $R\otimes_{\KK}\KK'$ is
also local. Let $d:=\dim_{\KK}(R)$. The above plethystic formula
then generalizes to:\[
\sum_{\rr\in\NN^{Q_0}}
\frac{\#\mu_{Q,\rr,R}^{-1}(0)}{\#\GL(\rr,R)}
\cdot q^{d\langle\rr,\rr\rangle_Q}t^{\rr}
=
\Exp_{q,t}\left(
\sum_{\rr\in\NN^{Q_0}\setminus\{0\}}
\frac{A_{Q,\rr,R}}{1-q^{-1}}\cdot t^{\rr}
\right).
\]

\end{rmk}

As a consequence of Theorem \ref{Thm/ExpFmlKacPol}, we prove that
the asymptotic count of jets over $\mu_{Q,\rr}^{-1}(0)$ is related
to the asymptotic count of absolutely indecomposable, locally-free
representations in rank $\rr$, when $\rr=\mathbf{1}$. This solves
a conjecture of Wyss \cite[Conj. 4.37.]{Wys17b}. Recall the definition
of these asymptotic counts (limits in $\mathcal{V}$ are computed
coordinate-wise):\[
A_{Q}:=
\underset{\alpha\rightarrow +\infty}{\lim}
\left(q^{-\alpha b(Q)}\cdot A_{Q,\mathbf{1},\alpha}\right),
\]
\[
B_{\mu_{Q}}:=
\underset{\alpha\rightarrow +\infty}{\lim}
\left(q^{-\alpha(2\# Q_1-\# Q_0+1)}\cdot\#\mu_{Q,\mathbf{1},\alpha}^{-1}(0)\right).
\] Then we show:

\begin{cor} \label{Cor/RelAvsB}

Let $Q$ be a 2-connected quiver. Then: \[
\frac{B_{\mu_{Q}}}{(1-q^{-1})^{\# Q_0}}=\frac{A_{Q}}{1-q^{-1}}.
\]

\end{cor}

\begin{proof}

Let us spell out Theorem \ref{Thm/ExpFmlKacPol} for $\rr=\mathbf{1}$.
Then:\[
q^{-\alpha(b(Q)+\# Q_0-\langle\rr,\rr\rangle_Q)}
\cdot
\frac{\#\mu_{Q,\mathbf{1},\alpha}^{-1}(0)}{\prod_{i\in Q_0}(1-q^{-1})}
=
q^{-\alpha b(Q)}\cdot\sum_{Q_0=I_1\sqcup\ldots\sqcup I_s}\prod_{j=1}^s\frac{A_{Q,\mathbf{1}\vert_{I_j},\alpha}}{1-q^{-1}}.
\] Note that $b(Q)+\#Q_{0}-\langle\rr,\rr\rangle_{Q}=2\#Q_{1}-\#Q_{0}+1$.
Moreover, for $I\subseteq Q_{0}$, $A_{Q,\mathbf{1}_{I},\alpha}$
is a polynomial of degree $\alpha b(Q\vert_{I})$ (see \cite[Cor. 4.35.]{Wys17b}).
By Proposition \ref{Prop/BettiNb}, only the term corresponding to
$I_{1}=Q_{0}$ contributes to the limit in the right-hand side. Therefore
we conclude:\[
\frac{B_{\mu_{Q}}}{(1-q^{-1})^{\# Q_0}}=\frac{A_{Q}}{1-q^{-1}}.
\] \end{proof}

\begin{rmk}

Wyss' conjecture actually relates the numerators of $A_{Q}$ and $B_{\mu_{Q}}$.
However, the conjecture does not hold as stated. The following quiver
is a counter-example (regardless of orientation):
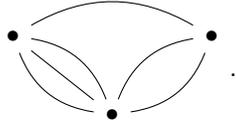
\[
\begin{tikzcd}[ampersand replacement=\&]
\bullet \ar[rr, dash, bend left] \& \& \bullet \ar[dl, dash, bend left]\ar[dl, dash, bend right] \\
\& \bullet \ar[ul, dash]\ar[ul, dash, bend left]\ar[ul, dash, bend right] \& 
\end{tikzcd}
.
\]

\end{rmk}

Finally, we generalize Crawley-Boevey and Van den Bergh's result relating
Kac polynomials and counts of $\FF_{q}$-points on the representation
variety of a \textit{deformed} preprojective algebra. Recall that
$\lambda\in\ZZ^{Q_{0}}$ is called generic with respect to a rank
vector $\rr\in\NN^{Q_{0}}$ if $\lambda\cdot\rr=0$ and $\lambda\cdot\rr'\ne0$
for all $0<\rr'<\rr$.

\begin{thm} \label{Thm/CountMomMapGenFib}

Let $\rr\in\NN^{Q_{0}}$ be an indivisible rank vector and $\lambda\in\ZZ^{Q_{0}}$
be generic with respect to $\rr$. Suppose that $\FF_{q}$ has characteristic
larger than $\sum_{i}\vert\lambda_{i}\vert\cdot r_{i}$. Then: \[
\frac{\#\mu_{Q,\rr,\alpha}^{-1}(t^{\alpha-1}\cdot\lambda)}{\#\GL(\rr,\mathcal{O}_{\alpha})}
=q^{-\alpha\langle\rr,\rr\rangle_Q}\cdot\frac{A_{Q,\rr,\alpha}}{1-q^{-1}}
,
\]where we view $\lambda$ as $(\lambda_{i}\cdot\mathrm{id})_{i\in Q_{0}}\in\mathfrak{gl}(\rr,\mathcal{O}_{\alpha})$.

\end{thm}

\begin{proof}

Consider the projection map $\pi:\mu_{Q,\rr,\alpha}^{-1}(t^{\alpha-1}\cdot\lambda)\subseteq R(\overline{Q},\rr,\mathcal{O}_{\alpha})\rightarrow R(Q,\rr,\mathcal{O}_{\alpha})$.
We claim that, when $\lambda$ is generic, the image of $\pi$ coincides
with the constructible subset of absolutely indecomposable representations.
Moreover, given $x$ an $\FF_{q}$-point of $R(Q,\rr,\mathcal{O}_{\alpha})$
corresponding to a representation $M$, then $\pi^{-1}(x)$ has cardinality
$\#\Ext^{1}(M,M)$ by Proposition \ref{Prop/MomMapExactSeq}. Therefore,
we can compute $\#\mu_{Q,\rr,\alpha}^{-1}(t^{\alpha-1}\cdot\lambda)$
by summing over isomorphism classes $[M]$ of absolutely indecomposable
representations in rank $\rr$:\[
\frac{\#\mu_{Q,\rr,\alpha}^{-1}(t^{\alpha-1}\cdot\lambda)}{\#\GL(\rr,\mathcal{O}_{\alpha})}
=
\sum_{[M]}\frac{\#\Ext^1(M,M)}{\#\Aut(M)}.
\] Moreover, for $M$ absolutely indecomposable, $\#\Aut(M)=\frac{q-1}{q}\cdot\#\End(M)=q^{\alpha\langle\rr,\rr\rangle_{Q}}\cdot\frac{q-1}{q}\cdot\#\Ext^{1}(M,M)$
by Propositions \ref{Prop/MomMapExactSeq} and \ref{Prop/EndRings}.
This yields:\[
\frac{\#\mu_{Q,\rr,\alpha}^{-1}(t^{\alpha-1}\cdot\lambda)}{\#\GL(\rr,\mathcal{O}_{\alpha})}
=q^{-\alpha\langle\rr,\rr\rangle_Q}\cdot\frac{A_{Q,\rr,\alpha}}{1-q^{-1}}.
\]

Let us now prove the claim. We actually prove the following stronger
fact, by analogy with \cite[Thm. 3.3.]{CB01}: given $x$ an $\FF_{q}$-point
of $R(Q,\rr,\mathcal{O}_{\alpha})$ corresponding to a representation
$M$, $x$ admits a lift in $\mu_{Q,\rr,\alpha}^{-1}(t^{\alpha-1}\cdot\lambda)$
if, and only if, any direct summand of $M$ of rank $\rr'\leq\rr$
satisfies $\lambda\cdot\rr'=0$.

Indeed, if $(x,y)$ is an $\FF_{q}$-point of $\mu_{Q,\rr,\alpha}^{-1}(t^{\alpha-1}\cdot\lambda)$
such that $\pi(x,y)=x$, then by Proposition \ref{Prop/MomMapExactSeq},
$\mu_{Q,\rr,\alpha}(x,y)=t^{\alpha-1}\cdot\lambda$ lies in the kernel
of $\mathfrak{gl}(\rr,\mathcal{O}_{\alpha})\rightarrow\End(M)^{\vee}$.
Therefore, pairing $\mu_{Q,\rr,\alpha}(x,y)=t^{\alpha-1}\cdot\lambda$
with the projection onto an $\rr'$-dimensional direct summand of
$M$ yields $\FF_{q}\ni\lambda\cdot\rr'=0$. This gives $\ZZ\ni\lambda\cdot\rr'=0$,
due to the bound on the characteristic of $\FF_{q}$.

Conversely, suppose that any direct summand of $M$ (call its rank
$\rr'$) satisfies $\lambda\cdot\rr'=0$. It is enough to show that,
for any indecomposable direct summand of $M$ (corresponding to an
$\FF_{q}$-point $x'$ of $R(Q,\rr',\mathcal{O}_{\alpha})$), $x'$
admits a lift in $\mu_{Q,\rr',\alpha}^{-1}(t^{\alpha-1}\cdot\lambda)$.
Thus we assume that $M$ is indecomposable. Switching to $\overline{\FF_{q}}$,
Proposition \ref{Prop/EndRings} implies that any $\xi\in\End(M)$
is the sum of a scalar $\xi_{0}\cdot\Id\in\overline{\mathbb{F}_{q}}$
and a nilpotent endomorphism, i.e. $\xi$ is the sum of $\xi_{0}\cdot\Id$
and a tuple of nilpotent matrices modulo $t$ (note that, over $\mathcal{O}_{\alpha}$,
not all nilpotent elements are traceless). Therefore, for all $\xi\in\End(M)$,
we get $\langle\xi,t^{\alpha-1}\cdot\lambda\rangle=\xi_{0}t^{\alpha-1}\lambda\cdot\rr=0$,
which implies, by Proposition \ref{Prop/MomMapExactSeq}, that $x$
lifts to $\mu_{Q,\rr,\alpha}^{-1}(t^{\alpha-1}\cdot\lambda)$. \end{proof}

\begin{rmk}

Note that the connection between moment map fibers $\mu_{Q,\rr,\alpha}^{-1}(\xi)$
and $A_{Q,\rr,\alpha}$ (for $\xi\in\mathcal{O}_{\alpha}^{Q_{0}}$,
viewed as $(\xi_{i}\cdot\mathrm{id})_{i\in Q_{0}}\in\mathfrak{gl}(\rr,\mathcal{O}_{\alpha})$)
occurs precisely when $\xi=t^{\alpha-1}\lambda$, for $\lambda\in\ZZ^{Q_{0}}$
generic. For instance, if $\xi=\lambda$ and $\lambda$ is generic,
we obtain:\[
\frac{\#\mu_{Q,\rr,\alpha}^{-1}(\lambda)}{\#\GL(\rr,\mathcal{O}_\alpha)}
=
q^{\alpha(1-2\langle\rr,\rr\rangle_Q)}\cdot\frac{A_{Q,\rr}}{q^{1-\langle\rr,\rr\rangle_Q}(1-q^{-1})}
,
\]where $A_{Q,\rr}:=A_{Q,\rr,1}$ is the usual Kac polynomial. Indeed,
by \cite[Lem. 2.1.3, Lem. 2.1.4]{CBVB04}, the quiver variety $\mu_{Q,\rr}^{-1}(\lambda)\git\GL(\rr,\FF_{q})$
is smooth, of dimension $1-2\langle\rr,\rr\rangle_{Q}$, and the quotient
map $\mu_{Q,\rr}^{-1}(\lambda)\rightarrow\mu_{Q,\rr}^{-1}(\lambda)\git\GL(\rr,\FF_{q})$
is a principal $\GL(\rr,\FF_{q})/\Delta(\mathbb{G}_{\mathrm{m}})$-bundle,
where $\Delta:\mathbb{G}_{\mathrm{m}}\hookrightarrow\GL(\rr,\FF_{q})$
is the diagonal embedding $u\mapsto(u\cdot\mathrm{id})_{i\in Q_{0}}$.
We further have the following point-count (a particular case of Theorem
\ref{Thm/CountMomMapGenFib}):\[
\frac{\#\mu_{Q,\rr}^{-1}(\lambda)}{\#\GL(\rr,\FF_q)}
=
q^{1-2\langle\rr,\rr\rangle_Q}\cdot\frac{A_{Q,\rr}}{q^{1-\langle\rr,\rr\rangle_Q}(1-q^{-1})}
.
\]

Since the construction of jet spaces is functorial \cite[Ch.3, \S 1]{CLNS18}
and compatible with \'{e}tale pullbacks \cite[Prop. 1.1]{Mus01}, we obtain
that the $(\alpha-1)$th jet space of $\mu_{Q,\rr}^{-1}(\lambda)\git\GL(\rr,\FF_{q})$
is a geometric quotient $\mu_{Q,\rr,\alpha}^{-1}(\lambda)\git\GL(\rr,\mathcal{O}_{\alpha})$
of dimension $\alpha(1-2\langle\rr,\rr\rangle_{Q})$, such that the
quotient map $\mu_{Q,\rr,\alpha}^{-1}(\lambda)\rightarrow\mu_{Q,\rr,\alpha}^{-1}(\lambda)\git\GL(\rr,\mathcal{O}_{\alpha})$
is a principal $\GL(\rr,\mathcal{O}_{\alpha})/\Delta(\mathcal{O}_{\alpha}^{\times})$-bundle.
Moreover, the truncation map\[
\mu_{Q,\rr,\alpha}^{-1}(\lambda)\git\GL(\rr,\mathcal{O}_{\alpha})
\rightarrow
\mu_{Q,\rr}^{-1}(\lambda)\git\GL(\rr,\FF_{q})
\]is an affine fibration of relative dimension $(\alpha-1)(1-2\langle\rr,\rr\rangle_{Q})$,
which yields the above point-count.

\end{rmk}

\section{Positivity and purity in the toric setting \protect\label{Section/PosPurity}}

In this section, we prove positivity results for counts of absolutely
indecomposable toric representations of quivers in higher depth. We
first prove positivity for asymptotic counts of quiver representations,
as conjectured in \cite[\S 4]{Wys17b}. We then proceed with the proof
of Theorem \ref{Thm/IntroPosToricKacPol} in fixed depth $\alpha\geq1$.
Finally, we prove a cohomological upgrade of Theorem \ref{Thm/ExpFmlKacPol}
for rank vectors $\rr\leq1$, in the spirit of \cite[Thm. A.]{Dav23c}.

\subsection{Positivity of asymptotic counts \protect\label{Sect/PosA}}

Let $Q$ be a quiver and $\KK=\FF_{q}$. In \cite[\S 4]{Wys17b},
Wyss studied the asymptotic count of locally free, absolutely indecomposable
representations of $Q$ over $\mathcal{O}_{\alpha}$ in rank $\rr=\mathbf{1}$:\[
A_Q:=
\underset{\alpha\rightarrow +\infty}{\lim}
\left(q^{-\alpha b(Q)}\cdot A_{Q,\mathbf{1},\alpha}(q)\right).
\] The above sequence converges if, and only if, $Q$ is 2-connected
and the limit is an explicit rational function in $q$ (see \cite[Cor. 4.35.]{Wys17b}):\[
A_Q=
(1-q^{-1})^{b(Q)}\cdot
\sum_{E_1\subsetneq E_2\ldots\subsetneq E_s=Q_1}\prod_{j=1}^{s-1}\frac{1}{q^{b(Q)-b(Q\vert_{E_j})}-1}.
\] The appearance of chains of subsets of $Q_{1}$ is reminiscent of
the order complex of $\Pi(Q_{1})$ and indeed, we show that $A_{Q}$
is essentially the Hilbert series of the Stanley-Reisner ring associated
to $\Delta:=\Delta(\Pi(Q_{1}))$, suitably specialized. This approach
was inspired from \cite[Prop. 1.9.]{MV24}.

\begin{thm} \label{Thm/PositivityA}

Let $Q$ be a 2-connected quiver. Consider the Stanley-Reisner ring
$\QQ[\Delta]$ associated to the order complex $\Delta$ of the poset
$\Pi(Q_{1})\setminus\{Q_{1}\}$. Then:

\[
A_Q(q)=(1-q^{-1})^{b(Q)}\cdot\Hilb_{\Delta}\left(u_E=q^{-(b(Q)-b(Q\vert_E))}\right).
\] Moreover, $\Hilb_{\Delta}\left(u_{E}=q^{-(b(Q)-b(Q\vert_{E}))}\right)$
can be written as a rational function in $q$, whose numerator has
non-negative coefficients and whose denominator is a product of factors
of the form $q^{d}$ and $(q^{e}-1)$ for some integers $d,e\geq1$.

\end{thm}

\begin{proof}

The above formula for $A_{Q}$ can be rewritten as:\[
\begin{split}
A_Q & =
(1-q^{-1})^{b(Q)}\cdot\sum_{E_1\subsetneq E_2\ldots\subsetneq E_{s-1}\subsetneq E_s=Q_1}\prod_{j=1}^{s-1}\frac{1}{q^{b(Q)-b(Q\vert_{E_j})}-1} \\
& =
(1-q^{-1})^{b(Q)}\cdot
\sum_{E_1\subsetneq E_2\ldots\subsetneq E_{s-1}\subsetneq E_s=Q_1}\prod_{j=1}^{s-1}\frac{(q^{-1})^{b(Q)-b(Q\vert_{E_j})}}{1-(q^{-1})^{b(Q)-b(Q\vert_{E_j})}} \\
& = (1-q^{-1})^{b(Q)}\cdot\Hilb_{\Delta}\left(u_E=q^{-(b(Q)-b(Q\vert_E))}\right).
\end{split}
\] Note that, for all $E\subsetneq Q_{1}$, we have $b(Q)-b(Q\vert_{E})>0$,
by Proposition \ref{Prop/BettiNb}, so that the assignment $\deg(x_{E}):=b(Q)-b(Q\vert_{E})$
yields a suitable $\ZZ_{\geq0}$-grading of $\QQ[\Delta]$. Moreover,
the order complex of $\Pi(Q_{1})\setminus\{Q_{1}\}$ is shellable,
by Proposition \ref{Prop/LattShell} (see also Example \ref{Exmp/OrderCpxShell}),
so it is Cohen-Macaulay, by Proposition \ref{Prop/ShellCM}. Therefore,
up to substituting $q^{-1}$ for $q$ and multiplying the numerator
and denominator by an adequate power of $q$, the result follows directly
from Proposition \ref{Prop/CMHilb}. \end{proof}

\begin{rmk}

Wyss initially conjectured \cite[Conj. 4.32]{Wys17b} that $B_{\mu_{Q}}(q)$
is the quotient of a polynomial with non-negative coefficients by
an explicit denominator, defined in \cite[\S 4.6.]{Wys17b}. By Corollary
\ref{Cor/RelAvsB} and Theorem \ref{Thm/PositivityA}, we obtain a
presentation of $B_{\mu_{Q}}(q)$ as the quotient of a polynomial
with non-negative coefficients by a denominator which encodes information
about the Stanley-Reisner ring $\QQ[\Delta]$, for $\Delta:=\Delta(\Pi(Q_{1})\setminus\{Q_{1}\})$.
Comparing both denominators would require information on the degrees
of a homogeneous regular sequence for $\QQ[\Delta]$, which we do
not investigate here. Theorem \ref{Thm/PositivityA} provides an explanation
to Wyss' observation that the numerator of the rational fraction $B_{\mu_{Q}}(q)$
has non-negative coefficients, but the precise conjecture \cite[Conj. 4.32]{Wys17b}
remains open.

\end{rmk}

\subsection{Positivity of toric Kac polynomials \protect\label{Sect/PosToricKacPol}}

We now turn to the proof of Theorem \ref{Thm/IntroPosToricKacPol}.
Recall from \cite[Prop. 4.34.]{Wys17b} that $A_{Q,\mathbf{1},\alpha}$
enjoys the following explicit formula:\[
A_{Q,\mathbf{1},\alpha}=
\sum_{\substack{E_1\subseteq\ldots\subseteq E_{\alpha}\subseteq Q_1 \\ c(Q\vert_{E_{\alpha}})=1}}
(q-1)^{b(Q\vert_{E_{\alpha}})}q^{\sum_{k=1}^{\alpha-1}b(Q\vert_{E_{k}})}.
\] Set $\rr=\mathbf{1}$, $\KK=\FF_{q}$ and call $R(Q,\mathcal{O}_{\alpha}):=R(Q,\mathbf{1},\mathcal{O}_{\alpha})$
for short. For $a=\sum_{k}a_{k}t^{k}\in\mathcal{O}_{\alpha}$, we
define\[
\ord(a):=\min\{k\in\ZZ_{\geq0}\ \vert\ a_{k}\ne0\}
.
\] Let $R(Q,\mathcal{O}_{\alpha})_{\mathrm{ind.}}\subseteq R(Q,\mathcal{O}_{\alpha})$
be the constructible subset of (absolutely) indecomposable representations\footnote{By Proposition \ref{Prop/EndRings}, a locally free representation
of rank $\mathbf{1}$ is indecomposable if, and only if, it is absolutely
indecomposable.}. Suppose also that $Q$ is connected, so that $R(Q,\mathcal{O}_{\alpha})_{\mathrm{ind.}}\ne\emptyset$.
The above formula is obtained by counting $\left(\mathcal{O}_{\alpha}^{\times}\right)^{Q_{0}}$-orbits
(of $\FF_{q}$-points) along a stratification of $R(Q,\mathcal{O}_{\alpha})_{\mathrm{ind.}}$,
which is defined by prescribing the integers $\ord(x_{a}),\ a\in Q_{1}$.
The datum of $\ord(x_{a})$, $a\in Q_{1}$ is encoded in the subsets
$E_{1}\subseteq\ldots\subseteq E_{\alpha}\subseteq Q_{1}$ and there
are $(q-1)^{b(Q\vert_{E_{\alpha}})}q^{\sum_{k=1}^{\alpha-1}b(Q\vert_{E_{k}})}$
orbits in the stratum associated to $E_{1}\subseteq\ldots\subseteq E_{\alpha}\subseteq Q_{1}$.
As the formula shows, the polynomial counting orbits in a given stratum
does not necessarily have non-negative coefficients. To fix this,
we consider a coarser stratification of $R(Q,\mathcal{O}_{\alpha})_{\mathrm{ind.}}$
based on valued spanning trees and inspired from \cite{AMRV22}. Let
us describe the algorithm we use to assign a valued spanning tree
to $x\in R(Q,\mathcal{O}_{\alpha})_{\mathrm{ind.}}$. From now on,
we will only consider valued spanning trees whose valuation takes
values in $\{0,\ldots,\alpha-1\}$.

\paragraph*{Contraction-deletion algorithm}

Fix a total ordering of $Q_{1}$. Let $x\in R(Q,\mathcal{O}_{\alpha})_{\mathrm{ind.}}$.
We build a valued spanning tree of $Q$, called $T_{x}$, by following
the algorithm below:
\begin{enumerate}
\item If $Q_{1}$ contains at least one non-loop arrow $a$ such that $x_{a}\in\mathcal{O}_{\alpha}^{\times}$,
let $a_{0}\in Q_{1}$ be the largest such arrow and apply step 1 to
the induced representation $x'\in R(Q/\{a_{0}\},\mathcal{O}_{\alpha})$.
Otherwise, apply step 2 to $x$.
\item If $Q_{1}$ contains at least one loop, let $a_{0}\in Q_{1}$ be the
largest loop and apply step 2 to the induced representation $x'\in R(Q\setminus\{a_{0}\},\mathcal{O}_{\alpha})$.
Otherwise, apply step 3 to $x$.
\item If $Q_{1}\ne\emptyset$ and $\ord(x_{a})>0$ for all $a\in Q_{1}$,
apply step 1 to the representation $x'\in R(Q,\mathcal{O}_{\alpha-1})$
induced by the isomorphism of $\mathcal{O}_{\alpha}$-modules $\mathcal{O}_{\alpha-1}\simeq t\mathcal{O}_{\alpha}$,
$x\mapsto tx$ (so that $x_{a}=tx'_{a}$ for all $a\in Q_{1}$).
\end{enumerate}
Note that the algorithm terminates in step 3, when $Q$ is a one-vertex
quiver with no loops. Indeed, since $x\in R(Q,\mathcal{O}_{\alpha})_{\mathrm{ind.}}$,
the restriction of $Q$ to $\{a\in Q_{1}\ \vert\ x_{a}\ne0\}\subseteq Q_{1}$
is connected and the algorithm ends up contracting exactly the edges
of a spanning tree of $Q$, which we call $T_{x}$. Note that the
assumption $\ord(x_{a})>0$ in step 3 is necessarily satisfied for
all $a\in Q_{1}$, as non-loop arrows (resp. loops) such that $x_{a}\in\mathcal{O}_{\alpha}^{\times}$
are contracted (resp. deleted) in step 1 (resp. step 2). We define
the valued spanning tree associated to $x$ as $T_{x}$, with labeling
$v:a\mapsto\ord(x_{a})$.

\begin{exmp}

Let us illustrate the above algorithm on the following quiver:
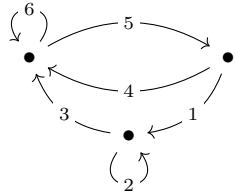
\[
\begin{tikzcd}[ampersand replacement=\&]
\bullet \arrow[loop, distance=2em, in=125, out=55, "6" description] \arrow[rr, bend left, "5" description] \& \& \bullet \arrow[ll, bend left, "4" description] \arrow[ld, bend left, "1" description] \\
 \& \bullet \arrow[lu, bend left, "3" description] \arrow[loop, distance=2em, in=305, out=235, "2" description] \&
\end{tikzcd}
\] and representation $x=(x_{1},x_{2},x_{3},x_{4},x_{5},x_{6})=(t,t^{2},t^{2},1,t,1)$.
Here, the order on $Q_{1}$ is the natural order on the integers.\[
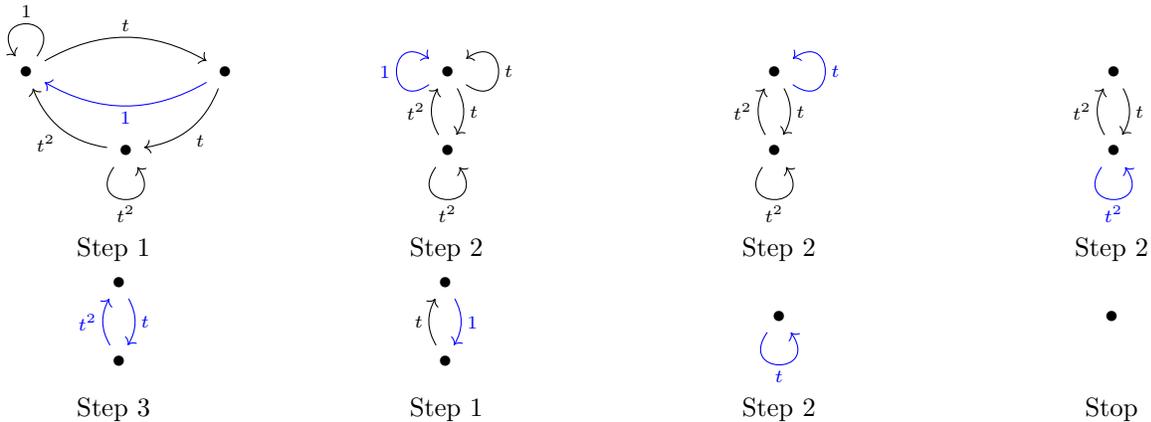
\begin{tabular}{>{\centering\arraybackslash}p{4cm} >{\centering\arraybackslash}p{4cm} >{\centering\arraybackslash}p{4cm} >{\centering\arraybackslash}p{4cm} >{\centering\arraybackslash}p{4cm}}
\begin{tikzcd}[ampersand replacement=\&]
\bullet \arrow[loop, distance=2em, in=125, out=55, "1", swap] \arrow[rr, bend left, "t"] \& \& \bullet \arrow[ll, bend left, "1", blue] \arrow[ld, bend left, "t"] \\
 \& \bullet \arrow[lu, bend left, "t^2"] \arrow[loop, distance=2em, in=305, out=235, "t^2", swap] \&
\end{tikzcd} &
\begin{tikzcd}[ampersand replacement=\&]
\bullet \arrow[loop, distance=2em, in=35, out=325, "t",  swap] \arrow[loop, distance=2em, in=145, out=215, "1", blue] \arrow[d, bend left, "t"] \\
\bullet \arrow[u, bend left, "t^2"] \arrow[loop, distance=2em, in=305, out=235, "t^2", swap]
\end{tikzcd} &
\begin{tikzcd}[ampersand replacement=\&]
\bullet \arrow[loop, distance=2em, in=35, out=325, "t",  swap, blue] \arrow[loop, distance=2em, in=145, out=215, phantom] \arrow[d, bend left, "t"] \\
\bullet \arrow[u, bend left, "t^2"] \arrow[loop, distance=2em, in=305, out=235, "t^2", swap]
\end{tikzcd} &
\begin{tikzcd}[ampersand replacement=\&]
\bullet \arrow[d, bend left, "t"]  \\
\bullet \arrow[u, bend left, "t^2"] \arrow[loop, distance=2em, in=305, out=235, "t^2", swap, blue]
\end{tikzcd} \\
\text{Step 1} & \text{Step 2} & \text{Step 2} & \text{Step 2} \\
\begin{tikzcd}[ampersand replacement=\&]
\bullet \arrow[d, bend left, "t", blue]  \\
\bullet \arrow[u, bend left, "t^2", blue]
\end{tikzcd} &
\begin{tikzcd}[ampersand replacement=\&]
\bullet \arrow[d, bend left, "1", blue]  \\
\bullet \arrow[u, bend left, "t"]
\end{tikzcd} &
\begin{tikzcd}[ampersand replacement=\&]
\bullet \arrow[loop, distance=2em, in=305, out=235, "t", swap, blue]
\end{tikzcd} &
\begin{tikzcd}[ampersand replacement=\&]
\bullet
\end{tikzcd} \\
\text{Step 3} & \text{Step 1} & \text{Step 2} & \text{End}
\end{tabular}
\]We thus obtain the following valued spanning tree $T_{x}$ (in blue):\[
\begin{tikzcd}[ampersand replacement=\&]
\bullet \arrow[loop, distance=2em, in=125, out=55, swap] \arrow[rr, bend left] \& \& \bullet \arrow[ll, bend left, "0", blue] \arrow[ld, bend left, blue, "1"] \\
 \& \bullet \arrow[lu, bend left] \arrow[loop, distance=2em, in=305, out=235, swap] \&
\end{tikzcd}
.
\]

\end{exmp}

Given a valued spanning tree $T$, we define $R(Q,\mathcal{O}_{\alpha})_{T}:=\{x\in R(Q,\mathcal{O}_{\alpha})_{\mathrm{ind.}}\ \vert\ T_{x}=T\}\subseteq R(Q,\mathcal{O}_{\alpha})_{\mathrm{ind.}}$.
Given that $T_{x}$ only depends on $\ord(x_{a}),\ a\in Q_{1}$, we
obtain that $R(Q,\mathcal{O}_{\alpha})_{T}$ is an $\left(\mathcal{O}_{\alpha}^{\times}\right)^{Q_{0}}$-invariant
constructible subset of $R(Q,\mathcal{O}_{\alpha})_{\mathrm{ind.}}$.
Let us describe the strata $R(Q,\mathcal{O}_{\alpha})_{T}$ in more
details.

Let $T$ be a valued spanning tree with valuation $v_{T}$, let $x\in R(Q,\mathcal{O}_{\alpha})_{T}$
and $a\in Q_{1}$ a non-loop edge. If $a\in T$, then $\ord(x_{a})=v_{T}(a)$
by construction. Suppose now that $a\notin T$. Then there exists
a unique (unoriented) path in $T$ joining the end vertices of $a$.
Let us denote by $T_{a}\subseteq Q_{1}$ the corresponding set of
edges and consider $\max v_{T}(T_{a})$ the maximum of $v_{T}$ on
$T_{a}$. When running the contraction-deletion algorithm, $a$ is
contracted into a loop exactly when the smallest edge of $T_{a}$
with valuation $\max v_{T}(T_{a})$ is contracted, i.e. when the last
edge of $T_{a}$ is contracted. Let us call $e_{T_{a}}$ this edge.
This means that $\ord(x_{a})\geq\max v_{T}(T_{a})$, otherwise $a$
would have been contracted before $e_{T_{a}}$. Moreover, if $a>e_{T_{a}}$,
then $\ord(x_{a})>\max v_{T}(T_{a})$, or again, $a$ would have been
contracted before $e_{T_{a}}$. One can check that these inequalities
imply $x\in R(Q,\mathcal{O}_{\alpha})_{T}$, so:\[
R(Q,\mathcal{O}_{\alpha})_{T}
=
\left\{
x\in R(Q,\mathcal{O}_{\alpha})_{\mathrm{ind.}}\
\left\vert 
\begin{array}{ll}
\ord(x_{a})=v_{T}(a), & a\in T \\
\ord(x_{a})\geq \max v_{T}(T_{a})+\delta_{a>e_{T_{a}}}, & a\not\in T
\end{array}
\right.
\right\}
.
\]When $a\in T$, we set $T_{a}:=\{a\}$, hence $\max v_{T}(T_{a})=v_{T}(a)$
and $e_{T_{a}}=a$.

Let us call $A_{Q,T,\alpha}$ the number of orbits (of $\FF_{q}$-points
in $R(Q,\mathcal{O}_{\alpha})_{\mathrm{ind.}}$) lying in $R(Q,\mathcal{O}_{\alpha})_{T}$.
Then summing over all valued spanning tree of $Q$, we obtain:\[
A_{Q,\mathbf{1},\alpha}=\sum_TA_{Q,T,\alpha}.
\] The advantage of considering coarser strata is that $A_{Q,T,\alpha}$
can be computed recursively as in \cite{AMRV22}, following the contraction-deletion
algorithm:

\begin{prop} \label{Prop/KacPolContDel}

Let $T$ be a valued spanning tree of $Q$. Denote by $Q',T'$ the
quiver and valued spanning tree obtained from $Q,T$ by running the
contraction-deletion algorithm for an arbitrary $x\in R(Q,\mathcal{O}_{\alpha})_{T}$
and stopping after the first occurrence of step 3. Consider: 
\begin{itemize}
\item $n_{1}$ the number of loops of $Q$; 
\item $n_{2}$ the number of non-loop arrows $a\in Q_{1}$ which get contracted
into loops during step 1 and satisfying $a<e_{T_{a}}$;
\item $n_{3}$ the number of non-loop arrows $a\in Q_{1}$ which get contracted
into loops during step 1 and satisfying $a>e_{T_{a}}$.
\end{itemize}
Then:\[
A_{Q,T,\alpha}=q^{\alpha n_1 + \alpha n_2 + (\alpha-1)n_3}\cdot A_{Q',T',\alpha-1}.
\]

\end{prop}

\begin{proof}

Throughout the proof, we will be considering orbits of $\FF_{q}$-rational
points. Let $a\in Q_{1}$ be the first arrow of $T$ which gets contracted
during step 1. Then the $\left(\mathcal{O}_{\alpha}^{\times}\right)^{Q_{0}}$-orbits
of $R(Q,\mathcal{O}_{\alpha})_{T}$ are in one-to-one correspondence
with $\left(\mathcal{O}_{\alpha}^{\times}\right)^{(Q/\{a\})_{0}}$-orbits
of $\{x\in R(Q,\mathcal{O}_{\alpha})_{T}\ \vert\ x_{a}=1\}$, where
the factor $\mathcal{O}_{\alpha}^{\times}\subseteq\left(\mathcal{O}_{\alpha}^{\times}\right)^{Q_{0}}$
corresponding to $[s(a)]=[t(a)]\in(Q/\{a\})_{0}$ acts diagonally
on the modules $\mathcal{O}_{\alpha}$ located at vertices $s(a)$
and $t(a)$.

Repeating this reasoning for all arrows $a_{1},\ldots,a_{s}$ of $T$
which get contracted (in step 1) during the construction of $Q'$,
we obtain that the $\left(\mathcal{O}_{\alpha}^{\times}\right)^{Q_{0}}$-orbits
of $R(Q,\mathcal{O}_{\alpha})_{T}$ are in one-to-one correspondence
with $\left(\mathcal{O}_{\alpha}^{\times}\right)^{Q'_{0}}$-orbits
of:\[
\{x\in R(Q,\mathcal{O}_{\alpha})_{T}\ \vert\ x_{a_1}=\ldots=x_{a_s}=1\}
\simeq
R(Q',\mathcal{O}_{\alpha})_{T'}\times\mathcal{O}_{\alpha}^{\oplus n_1}\times\mathcal{O}_{\alpha}^{\oplus n_2}\times\left(t\mathcal{O}_{\alpha}\right)^{\oplus n_3}
.
\] The above isomorphism is $\left(\mathcal{O}_{\alpha}^{\times}\right)^{Q'_{0}}$-equivariant,
with trivial action on the second and third factors of the right-hand
side. Here $T'$ is obtained from $T$ by contracting $a_{1},\ldots,a_{s}$
and leaving valuations unchanged. By abuse of notation, let us also
call $T'$ the valued spanning tree obtained by contracting $a_{1},\ldots,a_{s}$
and decreasing valuations by 1. Then we obtain as claimed:\[
A_{Q,T,\alpha}=q^{\alpha n_1 + \alpha n_2 + (\alpha-1)n_2}\cdot A_{Q',T',\alpha-1}.
\]\end{proof}

The contraction-deletion algorithm terminates in step 3, when $Q'$
is a one-vertex quiver without arrows. Then $A_{Q',\bullet,\alpha'}(q)=1$
and it follows from Proposition \ref{Prop/KacPolContDel} that\[
A_{Q,T,\alpha}(q)=q^{n_T}
\ ; \ 
n_T:=\sum_{a\not\in T}(\alpha-\max v_{T}(T_{a})-\delta_{a>e_{T_a}}).
\]This yields the following:

\begin{thm} \label{Thm/PositivityToricKacPol}

Let $\rr=\mathbf{1}$, then $A_{Q,\rr,\alpha}$ has non-negative coefficients.

\end{thm}

\begin{rmk}

In \cite{AMRV22}, toric Kac polynomials are shown to be specializations
of Tutte polynomials (see also \cite[Rmk. 7.11]{MR14}). This can
be deduced from a recursive relation satisfied by $A_{Q,\mathbf{1},1}$
under contraction-deletion of edges, for which Tutte polynomials are
universal. This relation relies on the fact that, for $x\in R(Q,\FF_{q})$
and $a\in Q_{1}$, either $x_{a}$ is invertible or $x_{a}=0$. This
is no longer the case for $\alpha>1$ and $A_{Q,\mathbf{1},\alpha}$
cannot be obtained as a specialization of the Tutte polynomial of
$Q$ (see also \cite[Prop. 7.16.]{HLRV24}). \end{rmk}

\subsection{Purity of toric preprojective stacks in higher depth \protect\label{Sect/CohPreprojStack}}

In this section, we upgrade Theorems \ref{Thm/CountMomMapGenFib}
and \ref{Thm/PositivityToricKacPol} to a cohomological result. Set
$\KK=\CC$. We show by direct computation that the compactly supported
cohomology of $\left[\mu_{Q,\rr,\alpha}^{-1}(t^{\alpha-1}\cdot\lambda)/\GL(\rr,\mathcal{O}_{\alpha})\right]$
($\rr=\mathbf{1}$) is pure and contains a pure Hodge structure with
E-polynomial $A_{Q,\rr,\alpha}$. This leads us to compute the compactly
supported cohomology of $\left[\mu_{Q,\rr,\alpha}^{-1}(0)/\GL(\rr,\mathcal{O}_{\alpha})\right]$,
which is also pure and satisfies a formula analogous to cohomological
integrality in \cite{Dav23c}.

We assume throughout that $Q$ is connected (the non-connected case
is analogous, by treating connected components separately). Set $\rr=\mathbf{1}$
and call $\mu_{Q,\rr,\alpha}:=\mu_{Q,\alpha}$ for short. We compute
$\HH_{\mathrm{c}}^{\bullet}\left(\left[\mu_{Q,\alpha}^{-1}(t^{\alpha-1}\cdot\lambda)/\GL(\rr,\mathcal{O}_{\alpha})\right]\right)$
using a $\GL(\rr,\mathcal{O}_{\alpha})$-equivariant stratification,
where strata are labeled by valued spanning trees. Let $\pi:\mu_{Q,\alpha}^{-1}(t^{\alpha-1}\cdot\lambda)\subseteq R(\overline{Q},\mathcal{O}_{\alpha})\rightarrow R(Q,\mathcal{O}_{\alpha})$
be the projection $(x,y)\mapsto x$. For $T$ a valued spanning tree
of $Q$, define the stratum $\mu_{Q,\alpha}^{-1}(t^{\alpha-1}\cdot\lambda)_{T}:=\pi^{-1}\left(R(Q,\mathcal{O}_{\alpha})_{T}\right)$.
We use the term ``stratification'' in a loose sense: the set of
valued spanning trees of $Q$ can be endowed with a partial order
such that, for any valued spanning tree $T$,\[
\bigsqcup_{T'\leq T}\mu_{Q,\alpha}^{-1}(t^{\alpha-1}\cdot\lambda)_{T'}\subseteq\mu_{Q,\alpha}^{-1}(t^{\alpha-1}\cdot\lambda)
\] is closed. One can then compute the (compactly supported) cohomology
of $\left[\mu_{Q,\rr,\alpha}^{-1}(t^{\alpha-1}\cdot\lambda)/\GL(\rr,\mathcal{O}_{\alpha})\right]$
from the cohomology of the strata using successive open-closed decompositions
and Lemma \ref{Lem/Strat}. The following proposition gives us the
partial order on strata that we are looking for:

\begin{prop} \label{Prop/Strata}

Let $T$ be a valued spanning tree of $Q$. Then the following closed
subset of $R(Q,\mathcal{O}_{\alpha})_{\mathrm{ind.}}$ is a union
of strata $R(Q,\mathcal{O}_{\alpha})_{T'}$:

\[
R(Q,\mathcal{O}_{\alpha})_{\leq T}
:=
\left\{
x\in R(Q,\mathcal{O}_{\alpha})_{\mathrm{ind.}}\
\left\vert 
\begin{array}{ll}
\ord(x_{a})\geq v_{T}(a), & a\in T \\
\ord(x_{a})\geq \max v_{T}(T_{a})+\delta_{a>e_{T_{a}}}, & a\not\in T
\end{array}
\right.
\right\}
.
\]

\end{prop}

\begin{proof}

Let $T'$ be a valued spanning tree of $Q$ and suppose that $R(Q,\mathcal{O}_{\alpha})_{\leq T}\cap R(Q,\mathcal{O}_{\alpha})_{T'}\ne\emptyset$.
We claim that for any $a\in Q_{1}$, $\max v_{T'}(T'_{a})+\delta_{a>e_{T'_{a}}}\geq\max v_{T}(T_{a})+\delta_{a>e_{T_{a}}}$,
hence $R(Q,\mathcal{O}_{\alpha})_{T'}\subseteq R(Q,\mathcal{O}_{\alpha})_{\leq T}$.
Note that when $a$ belongs to $T$, we have $T_{a}=\{a\}$ and $\max v_{T}(T_{a})=v_{T}(a)$
and $e_{T_{a}}=a$, hence the right-hand side simplifies to $v_{T}(a)$.
The same goes for the left-hand side when $a$ belongs to $T'$. We
write $a\in T$ for short when $a$ is an arrow of $T$.

Let us prove the claim. Take $x\in R(Q,\mathcal{O}_{\alpha})_{\leq T}\cap R(Q,\mathcal{O}_{\alpha})_{T'}$.
Since there is an unoriented path in $T'_{a}$ joining the end vertices
of $a$, the union of edges $\cup_{b\in T'_{a}}T_{b}$ joins the end
vertices of all edges in that path, hence also the endpoints of $a$.
This implies that $T_{a}\subseteq\cup_{b\in T'_{a}}T_{b}$, as $T_{a}$
is the set of edges of the unique (and shortest) unoriented path in
$T$ joining the end vertices of $a$. Thus, we obtain:\[
\max v_{T'}(T'_{a})=\max_{b\in T'_{a}}\{v_{T'}(b)\}=\max_{b\in T'_{a}}\{\ord(x_{b})\}\geq\max_{b\in T'_{a}}\{\max v_{T}(T_b)\}\geq \max v_{T}(T_{a}).
\]The second equality holds because $x\in R(Q,\mathcal{O}_{\alpha})_{T'}$,
whereas $\ord(x_{b})\geq\max v_{T}(T_{b})$ because $x\in R(Q,\mathcal{O}_{\alpha})_{\leq T}$.
We need to further show that, if $e_{T_{a}}<a\leq e_{T'_{a}}$, then
$\max v_{T'}(T'_{a})>\max v_{T}(T_{a})$. Suppose $e_{T_{a}}<a\leq e_{T'_{a}}$
and consider some $b\in T'_{a}$ such that $e_{T_{a}}\in T_{b}$.
Then:\[
\max v_{T'}(T'_{a})\geq v_{T'}(b)=\ord(x_b)\geq \max v_{T}(T_b)+\delta_{b>e_{T_b}}\geq v_T(e_{T_a})=\max v_{T}(T_{a}).
\]We argue that either the leftmost or the rightmost inequality is strict.
Indeed, if $\max v_{T'}(T'_{a})=v_{T'}(b)$ and $\max v_{T}(T_{b})=\max v_{T}(T_{a})$,
then $b\geq e_{T'_{a}}$ (as $v_{T'}(b)=v_{T'}(e_{T'_{a}})$) and
$e_{T_{a}}\geq e_{T_{b}}$ (as $e_{T_{a}}\in T_{b}$ and $v_{T}(e_{T_{a}})=v_{T}(e_{T_{b}})$).
Since we assumed that $e_{T_{a}}<a\leq e_{T'_{a}}$, we obtain $b>e_{T_{b}}$
and the rightmost inequality is strict. This concludes the proof.
\end{proof}

The partial order on valued spanning trees is then given by $T'\leq T$
if, and only if, $R(Q,\mathcal{O}_{\alpha})_{T'}\subseteq R(Q,\mathcal{O}_{\alpha})_{\leq T}$.
Let us now compute the compactly supported cohomology of $\left[\mu_{Q,\alpha}^{-1}(t^{\alpha-1}\cdot\lambda)_{T}/\GL(\rr,\mathcal{O}_{\alpha})\right]$.

\begin{prop} \label{Prop/CohStrat}

Let $T$ be a valued spanning tree of $Q$. Denote by $Q',T'$ the
quiver and valued spanning tree obtained from $Q,T$ by running the
contraction-deletion algorithm for an arbitrary $x\in R(Q,\mathcal{O}_{\alpha})_{T}$
and stopping after the first occurrence of step 3. Let $E\subset Q_{1}$
be the set of arrows of $T$ which are contracted during step 1 and
$\lambda':=\lambda/E$. Consider: 
\begin{itemize}
\item $n_{1}$ the number of loops of $Q$; 
\item $n_{2}$ the number of non-loop arrows $a\in Q_{1}$ which get contracted
into loops during step 1 and satisfying $a<e_{T_{a}}$;
\item $n_{3}$ the number of non-loop arrows $a\in Q_{1}$ which get contracted
into loops during step 1 and satisfying $a>e_{T_{a}}$.
\end{itemize}
Then:\[
\HH_{\mathrm{c}}^{\bullet}\left(\left[\mu_{Q,\alpha}^{-1}(t^{\alpha-1}\cdot\lambda)_T/\left(\mathcal{O}_{\alpha}^{\times}\right)^{Q_0}\right]\right)
\simeq
\HH_{\mathrm{c}}^{\bullet}\left(\left[\mu_{Q',\alpha-1}^{-1}(t^{\alpha-2}\cdot\lambda')_{T'}/\left(\mathcal{O}_{\alpha-1}^{\times}\right)^{Q'_0}\right]\right)
\otimes\mathbb{L}^{\otimes\left(2\alpha n_1+2\alpha n_2+(2\alpha-1)n_3+\# Q'_1-\# Q'_0\right)}.
\]

\end{prop}

\begin{proof}

We claim that there is an $\left(\mathcal{O}_{\alpha}^{\times}\right)^{Q_{0}}$-equivariant
isomorphism:\[
\Psi:
\mu_{Q,\alpha}^{-1}(t^{\alpha-1}\cdot\lambda)_T
\simeq
\left(
\mu_{Q',\alpha-1}^{-1}(t^{\alpha-2}\cdot\lambda')_{T'}
\times^{\left(\mathcal{O}_{\alpha}^{\times}\right)^{Q'_0}}\left(\mathcal{O}_{\alpha}^{\times}\right)^{Q_0}
\right)
\times\mathbb{A}^{2\alpha n_1}\times\mathbb{A}^{2\alpha n_2+(2\alpha-1)n_3}\times\mathbb{A}^{\# Q'_1}.
\]Here, we consider the morphism of algebraic groups $\left(\mathcal{O}_{\alpha}^{\times}\right)^{Q'_{0}}\rightarrow\left(\mathcal{O}_{\alpha}^{\times}\right)^{Q_{0}}$
given by the formula $(g_{i'})_{i'\in Q'_{0}}\mapsto(g_{[i]})_{i\in Q_{0}}$,
where $[i]$ denotes the class of $i\in Q_{0}$ in $Q'_{0}=Q_{0}/\sim_{E}$.
Let us call:
\begin{itemize}
\item $a_{1},\ldots,a_{s}\in Q_{1}$ the arrows of $T$ which get contracted
(in step 1) during the construction of $Q'$;
\item $a'_{1},\ldots,a'_{n_{1}}\in Q_{1}$ the loops of $Q$;
\item $a''_{1},\ldots,a''_{n_{2}}\in Q_{1}$ the non-loop arrows of $Q$
which get contracted into loops (in step 1) and deleted (in step 2)
during the construction of $Q'$ and which satisfy $a<e_{T_{a}}$;
\item $b''_{1},\ldots,b''_{n_{3}}\in Q_{1}$ the non-loop arrows of $Q$
which get contracted into loops (in step 1) and deleted (in step 2)
during the construction of $Q'$ and which satisfy $a>e_{T_{a}}$.
\end{itemize}
Let $(x,y)\in\mu_{Q,\alpha}^{-1}(t^{\alpha-1}\cdot\lambda)_{T}$.
One can check that $(x,y)$ determines a point $(x',y')\in\mu_{Q',\alpha}^{-1}(t^{\alpha-1}\cdot\lambda')_{T'}$
such that, for all $a\in Q'_{1}$, $x'_{a}\in t\mathcal{O}_{\alpha}$.
We describe the components of $\Psi(x,y)$:
\begin{itemize}
\item Note that $\mu_{Q',\alpha-1}^{-1}(t^{\alpha-2}\cdot\lambda')_{T'}\times^{\left(\mathcal{O}_{\alpha}^{\times}\right)^{Q'_{0}}}\left(\mathcal{O}_{\alpha}^{\times}\right)^{Q_{0}}\simeq\mu_{Q',\alpha-1}^{-1}(t^{\alpha-2}\cdot\lambda')_{T'}\times\left(\mathcal{O}_{\alpha}^{\times}\right)^{s}$.
The component of $\Psi(x,y)$ along $\mu_{Q',\alpha-1}^{-1}(t^{\alpha-2}\cdot\lambda')_{T'}$
is induced by $(x',y')$ via the morphisms of $\mathcal{O}_{\alpha}$-modules
$t\mathcal{O}_{\alpha}\simeq\mathcal{O}_{\alpha-1}$ (for the $x$-coordinate)
and $\mathcal{O}_{\alpha}\twoheadrightarrow\mathcal{O}_{\alpha-1}$
(for the $y$-coordinate). The component along $\left(\mathcal{O}_{\alpha}^{\times}\right)^{s}$
is $(x_{a_{t}})_{1\leq t\leq s}$;
\item $\Psi(x,y)_{\mathbb{A}^{2\alpha n_{1}}}=(x_{a'_{n}},y_{a'_{n}})_{1\leq n\leq n_{1}}$
- note that the moment map equation imposes no conditions on $(x_{a'_{n}},y_{a'_{n}})$,
as $\rr=\underline{1}$;
\item $\Psi(x,y)_{\mathbb{A}^{2\alpha n_{2}+(2\alpha-1)n_{3}}}=\left((x_{a''_{n}},y_{a''_{n}})_{1\leq n\leq n_{2}},(x_{b''_{n}},y_{b''_{n}})_{1\leq n\leq n_{3}}\right)$
- note that $x_{a''_{n}}\in t\mathcal{O}_{\alpha}$ by assumption;
moreover the coordinates $(x_{a''_{n}},y_{a''_{n}})_{1\leq n\leq n_{2}},(x_{b''_{n}},y_{b''_{n}})_{1\leq n\leq n_{3}}$
may be chosen freely and determine $x_{a_{1}}y_{a_{1}},\ldots,x_{a_{s}}y_{a_{s}}$
through the moment map equation;
\item Write $y'_{a}=\sum_{k}y'_{a,k}\cdot t^{k}\in\mathcal{O}_{\alpha}$
for $a\in Q'_{1}$; then $\Psi(x,y)_{\mathbb{A}^{\sharp Q'_{1}}}=(y'_{a,\alpha-1})_{a\in Q'_{1}}$.
\end{itemize}
Consequently, $\Psi(x,y)$ contains all the coordinates of $(x,y)$
except for $y_{a_{1}},\ldots,y_{a_{s}}$. These are determined from
$\Psi(x,y)$ using the moment map equation. The action of $\left(\mathcal{O}_{\alpha}^{\times}\right)^{Q_{0}}$
on $\mu_{Q',\alpha-1}^{-1}(t^{\alpha-2}\cdot\lambda')_{T'}\times\left(\mathcal{O}_{\alpha}^{\times}\right)^{s}$
is induced by a choice of splitting $\left(\mathcal{O}_{\alpha}^{\times}\right)^{Q_{0}}\simeq\left(\mathcal{O}_{\alpha}^{\times}\right)^{Q'_{0}}\times\left(\mathcal{O}_{\alpha}^{\times}\right)^{s}$
and makes the morphism $\mu_{Q,\alpha}^{-1}(t^{\alpha-1}\cdot\lambda)_{T}\rightarrow\mu_{Q',\alpha-1}^{-1}(t^{\alpha-2}\cdot\lambda')_{T'}\times^{\left(\mathcal{O}_{\alpha}^{\times}\right)^{Q'_{0}}}\left(\mathcal{O}_{\alpha}^{\times}\right)^{Q_{0}}$
$\left(\mathcal{O}_{\alpha}^{\times}\right)^{Q_{0}}$-equivariant.
The $\left(\mathcal{O}_{\alpha}^{\times}\right)^{Q_{0}}$-action on
the remaining components of the right-hand side is transferred from
the action on the left-hand side.

Using the isomorphism $\Psi$, Lemmas \ref{Lem/AffFib} and \ref{Lem/GrpChg}
yield:\[
\HH_{\mathrm{c}}^{\bullet}\left(\left[\mu_{Q,\alpha}^{-1}(t^{\alpha-1}\cdot\lambda)_T/\left(\mathcal{O}_{\alpha}^{\times}\right)^{Q_0}\right]\right)
\simeq
\HH_{\mathrm{c}}^{\bullet}\left(\left[\mu_{Q',\alpha-1}^{-1}(t^{\alpha-2}\cdot\lambda')_{T'}/\left(\mathcal{O}_{\alpha}^{\times}\right)^{Q'_0}\right]\right)
\otimes\mathbb{L}^{\otimes\left(2\alpha n_1+2\alpha n_2+(2\alpha-1)n_3+\# Q'_1\right)}.
\] Since the action of $\left(\mathcal{O}_{\alpha}^{\times}\right)^{Q'_{0}}$
on $\mu_{Q',\alpha-1}^{-1}(t^{\alpha-2}\cdot\lambda')_{T'}$ factors
through the quotient group $\left(\mathcal{O}_{\alpha-1}^{\times}\right)^{Q'_{0}}$,
Lemma \ref{Lem/DepthChg} gives the desired formula. \end{proof}

When the contraction-deletion algorithm terminates, the computation
finishes with $\HH_{\mathrm{c}}^{\bullet}\left(\left[\mathrm{pt}/\mathcal{O}_{\alpha'}^{\times}\right]\right)\simeq\mathbb{L}^{\otimes-\alpha'}\otimes\HH_{\mathrm{c}}^{\bullet}\left(\mathrm{B}\mathbb{G}_{\mathrm{m}}\right)$
, which is pure. This shows that $\HH_{\mathrm{c}}^{\bullet}\left(\left[\mu_{Q,\alpha}^{-1}(t^{\alpha-1}\cdot\lambda)_{T}/\left(\mathcal{O}_{\alpha}^{\times}\right)^{Q_{0}}\right]\right)$
is pure, of Tate type, which allows us to conclude using Lemma \ref{Lem/Strat},
Proposition \ref{Prop/E-seriesVSCountF_q} and Theorem \ref{Thm/CountMomMapGenFib}:

\begin{thm} \label{Thm/CohMomMapGenFib}

Let $\rr=\mathbf{1}$. Suppose that $\lambda$ is generic with respect
to $\rr$. Then there is a non-canonical isomorphism of $\ZZ$-graded
polarizable mixed Hodge structures:\[
\HH_{\mathrm{c}}^{\bullet}\left(\left[\mu_{Q,\rr,\alpha}^{-1}(t^{\alpha-1}\cdot\lambda)/\GL(\rr,\mathcal{O}_{\alpha})\right]\right)
\simeq
A_{Q,\rr,\alpha}(\mathbb{L})\otimes\mathbb{L}^{1-\alpha\langle\rr,\rr\rangle_Q}\otimes\HH_{\mathrm{c}}^{\bullet}\left(\mathrm{B}\mathbb{G}_{\mathrm{m}}\right)
,
\]where $A_{Q,\rr,\alpha}(\mathbb{L})$ and $\HH_{\mathrm{c}}^{\bullet}\left(\mathrm{B}\mathbb{G}_{\mathrm{m}}\right)$
are as discussed in Section \ref{Sect/EqCoh}. In particular, $\HH_{\mathrm{c}}^{\bullet}\left(\left[\mu_{Q,\rr,\alpha}^{-1}(t^{\alpha-1}\cdot\lambda)/\GL(\rr,\mathcal{O}_{\alpha})\right]\right)$
carries a pure Hodge structure.

\end{thm}

There is more: as can be seen from the proof of Proposition \ref{Prop/Strata},
the parameter $\lambda$ plays no role in computing the cohomology
of the strata. If we replace $\lambda$ with 0, we may compute $\HH_{\mathrm{c}}^{\bullet}\left(\left[\mu_{Q,\rr,\alpha}^{-1}(0)/\GL(\rr,\mathcal{O}_{\alpha})\right]\right)$
by pulling back a stratification of $R(Q,\mathcal{O}_{\alpha})$ instead
of $R(Q,\mathcal{O}_{\alpha})_{\mathrm{ind.}}$. For a general $x\in R(Q,\mathcal{O}_{\alpha})$
the restriction of $Q$ to $\{a\in Q_{1}\ \vert\ x_{a}\ne0\}$ may
not be connected, so we should index strata of $R(Q,\mathcal{O}_{\alpha})$
by (i) partitions $Q_{0}=I_{1}\sqcup\ldots\sqcup I_{s}$ and (ii)
valued spanning trees for $Q\vert_{I_{1}},\ldots,Q\vert_{I_{s}}$.
Therefore, given such a collection $(I=\{I_{1},\ldots,I_{s}\},T=\{T_{1},\ldots,T_{s}\})$\footnote{that is, $Q_{0}=\bigsqcup_{1\leq t\leq s}I_{t}$ and, for all $1\leq t\leq s$,
the valued tree $T_{t}$ spans $Q\vert_{I_{t}}$.}, we define:\[
R(Q,\mathcal{O}_{\alpha})_{(I,T)}:=
\left\{
x\in R(Q,\mathcal{O}_{\alpha})\
\left\vert
\begin{array}{l}
\forall 1\leq t\leq s,\ (x_a)_{a\in Q_{1,I_s}}\in R(Q\vert_{I_s},\mathcal{O}_{\alpha})_{T_s} \\
\forall a\not\in\bigcup_{t=1}^sQ_{1,I_t},\ x_a=0
\end{array}
\right.
\right\}
.
\]In other words, $x\in R(Q,\mathcal{O}_{\alpha})_{(I,T)}$ if, and
only if, the indecomposable summands of $M(x)$ have rank vectors
$\mathbf{1}\vert_{I_{1}},\ldots,\mathbf{1}\vert_{I_{s}}$ and their
associated valued spanning trees are respectively $T_{1},\ldots,T_{s}$.

Similarly to our previous stratification of $R(Q,\mathcal{O}_{\alpha})_{\mathrm{ind.}}$,
we define a partial order on the strata $R(Q,\mathcal{O}_{\alpha})_{(I,T)}$
using the following Proposition:

\begin{prop} \label{Prop/Strata2}

Let $I=\{I_{1},\ldots,I_{s}\}$ be a partition of $Q_{0}$ and $T=\{T_{1},\ldots,T_{s}\}$
a collection of valued spanning trees for the subquivers $Q\vert_{I_{t}},\ 1\leq t\leq s$.
The following closed subset of $R(Q,\mathcal{O}_{\alpha})$ is a union
of strata $R(Q,\mathcal{O}_{\alpha})_{(I',T')}$:

\[
R(Q,\mathcal{O}_{\alpha})_{\leq (I,T)}
:=
\left\{
x\in R(Q,\mathcal{O}_{\alpha})\
\left\vert 
\begin{array}{l}
\forall 1\leq t\leq s,\ (x_a)_{a\in Q_{1,I_s}}\in R(Q\vert_{I_s},\mathcal{O}_{\alpha})_{\leq T_s} \\
\forall a\not\in\bigcup_{t=1}^sQ_{1,I_t},\ x_a=0
\end{array}
\right.
\right\}
.
\]

\end{prop}

\begin{proof}

Consider a partition $I'=\{I'_{1},\ldots,I'_{s'}\}$ of $Q_{0}$ and
a collection $T'=\{T'_{1},\ldots,T'_{s'}\}$ of valued spanning trees,
where $T'_{t'}$ spans $Q\vert_{I'_{t'}}$ for all $1\leq t'\leq s'$.
Suppose that $R(Q,\mathcal{O}_{\alpha})_{(I',T')}\cap R(Q,\mathcal{O}_{\alpha})_{\leq(I,T)}\ne\emptyset$.
It is enough to show that:
\begin{enumerate}
\item The partition $I'$ refines the partition $I$, i.e. there is a surjective
function $\varphi:\{1,\ldots,s'\}\rightarrow\{1,\ldots,s\}$ such
that, for all $1\leq t\leq s$, we have $I_{t}=\bigsqcup_{t'\in\varphi^{-1}(t)}I'_{t'}$.
\item For all $t'\in\{1,\ldots,s'\}$ and $t=\varphi(t')$, for all $a\in Q_{1,I'_{t'}}$,
we have:\[
\max v_{T'_{t'}}((T'_{t'})_a)+\delta_{a>e_{(T'_{t'})_a}}
\geq
\max v_{T_t}((T_t)_a)+
\delta_{a>e_{(T_t)_a}}
.
\]
\end{enumerate}
Let us prove 1. Consider $x\in R(Q,\mathcal{O}_{\alpha})_{(I',T')}\cap R(Q,\mathcal{O}_{\alpha})_{\leq(I,T)}$.
Then $M(x)$ decomposes into indecomposable summands:\[
M(x)\simeq\bigoplus_{1\leq t'\leq s'}M_{t'}
\]such that $\mathbf{rk}(M_{t'})=\mathbf{1}\vert_{I'_{t'}}$ for all
$1\leq t'\leq s'$. Since $x\in R(Q,\mathcal{O}_{\alpha})_{\leq(I,T)}$,
each subrepresentation $M_{t'}$ must split into a direct sum of representations
supported respectively on $I_{t},\ 1\leq t\leq s$. Since $M_{t'}$
is indecomposable, this decomposition is trivial, hence $I'_{t'}\subset I_{t}$
for a unique $t:=\varphi(t')\in\{1,\ldots,s\}$. This proves 1.

We now prove 2. Fix $1\leq t'\leq s'$ and $a\in Q_{1,I'_{t'}}$.
Consider also $x\in R(Q,\mathcal{O}_{\alpha})_{(I',T')}\cap R(Q,\mathcal{O}_{\alpha})_{\leq(I,T)}$.
For simplicity, we will write $I':=I'_{t'}$ (resp. $I:=I_{\varphi(t')}$)
and $T':=T'_{t'}$ (resp. $T:=T_{\varphi(t')}$). As in the proof
of Proposition \ref{Prop/Strata}, we have $T_{a}\subset\cup_{b\in T'_{a}}T_{b}$,
hence the inequality:\[
\max v_{T'}(T'_{a})=\max_{b\in T'_{a}}\{v_{T'}(b)\}=\max_{b\in T'_{a}}\{\ord(x_{b})\}\geq\max_{b\in T'_{a}}\{\max v_{T}(T_b)\}\geq \max v_{T}(T_{a}).
\]Let us further choose $b\in T'_{a}$ such that $e_{T_{a}}\in T_{b}$.
Then, as in the proof of Proposition \ref{Prop/Strata}, we have the
chain of inequalities:\[
\max v_{T'}(T'_{a})\geq v_{T'}(b)=\ord(x_b)\geq \max v_{T}(T_b)+\delta_{b>e_{T_b}}\geq v_T(e_{T_a})=\max v_{T}(T_{a})
,
\]where the leftmost or the rightmost inequality is strict if $e_{T_{a}}<a\leq e_{T'_{a}}$.
This proves that $\max v_{T'_{t'}}((T'_{t'})_{a})+\delta_{a>e_{(T'_{t'})_{a}}}\geq\max v_{T_{t}}((T_{t})_{a})+\delta_{a>e_{(T_{t})_{a}}}$,
as required.\end{proof}

As in the case of $R(Q,\mathcal{O}_{\alpha})_{\mathrm{ind.}}$, we
thus obtain a partial order on collections $(I,T)$, given by: $(I',T')\leq(I,T)$
if, and only if, $R(Q,\mathcal{O}_{\alpha})_{(I',T')}\subset R(Q,\mathcal{O}_{\alpha})_{\leq(I,T)}$.
By construction, \[
R(Q,\mathcal{O}_{\alpha})_{\leq (I,T)}=\bigsqcup_{(I',T')\leq (I,T)}R(Q,\mathcal{O}_{\alpha})_{(I',T')}
\]is closed. We can then prove our final cohomological result.

\begin{thm} \label{Thm/CohIntgr}

Let $\rr=\mathbf{1}$. Then there is a non-canonical isomorphism of
$\ZZ$-graded polarizable mixed Hodge structures:

\[
\HH_{\mathrm{c}}^{\bullet}\left(\left[\mu_{Q,\rr,\alpha}^{-1}(0)/\GL(\rr,\mathcal{O}_{\alpha})\right]\right)
\otimes\mathbb{L}^{\otimes\alpha\langle\rr,\rr\rangle_Q}
\simeq
\bigoplus_{Q_0=I_1\sqcup\ldots\sqcup I_s}
\bigotimes_{j=1}^s
\left(
A_{Q\vert_{I_j},\rr\vert_{I_j},\alpha}(\mathbb{L})\otimes\mathbb{L}\otimes \HH_{\mathrm{c}}^{\bullet}(\mathrm{B}\mathbb{G}_m)
\right).
\]

In particular, $\HH_{\mathrm{c}}^{\bullet}\left(\left[\mu_{Q,\rr,\alpha}^{-1}(0)/\GL(\rr,\mathcal{O}_{\alpha})\right]\right)$
carries a pure Hodge structure.

\end{thm}

\begin{proof}

We claim that, for any partition $I=\{I_{1},\ldots,I_{s}\}$ of $Q_{0}$
and for any collection $T=\{T_{1},\ldots,T_{s}\}$ of valued spanning
trees of the subquivers $Q\vert_{I_{t}},\ 1\leq t\leq s$, the $\ZZ$-graded
mixed Hodge structure $\HH_{\mathrm{c}}^{\bullet}\left(\left[\mu_{Q,\alpha}^{-1}(0)_{(I,T)}/\left(\mathcal{O}_{\alpha}^{\times}\right)^{Q_{0}}\right]\right)$
is pure, of Tate type. Then it follows from Lemma \ref{Lem/Strat}
that $\HH_{\mathrm{c}}^{\bullet}\left(\left[\mu_{Q,\alpha}^{-1}(0)/\left(\mathcal{O}_{\alpha}^{\times}\right)^{Q_{0}}\right]\right)$
is pure, of Tate type. Therefore, $\HH_{\mathrm{c}}^{\bullet}\left(\left[\mu_{Q,\alpha}^{-1}(0)/\left(\mathcal{O}_{\alpha}^{\times}\right)^{Q_{0}}\right]\right)$
can be computed from its E-series, as discussed below Theorem \ref{Prop/E-seriesVSCountF_q}.
The required formula then follows from Theorem \ref{Prop/E-seriesVSCountF_q}
and Theorem \ref{Thm/ExpFmlKacPol}.

Let us prove the claim. We first note that the forgetful map\[
\left[\mu_{Q,\alpha}^{-1}(0)_{(I,T)}/\left(\mathcal{O}_{\alpha}^{\times}\right)^{Q_{0}}\right]
\rightarrow
\prod_{t=1}^s\left[\mu_{Q\vert_{I_t},\alpha}^{-1}(0)_{T_t}/\left(\mathcal{O}_{\alpha}^{\times}\right)^{I_t}\right]
\]is an affine fibration. Indeed, since for any $a\in Q_{1}\setminus\bigcup_{t=1}^{s}Q_{1,I_{t}}$
and $(x,y)\in\mu_{Q,\alpha}^{-1}(0)_{(I,T)}$, we have $x_{a}=0$,
the coordinate $y_{a}$ is unconstrained by the equations $\mu_{Q,\alpha}(x,y)=0$.
Then by Lemmas \ref{Lem/AffFib} and \ref{Lem/Kunneth}, we obtain:\[
\HH_{\mathrm{c}}^{\bullet}\left(\left[\mu_{Q,\alpha}^{-1}(0)_{(I,T)}/\left(\mathcal{O}_{\alpha}^{\times}\right)^{Q_{0}}\right]\right)
\simeq
\mathbb{L}^{\otimes d}\otimes
\bigotimes_{t=1}^s
\HH_{\mathrm{c}}^{\bullet}\left(\left[\mu_{Q\vert_{I_t},\alpha}^{-1}(0)_{T_t}/\left(\mathcal{O}_{\alpha}^{\times}\right)^{I_t}\right]\right)
,
\]for a certain integer $d\geq0$. Moreover, we know from the proof
of Theorem \ref{Thm/CohMomMapGenFib} that, for all $1\leq t\leq s$,
$\HH_{\mathrm{c}}^{\bullet}\left(\left[\mu_{Q\vert_{I_{t}},\alpha}^{-1}(0)_{T_{t}}/\left(\mathcal{O}_{\alpha}^{\times}\right)^{I_{t}}\right]\right)$
is pure, of Tate type. Therefore, $\HH_{\mathrm{c}}^{\bullet}\left(\left[\mu_{Q,\alpha}^{-1}(0)_{(I,T)}/\left(\mathcal{O}_{\alpha}^{\times}\right)^{Q_{0}}\right]\right)$
is also pure, of Tate type, which finishes the proof. \end{proof}

\begin{rmk}

Theorem \ref{Thm/CohIntgr} may be interpreted as a higher depth analog
of a PBW theorem for preprojective cohomological Hall algebras \cite{Dav23c},
restricted to $\rr\leq\mathbf{1}$. Let us call $\mathrm{Sym}$ the
functor between categories of $\ZZ\times\ZZ_{\geq0}^{Q_{0}}$-graded
polarizable mixed Hodge structures which recovers the plethystic exponential
upon taking E-series - see \cite[\S 3.2]{DM20}. Then Theorem \ref{Thm/CohIntgr}
can be stated as follows:\[
\bigoplus_{\rr\leq\mathbf{1}}\HH_{\mathrm{c}}^{\bullet}\left(\left[\mu_{Q,\rr,\alpha}^{-1}(0)/\GL(\rr,\mathcal{O}_{\alpha})\right]\right)\otimes\mathbb{L}^{\otimes\alpha\langle\rr,\rr\rangle_Q}
\simeq
\mathrm{Sym}
\left.
\left(
\bigoplus_{\rr>0}A_{Q,\rr,\alpha}(\mathbb{L})\otimes\mathbb{L}\otimes \HH_{\mathrm{c}}^{\bullet}(\mathrm{B}\mathbb{G}_m)
\right)
\right\vert_{\rr\leq\mathbf{1}}.
\]In other words, Theorem \ref{Thm/CohIntgr} is the cohomological upgrade
(in the category of $\ZZ\times\ZZ_{\geq0}^{Q_{0}}$-graded polarizable
mixed Hodge structures) of Theorem \ref{Thm/ExpFmlKacPol}, restricted
to $\rr\leq\mathbf{1}$.

\end{rmk}

\begin{rmk} \label{Rmk/CohUpgrade=000026Positivity}

Note that, while in \cite{Dav18,Dav23c}, positivity of Kac polynomials
is deduced from a cohomological PBW isomorphism, here the proof of
Theorem \ref{Thm/CohIntgr} \textit{relies} on positivity for $A_{Q,\rr,\alpha}$
- in order to make sense of the graded pure Hodge structure $A_{Q,\rr,\alpha}(\mathbb{L})$.

In \cite{Dav23c}, the existence of graded mixed Hodge structures
$\mathrm{BPS}_{Q,\dd}^{\vee},\ \dd\in\ZZ_{\geq0}^{Q_{0}}$ satisfying\[
\bigoplus_{\dd\geq0}
\HH_{\mathrm{c}}^{\bullet}\left(\left[\mu_{Q,\dd}^{-1}(0)/\GL(\dd)\right]\right)\otimes\mathbb{L}^{\otimes\langle\dd,\dd\rangle_Q}
\simeq
\mathrm{Sym}
\left(
\bigoplus_{\dd\geq0}\mathrm{BPS}_{Q,\dd}^{\vee}\otimes\mathbb{L}\otimes \HH_{\mathrm{c}}^{\bullet}(\mathrm{B}\mathbb{G}_m)
\right)
\] is proved using cohomological Donaldson-Thomas theory. Purity of
$\HH_{\mathrm{c}}^{\bullet}\left(\left[\mu_{Q,\dd}^{-1}(0)/\GL(\dd)\right]\right)$
then implies purity of $\mathrm{BPS}_{Q,\dd}^{\vee}$, which in turn
shows that $A_{Q,\dd}$ has non-negative coefficients.

In our setting however, the isomorphism of Theorem \ref{Thm/CohIntgr}
is deduced from the fact that both sides are pure, of Tate type and
have the same E-polynomial by Theorem \ref{Thm/ExpFmlKacPol}. The
well-definedness of $A_{Q,\rr,\alpha}(\mathbb{L})$ relies on Theorem
\ref{Thm/PositivityToricKacPol}, since we do not know by other means
that there exists a graded pure Hodge structure analogous to $\mathrm{BPS}_{Q,\dd}^{\vee}$.

\end{rmk}

\section{Beyond the toric setting \protect\label{Sect/HigherRk}}

In this last section, we collect some computations of $A_{Q,\rr,\alpha}$
for larger rank vectors. We follow the approach of Kac, Stanley and
later Hua and use Burnside's lemma to compute $A_{Q,\rr,\alpha}$
\cite{Kac83,Hua00a}. The computation turns out to be rather involved,
already for low ranks and simple quivers, as a good knowledge of conjugacy
classes of $\GL(\rr,\mathcal{O}_{\alpha})$ is required. We exploit
results of Avni, Onn, Prasad, Vaserstein and Avni, Klopsch, Onn, Voll
\cite{AOPV09,AKOV16} to obtain closed formulas for $g$-loop quivers
in ranks 2 and 3. This involves solving a linear recurrence relation
over $\alpha\geq1$. Throughout this section, we assume that $\KK=\FF_{q}$.

Let us first apply Burnside's lemma. Consider $M_{Q,\rr,\alpha}\in\mathcal{V}$
the count of all isomorphism classes of locally free quiver representations
over $\mathcal{O}_{\alpha}$, in rank $\rr$. This is the orbit-count
for the action $\GL(\rr,\mathcal{O}_{\alpha})\circlearrowleft R(Q,\rr,\mathcal{O}_{\alpha})$
and so we obtain:\[
M_{Q,\rr,\alpha}=\sum_{[\gamma]\in\mathrm{Cl}(\rr,\mathcal{O}_{\alpha})}\frac{\# R(Q,\rr,\mathcal{O}_{\alpha})^{\gamma}}{\# Z_{\GL(\rr,\mathcal{O}_{\alpha})}(\gamma)},
\] where the sum runs over the set $\mathrm{Cl}(\rr,\mathcal{O}_{\alpha})$
of conjugacy classes of $\GL(\rr,\mathcal{O}_{\alpha})$, $R(Q,\rr,\mathcal{O}_{\alpha})^{\gamma}$
is the set of points which are fixed by $\gamma$ and $Z_{\GL(\rr,\mathcal{O}_{\alpha})}(\gamma)$
is the centralizer of $\gamma$ in $\GL(\rr,\mathcal{O}_{\alpha})$.
Then we can recover $A_{Q,\rr,\alpha}$ from $M_{Q,\rr,\alpha}$ using
the following formula \cite[Thm. 1.2.]{Moz19}:\[
\sum_{\rr\in\NN^{Q_0}}M_{Q,\rr,\alpha}\cdot t^{\rr}=\Exp_{q,t}\left(\sum_{\rr\in\NN^{Q_0}\setminus\{0\}}A_{Q,\rr,\alpha}\cdot t^{\rr}\right).
\]

Conjugacy classes of $\GL(\rr,\mathcal{O}_{\alpha})$ are completely
classified in ranks 2 and 3 \cite{AOPV09,AKOV16}. However, as noted
in \cite[\S 4]{AOPV09}, computing the space of intertwiners between
two distinct conjugacy classes proves untractable. Therefore, we reduce
to the case of a quiver $Q$ with one vertex and $g$ loops. One could
then compute the centralizers in $\GL(\rr,\mathcal{O}_{\alpha})$
and $\mathfrak{gl}(\rr,\mathcal{O}_{\alpha})$ for representatives
of all conjugacy classes in $\GL(\rr,\mathcal{O}_{\alpha})$. However,
we prefer to compute summands in Burnside's formula by induction on
$\alpha$, using the branching rules for conjugacy classes described
in \cite{AOPV09,AKOV16}. Let us give details on this when $r=2$.

Recall from \cite[\S 2]{AOPV09} that matrices in $\mathfrak{gl}(2,\mathcal{O}_{\alpha})$
are either scalar (we call these of type I) or conjugate to a unique
matrix of the form \[
d\cdot\mathrm{I}_2+t^j\cdot
\left(
\begin{array}{cc}
0 & a_0 \\
1 & a_1
\end{array}
\right),
\] where $0\leq j\leq\alpha-1$, $d\in\bigoplus_{l=0}^{j-1}\FF_{q}\cdot t^{l}$
and $a_{0},a_{1}\in\mathcal{O}_{\alpha-j}$ (we call these of type
II). We split the set of conjugacy classes in $\GL(2,\mathcal{O}_{\alpha})$
in four disjoint subsets $\mathrm{Cl}(2,\mathcal{O}_{\alpha})=\mathrm{Cl}(2,\mathcal{O}_{\alpha})_{\mathrm{I}}\sqcup\mathrm{Cl}(2,\mathcal{O}_{\alpha})_{\mathrm{II_{1}}}\sqcup\mathrm{Cl}(2,\mathcal{O}_{\alpha})_{\mathrm{II_{2}}}\sqcup\mathrm{Cl}(2,\mathcal{O}_{\alpha})_{\mathrm{II_{3}}}$
and compute the sums\[
S_{\bullet,\alpha}:=\sum_{[\gamma]\in\mathrm{Cl}(2,\mathcal{O}_{\alpha})_{\bullet}}\frac{\# R(Q,2,\mathcal{O}_{\alpha})^{\gamma}}{\# Z_{\GL(2,\mathcal{O}_{\alpha})}(\gamma)}
\] recursively, for $\bullet\in\{\mathrm{I},\mathrm{II}_{1},\mathrm{II}_{2},\mathrm{II}_{3}\}$.
The subsets $\mathrm{Cl}(2,\mathcal{O}_{\alpha})_{\bullet}$ are defined
as follows:
\begin{itemize}
\item $\mathrm{Cl}(2,\mathcal{O}_{\alpha})_{\mathrm{I}}$ consists of conjugacy
classes of scalar matrices;
\item $\mathrm{Cl}(2,\mathcal{O}_{\alpha})_{\mathrm{II}_{1}}$ consists
of conjugacy classes of type II, where $T^{2}-a_{1}\cdot T-a_{0}$
is split in $\FF_{q}[T]$ and has two distinct simple roots;
\item $\mathrm{Cl}(2,\mathcal{O}_{\alpha})_{\mathrm{II}_{2}}$ consists
of conjugacy classes of type II, where $T^{2}-a_{1}\cdot T-a_{0}$
is split in $\FF_{q}[T]$ and has a double root;
\item $\mathrm{Cl}(2,\mathcal{O}_{\alpha})_{\mathrm{II}_{3}}$ consists
of conjugacy classes of type II, where $T^{2}-a_{1}\cdot T-a_{0}$
is irreducible in $\FF_{q}[T]$.
\end{itemize}
Conjugacy classes belonging to the same subset $\mathrm{Cl}(2,\mathcal{O}_{\alpha})_{\bullet}$
have similar centralizers in $\GL(2,\mathcal{O}_{\alpha})$ and $\mathfrak{gl}(2,\mathcal{O}_{\alpha})$.
Let $\GL^{1}(2,\mathcal{O}_{\alpha})\subseteq\GL(2,\mathcal{O}_{\alpha})$
be the kernel of reduction modulo $t$, i.e. $\GL^{1}(2,\mathcal{O}_{\alpha})=\mathrm{I}_{2}+t\cdot\mathfrak{gl}(2,\mathcal{O}_{\alpha})$.
Following \cite[\S 2.1.]{AKOV16}, for $\sigma\in\{\mathrm{I},\mathrm{II}_{1},\mathrm{II}_{2},\mathrm{II}_{3}\}$
and $\gamma\in\GL(2,\mathcal{O}_{\alpha})$ of type $\sigma$, we
call $\Vert\sigma\Vert$ the cardinality of $Z_{\GL(2,\mathcal{O}_{\alpha})}(\gamma)/\left(Z_{\GL(2,\mathcal{O}_{\alpha})}(\gamma)\cap\GL^{1}(2,\mathcal{O}_{\alpha})\right)$
and $\dim(\sigma)$ its dimension as an algebraic group. These do
not depend on the choice of $\gamma$.

Consider $\gamma\in\GL(2,\mathcal{O}_{\alpha+1})$, of type $\sigma$,
and its reduction $\overline{\gamma}\in\GL(2,\mathcal{O}_{\alpha})$,
of type $\tau$. Then, arguing as in the proof of \cite[Prop. 2.5.]{AKOV16},
we easily obtain:\[
\frac{\# Z_{\GL(2,\mathcal{O}_{\alpha+1})}(\gamma)}{\# Z_{\GL(2,\mathcal{O}_{\alpha})}(\overline{\gamma})}
=
q^{\dim(\tau)}\cdot\frac{\Vert\sigma\Vert}{\Vert\tau\Vert},
\]
\[
\frac{\# Z_{\mathfrak{gl}(2,\mathcal{O}_{\alpha+1})}(\gamma)}{\# Z_{\mathfrak{gl}(2,\mathcal{O}_{\alpha})}(\overline{\gamma})}
=
q^{\dim(\sigma)}.
\] Moreover, one can deduce the following branching rules from the description
of conjugacy classes in $\GL(2,\mathcal{O}_{\alpha})$. Let $a_{\tau,\sigma}(q)$
be the number of conjugacy classes of type $\sigma$ in $\GL(2,\mathcal{O}_{\alpha+1})$
whose reduction in $\GL(2,\mathcal{O}_{\alpha})$ is of type $\tau$.
Then: \[
\left(a_{\tau,\sigma}(q)\right)_{\substack{\sigma\in\{\mathrm{I},\mathrm{II}_{1},\mathrm{II}_{2},\mathrm{II}_{3}\} \\ \tau\in\{\mathrm{I},\mathrm{II}_{1},\mathrm{II}_{2},\mathrm{II}_{3}\}}}
=
\left(
\begin{array}{cccc}
q & 0 & 0 & 0 \\
\frac{q(q-1)}{2} & q^2 & 0 & 0 \\
q & 0 & q^2 & 0 \\
\frac{q(q-1)}{2} & 0 & 0 & q^2
\end{array}
\right)
,
\] where rows are labeled by $\sigma$ and columns are labeled by $\tau$.
Putting everything together, we obtain:\[
\left(
\begin{array}{c}
S_{\mathrm{I},\alpha+1} \\
S_{\mathrm{II}_1,\alpha+1} \\
S_{\mathrm{II}_2,\alpha+1} \\
S_{\mathrm{II}_3,\alpha+1}
\end{array}
\right)
=
\left(
\begin{array}{cccc}
q^{4g-3} & 0 & 0 & 0 \\
\frac{1}{2}q^{2g-2}(q-1)(q+1) & q^{2g} & 0 & 0 \\
q^{2g-3}(q-1)(q+1) & 0 & q^{2g} & 0 \\
\frac{1}{2}q^{2g-2}(q-1)^2 & 0 & 0 & q^{2g}
\end{array}
\right)
\times
\left(
\begin{array}{c}
S_{\mathrm{I},\alpha} \\
S_{\mathrm{II}_1,\alpha} \\
S_{\mathrm{II}_2,\alpha} \\
S_{\mathrm{II}_3,\alpha}
\end{array}
\right)
\ ;\ 
\left(
\begin{array}{c}
S_{\mathrm{I},1} \\
S_{\mathrm{II}_1,1} \\
S_{\mathrm{II}_2,1} \\
S_{\mathrm{II}_3,1}
\end{array}
\right)
=
\left(
\begin{array}{c}
\frac{q^{4g}}{q(q-1)(q+1)} \\
\frac{q^2(q-2)}{2(q-1)} \\
q^{2g-1} \\
\frac{q^{2g+1}}{2(q+1)}
\end{array}
\right)
,
\] where the $(\sigma,\tau)$-entry of the transition matrix is $q^{g\dim(\sigma)-\dim(\tau)}\cdot\frac{\Vert\tau\Vert}{\Vert\sigma\Vert}\cdot a_{\tau,\sigma}(q)$.
This matrix can easily be diagonalized and with the help of a computer,
we finally get:\[
A_{Q,2,\alpha}= \frac{q^{2\alpha g-1}(q^{2g}-1)(q^{\alpha(2g-3)}-1)}{(q^2-1)(q^{2g-3}-1)}.
\]

In rank $r=3$, we split $\mathrm{Cl}(3,\mathcal{O}_{\alpha})$ into
ten types $\mathcal{G},\mathcal{L},\mathcal{J},\mathcal{T}_{1},\mathcal{T}_{2},\mathcal{T}_{3},\mathcal{M},\mathcal{N},\mathcal{K}_{0},\mathcal{K}_{\infty}$,
following \cite[\S 2.2.]{AKOV16}, and we obtain the following recursive
relation:\[
\begin{array}{l}
\left(
\begin{array}{c}
S_{\mathcal{G},\alpha+1} \\
S_{\mathcal{L},\alpha+1} \\
S_{\mathcal{J},\alpha+1} \\
S_{\mathcal{T}_{1},\alpha+1} \\
S_{\mathcal{T}_{2},\alpha+1} \\
S_{\mathcal{T}_{3},\alpha+1} \\
S_{\mathcal{M},\alpha+1} \\
S_{\mathcal{N},\alpha+1} \\
S_{\mathcal{K}_0,\alpha+1} \\
S_{\mathcal{K}_{\infty},\alpha+1}
\end{array}
\right)
=
M
\times
\left(
\begin{array}{c}
S_{\mathcal{G},\alpha} \\
S_{\mathcal{L},\alpha} \\
S_{\mathcal{J},\alpha} \\
S_{\mathcal{T}_{1},\alpha} \\
S_{\mathcal{T}_{2},\alpha} \\
S_{\mathcal{T}_{3},\alpha} \\
S_{\mathcal{M},\alpha} \\
S_{\mathcal{N},\alpha} \\
S_{\mathcal{K}_0,\alpha} \\
S_{\mathcal{K}_{\infty},\alpha}
\end{array}
\right)
\ ; \ 
\left(
\begin{array}{c}
S_{\mathcal{G},1} \\
S_{\mathcal{L},1} \\
S_{\mathcal{J},1} \\
S_{\mathcal{T}_{1},1} \\
S_{\mathcal{T}_{2},1} \\
S_{\mathcal{T}_{3},1} \\
S_{\mathcal{M},1} \\
S_{\mathcal{N},1} \\
S_{\mathcal{K}_0,1} \\
S_{\mathcal{K}_{\infty},1}
\end{array}
\right)
=
\left(
\begin{array}{c}
\frac{q^{9g-3}}{(q^2-1)(q^3-1)} \\
\frac{q^{5g-1}(q-2)}{(q-1)(q^2-1)} \\
\frac{q^{5g-3}}{q-1} \\
\frac{q^{3g}(q-2)(q-3)}{6(q-1)^2} \\
\frac{q^{3g+1}}{2(q+1)} \\
\frac{q^{3g+1}(q^2-1)}{3(q^3-1)} \\
\frac{q^{3g-1}(q-2)}{q-1} \\
q^{3g-2} \\
0 \\
0
\end{array}
\right)
\text{, where:}
\\ \\
M=
\left(
\begin{array}{cccccccccc}
q^{9g-8} & 0 & 0 & 0 & 0 & 0 & 0 & 0 & 0 & 0 \\
q^{5g-6}(q^3-1) & q^{5g-3} & 0 & 0 & 0 & 0 & 0 & 0 & 0 & 0 \\
\frac{q^{5g-8}(q^2-1)(q^3-1)}{q-1} & 0 & q^{5g-3} & 0 & 0 & 0 & 0 & 0 & 0 & 0 \\
\frac{q^{3g-5}(q-2)(q^2-1)(q^3-1)}{6(q-1)} & \frac{q^{3g-2}(q^2-1)}{2} & 0 & q^{3g} & 0 & 0 & 0 & 0 & 0 & 0 \\
\frac{q^{3g-4}(q-1)(q^3-1)}{2} & \frac{q^{3g-2}(q-1)^2}{2} & 0 & 0 & q^{3g} & 0 & 0 & 0 & 0 & 0 \\
\frac{q^{3g-5}(q-1)(q^2-1)^2}{3} & 0 & 0 & 0 & 0 & q^{3g} & 0 & 0 & 0 & 0 \\
q^{3g-6}(q^2-1)(q^3-1) & q^{3g-3}(q^2-1) & q^{3g-1}(q-1) & 0 & 0 & 0 & q^{3g} & 0 & 0 & 0 \\
q^{3g-7}(q^2-1)(q^3-1) & 0 & q^{3g-3}(q-1)^2 & 0 & 0 & 0 & 0 & q^{3g} & 0 & 0 \\
0 & 0 & q^{3g-3}(q-1) & 0 & 0 & 0 & 0 & 0 & q^{3g} & 0 \\
0 & 0 & 0 & 0 & 0 & 0 & 0 & 0 & 0 & q^{3g}
\end{array}
\right)
.
\end{array}
\]

Again, the transition matrix $M$ can be diagonalized and we get the
following closed formula.

\begin{prop}

Let $Q$ be the quiver with one vertex and $g\geq1$ loops. Then:\[
\begin{split}
A_{Q,2,\alpha}= & \frac{q^{2\alpha g-1}(q^{2g}-1)(q^{\alpha(2g-3)}-1)}{(q^2-1)(q^{2g-3}-1)}, \\
A_{Q,3,\alpha}= &
\frac{q^{3\alpha g-2}(q^{2g}-1)(q^{2g-1}-1)}{(q^2-1)(q^3-1)(q^{2g-3}-1)(q^{6g-8}-1)(q^{4g-5}-1)} \\
 & \cdot
\left(
q^{\alpha(6g-8)-1}(q^{6g-7}-1)(q^{2g}+1)
-q^{\alpha(6g-8)+2g-4}(q^2-1)(q^{4g-3}+1) \right. \\
 & \left. +q^{\alpha(2g-3)-1}(q^2+q+1)(q^{2g-1}-1)(q^{6g-8}-1)+(q+1)(q^{8g-10}-1)+q^{2g-4}(q^4+1)(q^{4g-5}-1)
\right)
.
\end{split}
\]In particular, $A_{Q,2,\alpha}$ and $A_{Q,3,\alpha}$ are polynomials
and $A_{Q,2,\alpha}$ has non-negative coefficients.

\end{prop}

Unfortunately, it is not obvious from the above expression that $A_{Q,3,\alpha}$
is a polynomial in $q$ with non-negative coefficients. But we can
still compute $A_{Q,3,\alpha}$ for small values of $g$ and $\alpha$
and check that this is indeed the case. We write $A_{g,3,\alpha}:=A_{Q,3,\alpha}$
when $Q$ is the $g$-loop quiver. Using a computer, we get:

\[
\begin{array}{l}
g=1: \\
\begin{split}
A_{1,3,1} & = q \\
A_{1,3,2} & = q^4 + q^3 + 2  q^2 \\
A_{1,3,3} & = q^7 + q^6 + 3  q^5 + 2  q^4 + 2  q^3 \\
A_{1,3,4} & = q^{10} + q^9 + 3  q^8 + 3  q^7 + 4  q^6 + 2  q^5 + 2  q^4 \\
A_{1,3,5} & = q^{13} + q^{12} + 3  q^{11} + 3  q^{10} + 5  q^9 + 4  q^8 + 4  q^7 + 2  q^6 + 2  q^5
\end{split} \\
\\ 
g=2: \\
\begin{split}
A_{2,3,1} & = q^{10} + q^8 + q^7 + q^6 + q^5 + q^4 \\
A_{2,3,2} & = q^{20} + q^{18} + 2q^{17} + 3q^{16} + 3q^{15} + 4q^{14} + 3q^{13} + 3q^{12} + 2q^{11} + 2q^{10} \\
A_{2,3,3} & = q^{30} + q^{28} + 2q^{27} + 3q^{26} + 3q^{25} + 5q^{24} + 5q^{23} + 7q^{22} + 6q^{21} + 7q^{20} + 5q^{19} + 4q^{18} + 3q^{17} + 2q^{16} \\
A_{2,3,4} & = q^{40} + q^{38} + 2q^{37} + 3q^{36} + 3q^{35} + 5q^{34} + 5q^{33} + 7q^{32} + 7q^{31} \\
& + 9q^{30} + 9q^{29} + 10q^{28} + 9q^{27} + 9q^{26} + 6q^{25} + 5q^{24} + 3q^{23} + 2q^{22} \\
A_{2,3,5} & = q^{50} + q^{48} + 2q^{47} + 3q^{46} + 3q^{45} + 5q^{44} + 5q^{43} + 7q^{42} + 7q^{41} + 9q^{40} + 9q^{39} + 11q^{38} \\
& + 11q^{37} + 13q^{36} + 12q^{35} + 13q^{34} + 11q^{33} + 10q^{32} + 7q^{31} + 5q^{30} + 3q^{29} + 2q^{28}
\end{split} \\
\\
g=3: \\
\begin{split}
A_{3,3,1} & = q^{19} + q^{17} + q^{16} + q^{15} + q^{14} + 2q^{13} + q^{12} + 2q^{11} + 2q^{10} + q^{9} + q^{8} + q^{7} \\
A_{3,3,2} & = q^{38} + q^{36} + q^{35} + q^{34} + q^{33} + 2q^{32} + 2q^{31} + 3q^{30} + 4q^{29} + 4q^{28} + 4q^{27} \\
& + 5q^{26} + 4q^{25} + 4q^{24} + 4q^{23} + 5q^{22} + 3q^{21} + 4q^{20} + 3q^{19} + 2q^{18} + q^{17} + q^{16} \\
A_{3,3,3} & = q^{57} + q^{55} + q^{54} + q^{53} + q^{52} + 2q^{51} + 2q^{50} + 3q^{49} + 4q^{48} + 4q^{47} + 4q^{46} + 5q^{45} + 4q^{44} + 5q^{43} + 5q^{42} + 7q^{41}  + 6q^{40} + \\
& 8q^{39} + 8q^{38} + 8q^{37} + 7q^{36} + 8q^{35} + 7q^{34} + 6q^{33} + 6q^{32} + 6q^{31} + 4q^{30} + 4q^{29} + 3q^{28} + 2q^{27} + q^{26} + q^{25} \\
A_{3,3,4} & = q^{76} + q^{74} + q^{73} + q^{72} + q^{71} + 2q^{70} + 2q^{69} + 3q^{68} + 4q^{67} + 4q^{66} + 4q^{65} + 5q^{64} + 4q^{63} + 5q^{62} + 5q^{61}  + 7q^{60} \\
& + 6q^{59} + 8q^{58} + 8q^{57} + 8q^{56}  + 8q^{55} + 9q^{54} + 9q^{53} + 9q^{52} + 10q^{51} + 11q^{50} + 10q^{49} + 11q^{48} + 11q^{47} + 11q^{46} \\
& + 9q^{45} + 10q^{44} + 8q^{43} + 7q^{42} + 6q^{41} + 6q^{40} + 4q^{39} + 4q^{38} + 3q^{37} + 2q^{36} + q^{35} + q^{34} \\
A_{3,3,5} & = q^{95} + q^{93} + q^{92} + q^{91} + q^{90} + 2q^{89} + 2q^{88} + 3q^{87} + 4q^{86} + 4q^{85} + 4q^{84} + 5q^{83} + 4q^{82} + 5q^{81}  + 5q^{80} + 7q^{79} \\
& + 6q^{78} + 8q^{77} + 8q^{76} + 8q^{75} + 8q^{74} + 9q^{73} + 9q^{72} + 9q^{71} + 10q^{70} + 11q^{69} + 10q^{68}  + 12q^{67} + 12q^{66} + 13q^{65} + 12q^{64} \\
& + 14q^{63} + 13q^{62} + 13q^{61} + 13q^{60} + 14q^{59} + 13q^{58} + 13q^{57} + 13q^{56} + 12q^{55} + 10q^{54} + 10q^{53} + 8q^{52} + 7q^{51} + 6q^{50} \\
& + 6q^{49} + 4q^{48} + 4q^{47} + 3q^{46} + 2q^{45} + q^{44} + q^{43}
\end{split}
\end{array}
\] 

From the data above, it seems reasonable to expect that $A_{Q,\rr,\alpha}$
has non-negative coefficients for any quiver and any rank vector.
We hope to address this conjecture in future works:

\begin{cj}

Let $Q$ be a quiver, $\rr\in\ZZ_{\geq0}^{Q_{0}}$ and $\alpha\geq1$.
Then $A_{Q,\rr,\alpha}\in\ZZ_{\geq0}[q]$ .

\end{cj}

\printbibliography

\end{document}